\documentclass[11pt, reqno]{amsart}
\usepackage{amsfonts,latexsym,enumerate}
\usepackage{amsmath}
\usepackage{amscd}
\usepackage{float,amsmath,amssymb,mathrsfs,bm,multirow,graphics}
\usepackage[dvips]{graphicx}
\usepackage[percent]{overpic}
\usepackage[pdftex]{color}
\usepackage{amsaddr}
\usepackage[numbers,sort&compress]{natbib}

\renewcommand{\theequation}{\mbox{\arabic{section}.\arabic{equation}}}
\newcommand{\R}{{\Bbb R}}

\newcommand{\C}{{\Bbb C}}

\newcommand{\tr}{\text{\upshape tr\,}}
\newcommand{\res}{\text{\upshape Res\,}}

\newcommand{\re}{\text{\upshape Re\,}}

\newcommand{\im}{\text{\upshape Im\,}}

\newcommand{\U}{\mathsf{U}}
\newcommand{\V}{\mathsf{V}}
\newcommand{\E}{\mathcal{E}}

\DeclareMathOperator{\Int}{int}

\def\XXint#1#2#3{{\setbox0=\hbox{$#1{#2#3}{\int}$}
\vcenter{\hbox{$#2#3$}}\kern-.5\wd0}}

\newtheorem{theorem}{Theorem}
\newtheorem{proposition}{Proposition}[section]

\newtheorem{lemma}[proposition]{Lemma}
\newtheorem{definition}[proposition]{Definition}

\newtheorem{remark}[proposition]{Remark}

\newtheorem{figuretext}{Figure}

\numberwithin{equation}{section}

\input epsf
\date{\today}
\title[The hyperbolic Ernst equation in a triangular domain]
{The hyperbolic Ernst equation \\ in a triangular domain}

\author{Jonatan Lenells and Julian Mauersberger}
\address{Department of Mathematics, KTH Royal Institute of Technology, \\ 100 44 Stockholm, Sweden.}
\email{jlenells@kth.se, julianma@kth.se}

\begin{document}
\begin{abstract} 
The collision of two plane gravitational waves in Einstein's theory of relativity can be described mathematically by a Goursat problem for the hyperbolic Ernst equation in a triangular domain. We use the integrable structure of the Ernst equation to present the solution of this problem via the solution of a Riemann--Hilbert problem. The formulation of the Riemann--Hilbert problem involves only the prescribed boundary data, thus the solution is as effective as the solution of a pure initial value problem via the inverse scattering transform. Our results are valid also for boundary data whose derivatives are unbounded at the triangle's corners---this level of generality is crucial for the application to colliding gravitational waves. Remarkably, for data with a singular behavior of the form relevant for gravitational waves, it turns out that the singular integral operator underlying the Riemann--Hilbert formalism can be explicitly inverted at the boundary. 
In this way, we are able to show exactly how the behavior of the given data at the origin transfers into a singular behavior of the solution near the boundary. 
\end{abstract}

\maketitle



\setcounter{tocdepth}{1}
\tableofcontents

\section{Introduction}
Half a century after Einstein presented his theory of relativity, F. J. Ernst made the remarkable discovery that, in the presence of one space-like and one time-like Killing vector, the entire solution of the vacuum Einstein field equations reduces to solving a single equation for a complex-valued function $\mathcal{E}$ of two variables \cite{E1968}. This single equation, now known as the (elliptic) Ernst equation, has proved instrumental in the study and construction of stationary axisymmetric spacetimes, cf. \cite{KR2005}. 

It later became clear that a similar reduction of Einstein's equations is possible also in the presence of two space-like Killing vectors, a situation relevant for the description of two colliding plane gravitational waves \cite{CF1984}. In this case the associated equation is known as the hyperbolic Ernst equation and can be written in the form
\begin{align}\label{ernst}  
  (\re \mathcal{E}) \left(\mathcal{E}_{xy} - \frac{\mathcal{E}_x + \mathcal{E}_y}{2(1-x-y)}\right) = \mathcal{E}_x \mathcal{E}_y, 
\end{align}  
where the Ernst potential $\mathcal{E}(x,y)$ is a complex-valued function of the two real variables $(x,y)$ and subscripts denote partial derivatives.

The problem of finding the nonlinear interaction of two plane gravitational waves following their collision has a distinguished history going back to the work of Khan and Penrose \cite{KP1971}, Szekeres \cite{S1972}, Nutku and Halil \cite{NH1977}, and Chandrasekhar and coauthors \cite{CF1984, CX1987}; see the monograph \cite{G1991} for further references and historical remarks.
In terms of the Ernst potential, this collision problem reduces to a Goursat problem for equation (\ref{ernst}) in the triangular region $D$ defined by (see Figure \ref{D.pdf})
\begin{align}\label{Ddef}
D = \{(x,y) \in \R^2\, | \, x \geq 0, \; y \geq 0, \; x+y < 1\}.
\end{align}
More precisely, the problem can be formulated as follows (see \cite{G1991} and the appendix): 
\begin{align}\label{Goursat}
\begin{cases}
\text{Given complex-valued functions $\mathcal{E}_0(x)$, $x \in [0, 1)$, and $\mathcal{E}_1(y)$, $y \in [0,1)$, } 
	\\
\text{find a solution $\mathcal{E}(x,y)$ of the hyperbolic Ernst equation (\ref{ernst}) in $D$}
	\\
\text{such that $\mathcal{E}(x,0) = \mathcal{E}_0(x)$ for $x \in [0,1)$ and $\mathcal{E}(0,y) = \mathcal{E}_1(y)$ for $y \in [0,1)$.}
\end{cases}
\end{align}
 
In this paper, we use the integrable structure of equation (\ref{ernst}) and Riemann--Hilbert (RH) techniques to analyze the Goursat problem (\ref{Goursat}). We present four main results, denoted by Theorem \ref{mainth1}-\ref{mainth4}:

\begin{itemize}
\item Theorem \ref{mainth1} is a solution representation result: Assuming that the given data satisfy the  following conditions for some $n \geq 2$:
\begin{align}\label{E0E1assumptions}
\begin{cases}
\mathcal{E}_0, \mathcal{E}_1 \in C([0,1)) \cap C^n((0,1)), 
\\
\text{$x^\alpha \mathcal{E}_{0x}, y^\alpha \mathcal{E}_{1y} \in C([0,1))$ for some $\alpha \in [0,1)$},
\\
\mathcal{E}_0(0) = \mathcal{E}_1(0) = 1,
\\
\text{$\re  \mathcal{E}_0(x) > 0$ for $x \in [0,1)$}, 
\\
\text{$\re  \mathcal{E}_1(y) > 0$ for $y \in [0,1)$},
\end{cases}
\end{align}
and the Goursat problem \eqref{Goursat} has a solution (in the precise sense specified in Definition \ref{solutiondef}), we give a representation formula for this solution. This formula is given in terms of the solution of a corresponding RH problem whose formulation only involves the given boundary data.

\item Theorem \ref{mainth2} is a uniqueness result: Assuming that the given data satisfy the conditions (\ref{E0E1assumptions}) for some $n \geq 2$, we show that the solution of the Goursat problem (\ref{Goursat}) is unique, if it exists.

\item Theorem \ref{mainth3} is an existence and regularity result: Assuming that the given data satisfy the conditions (\ref{E0E1assumptions}) for some $n \geq 2$, we show that there exists a unique solution $\mathcal{E}$ of the problem (\ref{Goursat}) whenever the associated RH problem has a solution, and this $\mathcal{E}$ has the same regularity as the given data. In the case of collinearly polarized waves, this yields existence for general data; for noncollinearly polarized waves, a small-norm assumption is also needed.

\item Theorem \ref{mainth4} provides exact formulas for the singular behavior of the solution $\mathcal{E}$ near the boundary for data satisfying \eqref{E0E1assumptions}. 
\end{itemize}

\begin{figure}
\bigskip
\begin{center}
 \begin{overpic}[width=.4\textwidth]{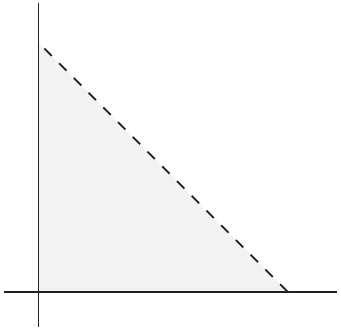}
 \put(29,33){ $D$}
 \put(103,9.7){\small $x$}
 \put(9.5,101){\small $y$}
 \put(83.5,5){\small $1$}
 \put(6,82.5){\small $1$}
 \put(27,3.5){\small $\mathcal{E}(x,0) = \mathcal{E}_0(x)$}
 \put(2,26){\small \rotatebox{90}{$\mathcal{E}(0,y) = \mathcal{E}_1(y)$}}
 %\put(45,55){\small \rotatebox{-45}{$x+y=1$}}
 \end{overpic}
   \begin{figuretext}\label{D.pdf}
      The triangular region $D$ defined in (\ref{Ddef}) and the boundary conditions relevant for the Goursat problem (\ref{Goursat}). 
      \end{figuretext}
   \end{center}
\end{figure}

We emphasize that the assumptions (\ref{E0E1assumptions}) allow for functions $\mathcal{E}_0(x)$ and $\mathcal{E}_1(y)$ whose derivatives blow up as $x$ and $y$ approach the origin. This level of generality is necessary for the application to gravitational waves. Indeed, in order for the problem (\ref{Goursat}) to be relevant in the context of gravitational waves, it turns out that the solution should obey the conditions (see \cite{G1991} and the appendix)
\begin{subequations}\label{boundaryconditions}
\begin{align}
& \lim_{x \downarrow 0} x^\alpha |\mathcal{E}_x(x,y)| = \frac{m_1\re \mathcal{E}_1(y)}{\sqrt{1-y}} && \text{for each $y \in [0,1)$},
	\\
& \lim_{y \downarrow 0} y^\alpha |\mathcal{E}_y(x,y)| = \frac{m_2\re \mathcal{E}_0(x)}{\sqrt{1-x}} && \text{for each $x \in [0,1)$},
\end{align}
\end{subequations}
where $m_1$ and $m_2$ are real constants such that $m_1, m_2 \in [1, \sqrt{2})$ and $\alpha = 1/2$. 
Remarkably, for data with a singular behavior at the origin of the form given in (\ref{E0E1assumptions}), the singular integral operator underlying the RH formalism can be explicitly inverted in the limit of small $x$  or $y$. This leads to the characterization of the boundary behavior given in Theorem \ref{mainth4}. In particular, it implies the following important conclusion for the collision of gravitational waves: A solution $\mathcal{E}(x,y)$ of the Goursat problem for (\ref{ernst}) fulfills (\ref{boundaryconditions}) iff the boundary data are such that 
$\lim_{x \downarrow 0}  x^\alpha |\mathcal{E}_{0x}(x)|$ and $\lim_{y \downarrow 0}  y^\alpha |\mathcal{E}_{1y}(y)|$ lie in the interval $[1, \sqrt{2})$.

The assumptions $\re  \mathcal{E}_0(x) > 0$ and $\re  \mathcal{E}_1(y) > 0$ in (\ref{E0E1assumptions})  are natural because in the context of gravitational waves the real part of the Ernst potential is automatically strictly positive. The assumption $\mathcal{E}_0(0) = \mathcal{E}_1(0)$ in (\ref{E0E1assumptions}) expresses the compatibility of the boundary values at the origin. If $\mathcal{E}$ is a solution of (\ref{ernst}), then so is $a\mathcal{E} + ib$ for any choice of the real constants $a$ and $b$. Thus, since $\mathcal{E}(0,0) \neq 0$ as a consequence of the assumption $\re  \mathcal{E}_0(x) > 0$, there is no loss of generality in assuming that $\mathcal{E}(0,0) = 1$.

The analysis of a boundary or initial-boundary value problem for an integrable equation is usually complicated by the fact that not all boundary values are known for a well-posed problem cf. \cite{FL2012}. This issue does not arise for (\ref{Goursat}) which is a Goursat problem. This means that the presented solution is as effective as the solution of the initial value problem via the inverse scattering transform for an equation such as the KdV or nonlinear Schr\"odinger equation.

Despite its great importance in the context of gravitational waves, there are few results in the literature on the Goursat problem (\ref{Goursat}). In fact, rather than solving a given initial or boundary value problem, most of the literature on the Ernst equation has dealt with the generation of new exact solutions via solution-generating techniques, cf. \cite{G1991, K1999, KR2005}. 
Solving an initial or boundary value problem is much more difficult than generating particular solutions. In fact, even if a large class of particular solutions are known, the problem of determining which of these solutions satisfies the given initial and boundary conditions remains a highly nonlinear problem, often as difficult as the original problem. As noted by Griffiths \cite[p. 210]{G1991}, ``What would be much more significant would be to find a practical way to determine the solution in the interaction region for an arbitrary set of initial conditions.'' 

Regarding the problem of determining the interaction of two colliding plane waves from arbitrary initial conditions, important first progress was made in a series of papers by Hauser and Ernst, see \cite{HE1990}. Their approach is based on the so-called Kinnersley $H$-potential \cite{K1977} rather than on equation (\ref{ernst}). In terms of the $2\times 2$-matrix valued Kinnersley potential $H(r,s)$, the problem of determining the spacetime metric in the interaction region can be formulated as a Goursat problem in the triangular region 
\begin{align}\label{Deltadef}  
  \Delta = \{(r,s) \in \R^2 \, | -1 \leq r < s \leq 1\}
\end{align} 
for the equation (see Eq. (2.10) in \cite{HE1990})
\begin{align}\label{HEequation}
2(s-r) H_{rs} \Omega - [H_r \Omega, H_s\Omega] = 0, \qquad \Omega = \begin{pmatrix} 0 & i \\ -i & 0 \end{pmatrix}.
\end{align}
Hauser and Ernst were able to relate the solution of this problem to the solution of a homogeneous Hilbert problem. 
The analysis of \cite{HE1990} relies, at least implicitly, on the fact that equation (\ref{HEequation}) admits the Lax pair (see Eq. (3.1) in \cite{HE1990})
$$P_r  = \frac{H_r \Omega}{2(\tau - r)}P, \qquad P_s= \frac{H_s \Omega}{2(\tau - s)}P,$$
where $P(r,s,\tau)$ is a $2 \times 2$-matrix valued eigenfunction and $\tau \in \C$ is the spectral parameter. 
%There exists a (nontrivial) relationship between equation (\ref{HEequation}) and the Ernst equation (\ref{ernst}) and their associated Lax pairs; we refer to \cite{FST1999} for some remarks on this relationship. 

More recently, the authors of \cite{FST1999} have addressed the Goursat problem in the triangle $\Delta$ for the equation
\begin{align}\label{FSTequation}
2(s-r)g_{rs} + g_r - g_s + (r-s)(g_rg^{-1}g_s + g_sg^{-1}g_r) = 0,
\end{align}
where $g(r,s)$ is a $2 \times 2$-matrix valued function. Equation (\ref{FSTequation}) is related to the hyperbolic Ernst equation (\ref{ernst}) as follows: Letting 
$$g(r,s) = \frac{s-r}{2\re \mathcal{E}} \begin{pmatrix} |\mathcal{E}|^2 & \im \mathcal{E} \\
\im \mathcal{E} & 1 \end{pmatrix},$$
equation (\ref{FSTequation}) reduces to the scalar equation
\begin{align}\label{ernstrs}  
  (\re \mathcal{E}) \left(\mathcal{E}_{rs} - \frac{\mathcal{E}_r - \mathcal{E}_s}{2(r-s)}\right) = \mathcal{E}_r \mathcal{E}_s,
\end{align}  
which is related to equation (\ref{ernst}) by the change of variables $y = (r+1)/2$ and $x = (1-s)/2$.
Through a clever series of steps, the authors of \cite{FST1999} express the solution of (\ref{FSTequation}) in terms of the solution of a RH problem. 

Our approach here is inspired by the recent works \cite{Lholedisk} and \cite{LFernst} on the elliptic Ernst equation. We have also drawn some inspiration from \cite{FST1999} and \cite{HE1990}, although in contrast to these references, we analyze equation (\ref{ernst}). Two further differences between the present work and \cite{FST1999} are: 
\begin{enumerate}[$(i)$]

\item It is assumed in \cite{FST1999} that the solution is $C^2$ on all of $\Delta$ up to and including the non-diagonal part of the boundary. However, as explained above (see equation (\ref{boundaryconditions})), the Ernst potentials relevant for gravitational waves have boundary values $\mathcal{E}(x,0)$ and $\mathcal{E}(0,y)$ whose derivatives are not continuous (actually unbounded) at the origin. Here we allow for such singularities in $\mathcal{E}_x(x,0)$ and $\mathcal{E}_y(0,y)$. These singularities transfer, in general, into singularities of the associated eigenfunction solutions of the Lax pair, and the rigorous treatment of all these singularities was one of the main challenges of the present work.

\item The normalization condition for the RH problem derived in \cite{FST1999} involves the solution itself; hence the solution representation is not effective. We circumvent this problem by defining the eigenfunctions on a Riemann surface $\mathcal{S}_{(x,y)}$ with branch points at $x$ and $1-y$. The Riemann surface $\mathcal{S}_{(x,y)}$ is dynamic in the sense that it depends on the spatial point $(x,y)$. This dependence on $(x,y)$ creates some technical difficulties which we handle by introducing a map $F_{(x,y)}$ from $\mathcal{S}_{(x,y)}$ to the standard Riemann sphere which takes the two moving branch points to the two fixed points $-1$ and $1$. After transferring the RH problem to the Riemann sphere in this way, we can analyze it using techniques from the theory of singular integral equations. 
\end{enumerate}

In the traditional implementation of the inverse scattering transform, the two equations in the Lax pair are treated separately---usually the spatial part of the Lax pair is first used to define the scattering data and the temporal part is then used to determine the time evolution. 
The Goursat problem (\ref{Goursat}) does not fit this pattern, so a different approach is required; this is one reason why the solution of the problem (\ref{Goursat}) has proved elusive. Actually, the approach in \cite{FST1999} was one of the first implementations of a general framework for the analysis of boundary value problems for integrable PDEs now known as the unified transform or Fokas method \cite{F1997}. In this method the two equations in the Lax pair are analyzed simultaneously rather than separately. The ideas of this method play an important role also in this paper.

It is an interesting open problem to investigate whether existence and uniqueness results for (\ref{Goursat}) can be obtained also via functional analytic techniques.
As was explained already in Chapter IV of Goursat's original treatise \cite{G1964}, existence and uniqueness results for Goursat problems for linear hyperbolic PDEs can be established by means of successive approximations and Riemann's method (see also \cite{CH1962}). It is possible to extend these ideas to prove existence theorems also for certain nonlinear Goursat problems \cite{P1969, T1962}. However, even in the linear case, these theorems tend to assume that $\{\mathcal{E}, \mathcal{E}_x, \mathcal{E}_y, \mathcal{E}_{xy}\}$ are all continuous \cite{P1969, G1964, CH1962}, or at least that the boundary values are Lipschitz \cite{T1962}. These conditions fail for the assumptions (\ref{E0E1assumptions}) relevant for gravitational waves. 

Let us finally point out that many exact solutions describing colliding plane gravitational waves are known (see e.g. \cite{NH1977, CX1987, TW1992, ET1989}) and that there is a growing literature on colliding gravitational waves which are not necessarily plane (see e.g. \cite{LR2017}).

\subsection{Organization of the paper}
We begin by establishing some notation in Section \ref{notationsec}. Our main results (Theorems \ref{mainth1}-\ref{mainth4}) are stated in Section \ref{mainresultsec}. 

In Section \ref{linearlimitsec}, as preparation for the general case, we analyze the special case in which the colliding waves have collinear polarization. In this case, the problem reduces to a problem for the so-called Euler-Darboux equation. We prove a theorem for this equation (Theorem \ref{linearmainth}) which is  analogous to Theorem \ref{mainth1}-\ref{mainth4}. %The insights gained in Section \ref{linearlimitsec} are crucial for the following sections, because the proof of Theorem \ref{linearmainth} follows steps which are very similar to, but technically simpler than, those used in the proofs of Theorem \ref{mainth1}-\ref{mainth4}. 

In Section \ref{ernstsec}, we discuss the Lax pair of equation \eqref{ernst} and analyze the spectral data as well as the uniqueness of the solution of the corresponding RH problem.

In Section \ref{proofsec}, we present the proofs of Theorem \ref{mainth1}-\ref{mainth4}. 

Section \ref{examplesec} contains two short examples and the appendix contains some background on the origin of the Goursat problem (\ref{Goursat}) in the context of colliding gravitational waves.

\subsection{Organization of the paper}
We begin by establishing some notation in Section \ref{notationsec}. Our main results (Theorems \ref{mainth1}-\ref{mainth4}) are stated in Section \ref{mainresultsec}. 

In Section \ref{linearlimitsec}, as preparation for the general case, we analyze the special case in which the colliding waves have collinear polarization. In this case, the problem reduces to a problem for the so-called Euler-Darboux equation. We prove a theorem for this equation (Theorem \ref{linearmainth}) which is  analogous to Theorem \ref{mainth1}-\ref{mainth4}. %The insights gained in Section \ref{linearlimitsec} are crucial for the following sections, because the proof of Theorem \ref{linearmainth} follows steps which are very similar to, but technically simpler than, those used in the proofs of Theorem \ref{mainth1}-\ref{mainth4}. 

In Section \ref{ernstsec}, we discuss the Lax pair of equation \eqref{ernst} and analyze the spectral data as well as the uniqueness of the solution of the corresponding RH problem.

In Section \ref{proofsec}, we present the proofs of Theorem \ref{mainth1}-\ref{mainth4}. 

Section \ref{examplesec} contains two short examples and the appendix contains some background on the origin of the Goursat problem (\ref{Goursat}) in the context of colliding gravitational waves.

\section{Notation}\label{notationsec}
We introduce notation that will be used throughout the paper. 

We let $D$ denote the triangular region defined in (\ref{Ddef}) and displayed in Figure \ref{D.pdf}. Given $\delta > 0$, we let $D_\delta$ denote the slightly smaller triangular region obtained by removing a narrow strip along the diagonal of $D$ as follows (see Figure \ref{Ddelta.pdf}):
\begin{align}\label{Ddeltadef}
D_\delta = \{(x,y) \in D \, | \, x+y < 1-\delta\},
\end{align}
The interiors of  $D$ and $D_\delta$ will be denoted by $\Int D$ and $\Int D_\delta$, respectively. 
The Riemann sphere will be denoted by $\hat{\C} = \C \cup \{\infty\}$.

\begin{figure}
\bigskip
\begin{center}
 \begin{overpic}[width=.4\textwidth]{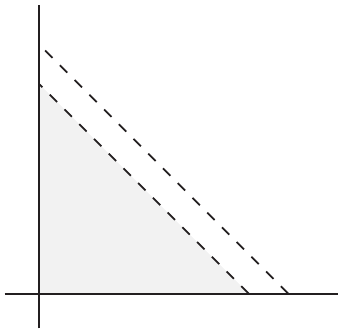}
 \put(25,28){ $D_\delta$}
 \put(102.5,9.7){\small $x$}
 \put(10,100){\small $y$}
 \put(85,4.5){\small $1$}
 \put(65,4){\small $1-\delta$}
 \put(6,82){\small $1$}
 \put(-4.5,72){\small $1-\delta$}
 %\put(45,55){\small \rotatebox{-45}{$x+y=1$}}
 \end{overpic}
   \begin{figuretext}\label{Ddelta.pdf}
      The triangle $D_\delta$ defined in (\ref{Ddeltadef}).
      \end{figuretext}
   \end{center}
\end{figure}

\subsection{The Riemann surface $\mathcal{S}_{(x,y)}$}
For each $(x,y) \in D$, we let $\mathcal{S}_{(x,y)}$ denote the Riemann surface consisting of all points $P := (\lambda, k) \in \C^2$ such that
\begin{align}\label{lambdadef}
\lambda^2 = \frac{k - (1-y)}{k - x}
\end{align}
together with two points $\infty^+=(1,\infty)$ and $\infty^-=(-1,\infty)$ at infinity and a branch point $x \equiv (\infty,x)$ which make the surface compact.
The surface $\mathcal{S}_{(x,y)}$ is two-sheeted in the sense that to each $k \in \hat{\C} \setminus \{x, 1-y\}$, there correspond exactly two values of $\lambda$. We introduce a branch cut in the complex $k$-plane from $x$ to $1-y$ and, for $k \in \hat{\C}\setminus [x,1-y]$, we let $k^+$ and $k^-$ denote the corresponding points on the upper and lower sheet of $\mathcal{S}_{(x,y)}$, respectively. By definition, the upper (lower) sheet is characterized by $\lambda \to 1$ ($\lambda \to -1$) as $k \to \infty$. Writing $\lambda(x,y,P)$ for the value of $\lambda$ corresponding to the point $P \in \mathcal{S}_{(x,y)}$, we have
\begin{align}\label{lambdasqrt}
 \lambda(x,y,k^+) = \sqrt{\frac{k - (1-y)}{k - x}} = -\lambda(x,y,k^-), \qquad k \in \hat{\C} \setminus [x, 1-y],
 \end{align}
where the sign of the square root in (\ref{lambdasqrt}) is chosen so that $\lambda(x,y,k^+)$ has positive real part.

\begin{figure}
\bigskip\bigskip
\begin{center}
 \begin{overpic}[width=.9\textwidth]{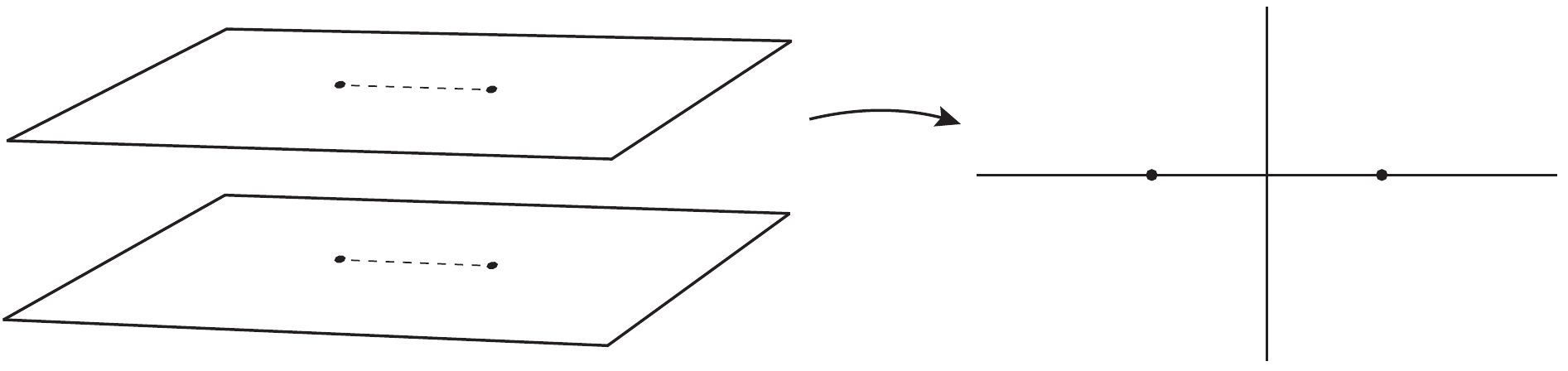}
 \put(18.9,17.8){\small $x$}
 \put(32.4,17.2){\small $1-y$}
 \put(19,7){\small $x$}
 \put(32.4,6){\small $1-y$}
 \put(71.5,9.6){\small $-1$}
 \put(87.6,9.6){\small $1$}
 \put(100.8,11.7){\small $\re z$}
 \put(79,25){\small $\im z$}
 \put(80,-4){\small $\hat{\C}$}
 \put(24,-4){\small $\mathcal{S}_{(x,y)}$}
 \put(53.5,18.5){\small $F_{(x,y)}$}
 \end{overpic}
   \bigskip\bigskip\medskip
   \begin{figuretext}\label{kzmap.pdf}
      The map $F_{(x,y)}:k \mapsto z = \frac{1+ \lambda}{1 - \lambda}$ is a biholomorphism from the two-sheeted Riemann surface $\mathcal{S}_{(x,y)}$ to the Riemann sphere $\hat{\C} = \C \cup \{\infty\}$. It maps the branch points $x$ and $1-y$ to $z = -1$ and $z = 1$, respectively, and the upper (lower) sheet to the outside (inside) of the unit circle.
      \end{figuretext}
   \end{center}
\end{figure}

\subsection{The map $F_{(x,y)}$}
For each point $(x,y) \in D$, $\mathcal{S}_{(x,y)}$ is a compact genus zero Riemann surface with branch points at $k = x$ and $k = 1-y$. In order to fix the locations of these branch points, we introduce a new variable $z$ by
$$z = \frac{1+ \lambda}{1 - \lambda},$$
and let $F_{(x,y)}:\mathcal{S}_{(x,y)} \to \hat{\C}$ be the map that sends $P$  to $z$, i.e.,
$$F_{(x,y)}(P) = \frac{1+ \lambda(x,y,P)}{1 - \lambda(x,y,P)}, \qquad P \in \mathcal{S}_{(x,y)}.$$
For each $(x,y) \in D$, $F_{(x,y)}$ is a biholomorphism (i.e. a bijective holomorphic function whose inverse is also holomorphic) from $\mathcal{S}_{(x,y)}$ to $\hat{\C}$ which maps the two branch points $x$ and $1-y$ to $z = -1$ and $z = 1$, respectively, see Figure \ref{kzmap.pdf}.

\subsection{The contours $\Sigma$ and $\Gamma$}
For each $(x,y) \in D$, we let $\Sigma_0 \equiv \Sigma_0(x,y)$ denote the shortest path from $0^+$ to $0^-$ in $\mathcal{S}_{(x,y)}$, and we let $\Sigma_1  \equiv \Sigma_1(x,y)$ denote the shortest path from $1^-$ to $1^+$ in $\mathcal{S}_{(x,y)}$. More precisely,
\begin{align}\label{Sigma01def}
\Sigma_0 = [0,x]^+ \cup [x,0]^-, \qquad \Sigma_1 =  [1,1-y]^- \cup [1-y,1]^+ ,
\end{align}
where, for a subset $S$ of the complex plane, we use the notation $S^\pm = \{k^\pm \in \mathcal{S}_{(x,y)}\,| \, k \in S\}$ to denote the sets in the upper and lower sheets of $\mathcal{S}_{(x,y)}$ which project onto $S$, see Figure \ref{Sigma01.pdf}.
We write $\Sigma := \Sigma_0 \cup \Sigma_1$ for the union of $\Sigma_0$ and $\Sigma_1$.

Given $(x,y) \in D$, we let $\Gamma_0 \equiv \Gamma_0(x,y)$ and $\Gamma_1 \equiv \Gamma_1(x,y)$ denote two clockwise nonintersecting smooth contours in the complex $z$-plane which encircle the real intervals
\begin{subequations}\label{twointervals}
\begin{align}
F_{(x,y)}(\Sigma_0) = \bigg[-\frac{\sqrt{1-y}+ \sqrt{x}}{\sqrt{1-y} - \sqrt{x}}, -\frac{\sqrt{1-y}- \sqrt{x}}{\sqrt{1-y}+ \sqrt{x}}\bigg]
\end{align}
and
\begin{align}
F_{(x,y)}(\Sigma_1) =\bigg[\frac{\sqrt{1-x} - \sqrt{y}}{\sqrt{1-x} + \sqrt{y}}, \frac{\sqrt{1-x} + \sqrt{y}}{\sqrt{1-x} - \sqrt{y}}\bigg],
\end{align}
\end{subequations}
respectively, but which do not encircle zero, see Figure \ref{Gamma.pdf}. We let $\Gamma \equiv \Gamma(x,y)$ denote the union $\Gamma := \Gamma_0 \cup \Gamma_1$ of $\Gamma_0$ and $\Gamma_1$.

\begin{figure}
\bigskip\bigskip
\begin{center}
\hspace{-.97cm} \begin{overpic}[width=.94\textwidth]{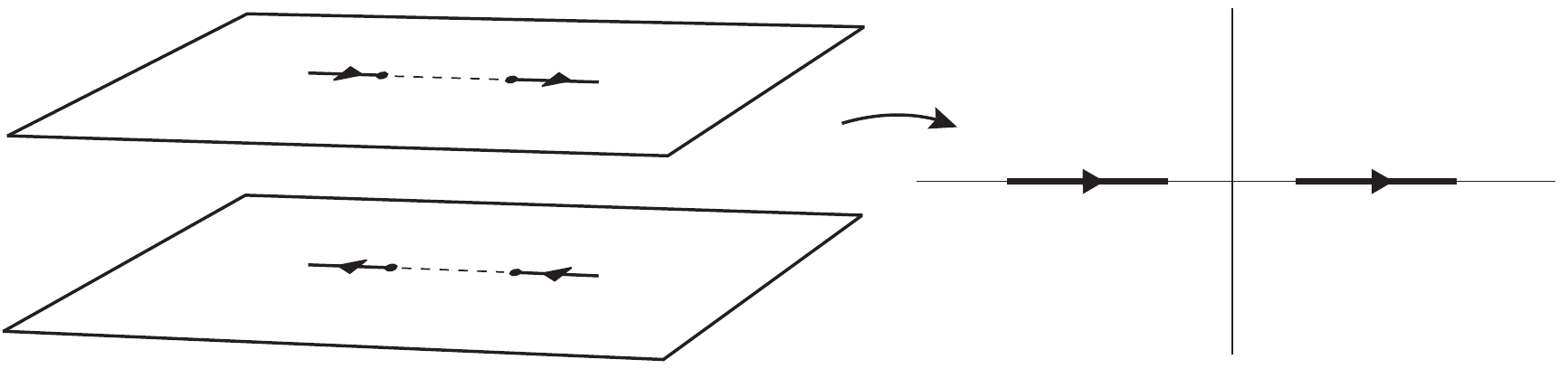}
 \put(21.1,16.6){\tiny $\Sigma_0$}
 \put(34.3,16.2){\tiny $\Sigma_1$}
 \put(21.1,4.3){\tiny $\Sigma_0$}
 \put(34.3,3.8){\tiny $\Sigma_1$}
  \put(17,18.5){\tiny $0^+$}
 \put(39,17.8){\tiny $1^+$}
 \put(17,6.4){\tiny $0^-$}
 \put(39,5){\tiny $1^-$}
 \put(64.2,8.8){\small $F_{(x,y)}(\Sigma_0)$}
 \put(82.7,8.8){\small $F_{(x,y)}(\Sigma_1)$}
 \put(101,11.3){\small $\re z$}
 \put(76.5,24.9){\small $\im z$}
  \put(54.5,18){\small $F_{(x,y)}$}
 \end{overpic} 
   \begin{figuretext}\label{Sigma01.pdf}
     The map $F_{(x,y)}$ sends the contours $\Sigma_0$ and $\Sigma_1$ onto the two real intervals $F_{(x,y)}(\Sigma_0)$ and $F_{(x,y)}(\Sigma_1)$, respectively.
      \end{figuretext}
   \end{center}
\end{figure}

\begin{figure}
\bigskip
\begin{center}
 \begin{overpic}[width=.6\textwidth]{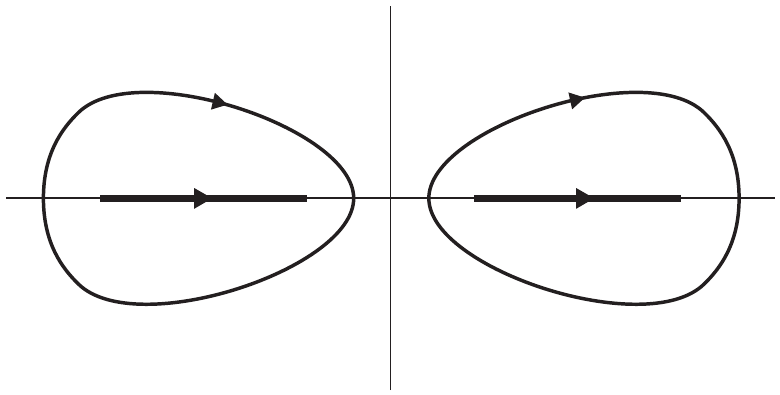}
 \put(101.5,24){\small $\re z$}
 \put(47,52){\small $\im z$}
 \put(27,40){\small $\Gamma_0$}
 \put(71,40.5){\small $\Gamma_1$}
\put(17,20.5){\small $F_{(x,y)}(\Sigma_0)$}
 \put(65,20.5){\small $F_{(x,y)}(\Sigma_1)$}
 \end{overpic}
   \begin{figuretext}\label{Gamma.pdf}
      The contour $\Gamma$ in the complex $z$-plane is the union of the loops $\Gamma_0$ and $\Gamma_1$ which encircle the intervals $F_{(x,y)}(\Sigma_0)$ and $F_{(x,y)}(\Sigma_1)$ respectively.
      \end{figuretext}
   \end{center}
\end{figure}

\subsection{Nontangential limits and function spaces}
Let $\Gamma \subset \C$ be a piecewise smooth contour. 
For an analytic function $m: \C \setminus \Gamma \to \C$, we denote the nontangential boundary values of $m$ from the left and right sides of $\Gamma$ by $m_+$ and $m_-$ respectively.
Given a subset $S \subset \R^n$, $n \geq 1$, we let $C(S)$ denote the space of complex-valued continuous functions on $S$. If $S$ is open, we define $C^n(S)$ as the space of complex-valued functions on $S$ which are $n$ times continuously differentiable, i.e., all partial derivatives of order $\leq n$ exist and are continuous.
By $\mathcal{B}(X,Y)$, we denote the space of bounded linear maps from a Banach space $X$ to another Banach space $Y$ equipped with the standard operator norm; if $X = Y$, we write $\mathcal{B}(X) \equiv \mathcal{B}(X,X)$.

\section{Main results}\label{mainresultsec}

We adopt the following notion of a $C^n$-solution of the Goursat problem (\ref{Goursat}).

\begin{definition}\label{solutiondef}\upshape
Let $\mathcal{E}_0(x)$, $x \in [0, 1)$, and $\mathcal{E}_1(y)$, $y \in [0,1)$, be complex-valued functions.
A function $\mathcal{E}:D \to \R$ is called a {\it $C^n$-solution of the Goursat problem for (\ref{ernst}) in $D$ with data $\{\mathcal{E}_0, \mathcal{E}_1\}$} if 
\begin{align*}
\begin{cases}
  \mathcal{E} \in C(D) \cap C^n(\Int(D)), 
  	\\
  \text{$\mathcal{E}(x,y)$ satisfies the hyperbolic Ernst equation (\ref{ernst}) in $\Int(D)$,}
	\\
 \text{$x^\alpha \mathcal{E}_x, y^\alpha \mathcal{E}_y, x^\alpha y^\alpha \mathcal{E}_{xy} \in C(D)$ for some $\alpha \in [0,1)$,}
  	\\
  \text{$\mathcal{E}(x,0) = \mathcal{E}_0(x)$ for $x \in [0,1)$,}
 	\\
  \text{$\mathcal{E}(0,y) = \mathcal{E}_1(y)$ for $y \in [0,1)$,}
	\\
  \text{$\re \mathcal{E}(x,y) > 0$ for $(x,y) \in D$.}	
\end{cases}
\end{align*}
\end{definition}

We next state the four main results of the paper (Theorem \ref{mainth1}-\ref{mainth4}), which all address different aspects of the Goursat problem (\ref{Goursat}). 

In the formulation of Theorem \ref{mainth1}-\ref{mainth4}, it is assumed that $n \geq 2$ is an integer and that $\mathcal{E}_0(x)$, $x \in [0, 1)$, and $\mathcal{E}_1(y)$, $y \in [0,1)$, are two complex-valued functions satisfying the assumptions in (\ref{E0E1assumptions}) for a fixed $\alpha \in [0,1)$. The first theorem provides a representation formula for the solution in terms of the given boundary data via a RH problem.

\begin{theorem}[Representation formula]\label{mainth1}
If $\mathcal{E}(x,y)$ is a $C^n$-solution of the Goursat problem for (\ref{ernst}) in $D$ with data $\{\mathcal{E}_0, \mathcal{E}_1\}$, then this solution can be expressed in terms of the boundary values $\mathcal{E}_0(x)$ and $\mathcal{E}_1(y)$ by 
\begin{align}\label{Erecover}
\mathcal{E}(x,y) = \frac{1 + (m(x,y,0))_{11} - (m(x,y,0))_{21}}{1 + (m(x,y,0))_{11} + (m(x,y,0))_{21}},
\end{align}
where $m(x,y,z)$ is the unique solution of the $2 \times 2$-matrix RH problem
\begin{align}\label{RHm}
\begin{cases} 
\text{$m(x, y, \cdot)$ is analytic in $\C \setminus \Gamma$}, 
	\\
\text{$m_+(x, y, z) = m_-(x, y, z) v(x, y, z)$ for all $z \in \Gamma$},
\\
\text{$m(x,y,z) = I + O(z^{-1})$ as $z \to \infty$},
\end{cases} 
\end{align}
and the jump matrix $v(x,y,z)$ is defined as follows: Let $\Phi_0$ and $\Phi_1$ be the unique solutions of the linear Volterra integral equations
\begin{subequations}\label{Phi0Phi1}
\begin{align}\label{Phi0Phi1a}
\Phi_0(x,k^\pm) = I + \int_0^x (\mathsf{U}_0\Phi_0)(x', k^\pm) dx', \qquad x \in [0, 1), \ k \in \C \setminus [0,1],
	\\ \label{Phi0Phi1b}
\Phi_1(y,k^\pm) = I + \int_0^y (\mathsf{V}_1\Phi_1)(y', k^\pm) dy', \qquad y \in [0, 1), \ k \in \C \setminus [0,1],	
\end{align}
\end{subequations}
where $\mathsf{U}_0$ and $\mathsf{V}_1$ are defined by 
\begin{subequations}\label{U0V1def}
\begin{align}\label{U0def}
& \mathsf{U}_0(x,k^\pm) = \frac{1}{2 \re \mathcal{E}_0(x)} \begin{pmatrix} \overline{\mathcal{E}_{0x}(x)} & \lambda(x,0,k^\pm) \overline{\mathcal{E}_{0x}(x)} \\
\lambda(x,0,k^\pm) \mathcal{E}_{0x}(x) & \mathcal{E}_{0x}(x) \end{pmatrix}, 
	\\ 
& \mathsf{V}_1(y,k^\pm) =  \frac{1}{2 \re \mathcal{E}_0(y)}  \begin{pmatrix} \overline{\mathcal{E}_{1y}(y)} & \frac{1}{\lambda(0,y,k^\pm)} \overline{\mathcal{E}_{1y}(y)} \\
\frac{1}{\lambda(0,y,k^\pm)} \mathcal{E}_{1y}(y) & \mathcal{E}_{1y}(y)  
\end{pmatrix}.
\end{align}
\end{subequations}
Then
\begin{align}\label{jumpdef}
v(x,y, z) = \begin{cases}
 \Phi_0\big(x, F_{(x,y)}^{-1}(z)\big), \quad & z \in \Gamma_0, 
 	\\
\Phi_1\big(y, F_{(x,y)}^{-1}(z)\big), & z \in \Gamma_1,
\end{cases} \quad (x,y) \in D.
\end{align}
\end{theorem}

Theorem \ref{mainth2} establishes uniqueness of the $C^n$-solution. 

\begin{theorem}[Uniqueness]\label{mainth2}
The $C^n$-solution $\mathcal{E}(x,y)$ of the Goursat problem for (\ref{ernst}) in $D$ with data $\{\mathcal{E}_0, \mathcal{E}_1\}$ is unique, if it exists. In fact, the value of $\mathcal{E}$ at a point $(x,y) \in D$ is uniquely determined by the boundary values $\mathcal{E}_0(x')$ and $\mathcal{E}_1(y')$ for $0 \leq x' \leq x$ and $0 \leq y' \leq y$.
\end{theorem}

Theorem \ref{mainth3} establishes existence of a $C^n$-solution---in the collinear case, for general data; otherwise under a small-norm assumption.

\begin{theorem}[Existence and regularity]\label{mainth3}
For each $\delta > 0$, the following three existence and regularity results hold:
	\begin{enumerate}[$(a)$] 
\item Suppose the $2 \times 2$-matrix RH problem \eqref{RHm} has a solution for all $(x,y)\in D_\delta$. Then there exists a $C^n$-solution of the Goursat problem for \eqref{ernst} in $D_\delta$ with data 
$\{\mathcal{E}_0|_{[0,1-\delta)},\mathcal{E}_1|_{[0,1-\delta)} \}$. 
\item Whenever the $L^1$-norms of $\mathcal{E}_{0x}/(\re \mathcal{E}_0)$ and $\mathcal{E}_{0y}/(\re \mathcal{E}_1)$ on $[0,1-\delta)$ are sufficiently small, there exists a $C^n$-solution of the Goursat problem for \eqref{ernst} in $D_\delta$ with data $\{\mathcal{E}_0, \mathcal{E}_1\}$.
\item If $\mathcal{E}_0, \mathcal{E}_1 >0$ on $[0,1-\delta)$, i.e., if the incoming waves are collinearly polarized, then there exists a $C^n$-solution of the Goursat problem for \eqref{ernst} in $D_\delta$ with data 
$\{\mathcal{E}_0|_{[0,1-\delta)},\mathcal{E}_1|_{[0,1-\delta)} \}$. 
	\end{enumerate}
\end{theorem}

\begin{remark}\upshape
Part $(a)$ of Theorem \ref{mainth3} shows that the solution $\mathcal{E}(x,y)$ exists and has the same regularity as the given data as long as the associated RH problem has a solution. By taking $\delta > 0$ arbitrarily small, we see that the same statement holds also in all of $D$.
\end{remark}

Theorem \ref{mainth4} establishes explicit formulas for the singular behavior of the solution near the boundary in terms of the given data. 

\begin{theorem}[Boundary behavior]\label{mainth4}
Let $\alpha \in (0,1)$ and $n \geq 2$ be an integer. Let $\mathcal{E}(x,y)$ be a $C^n$-solution of the Goursat problem for (\ref{ernst}) in $D$ with data $\{\mathcal{E}_0, \mathcal{E}_1\}$.
Let $m_1, m_2 \in \C$ denote the values of these functions at the origin, i.e.,
\begin{align}\label{m1m2def}
m_1 = \lim_{x \downarrow 0}  x^\alpha \mathcal{E}_{0x}(x), \qquad
m_2 = \lim_{y \downarrow 0}  y^\alpha \mathcal{E}_{1y}(y).
\end{align}
Then the solution  $\mathcal{E}(x,y)$ has the following behavior near the boundary:
\begin{subequations}\label{boundarylimit}
\begin{align}\label{boundarylimita}
& \lim_{x \downarrow 0} x^\alpha \mathcal{E}_x(x,y) = m_1 \frac{e^{i\int_0^y \frac{\im \mathcal{E}_{1y}(y')}{\re \mathcal{E}_1(y')} dy'} \re \mathcal{E}_1(y)}{\sqrt{1-y}} && \text{for each $y \in [0,1)$},
	\\ \label{boundarylimitb}
& \lim_{y \downarrow 0} y^\alpha \mathcal{E}_y(x,y) = m_2 \frac{e^{i\int_0^x \frac{\im \mathcal{E}_{0x}(x')}{\re \mathcal{E}_0(x')} dx'} \re \mathcal{E}_0(x)}{\sqrt{1-x}} && \text{for each $x \in [0,1)$}.
\end{align}
\end{subequations}
In particular,
\begin{align*}
& \lim_{x \downarrow 0} x^\alpha |\mathcal{E}_x(x,y)| = |m_1| \frac{\re \mathcal{E}_1(y)}{\sqrt{1-y}} && \text{for each $y \in [0,1)$},
	\\
& \lim_{y \downarrow 0} y^\alpha |\mathcal{E}_y(x,y)| = |m_2| \frac{\re \mathcal{E}_0(x)}{\sqrt{1-x}} && \text{for each $x \in [0,1)$}.
\end{align*}

\end{theorem}

\begin{remark}\upshape
Theorem \ref{mainth4} yields the following important result for the collision of plane gravitational waves: A solution $\mathcal{E}(x,y)$ of the Goursat problem for (\ref{ernst}) fulfills the gravitational wave boundary conditions (\ref{boundaryconditions}) if and only if the boundary data  $\mathcal{E}(x,0) = \mathcal{E}_0(x)$ and $\mathcal{E}(0,y) = \mathcal{E}_1(y)$  are such that 
$\lim_{x \downarrow 0}  x^{\alpha} |\mathcal{E}_{0x}(x)|$ and $\lim_{y \downarrow 0}  y^{\alpha} |\mathcal{E}_{1y}(y)|$ belong to the real interval $[1, \sqrt{2})$.
In particular, the behavior of $\mathcal{E}_{x}(x,0)$ and $\mathcal{E}_{y}(0,y)$ at the origin fully determines whether the functions $\mathcal{E}_x(x,y)$ and $\mathcal{E}_y(x,y)$ have the appropriate singular behavior near the edges $\partial D \cap \{x =0\}$  and $\partial D \cap \{y =0\}$.
\end{remark}

\section{Collinearly polarized waves}\label{linearlimitsec}
Before turning to the general case, it is useful to first consider the special case in which the Ernst potential $\mathcal{E}$ is strictly positive. In the context of gravitational waves, this corresponds to the important situation when the two colliding waves have collinear polarization, see \cite{G1991}. 

%Before turning to the nonlinear Ernst equation (\ref{ernst}) and the proof of Theorem \ref{mainth}, it is useful to consider the linearized version of (\ref{ernst}). Indeed, the linearized version of (\ref{ernst}) can be analyzed following steps which are conceptually very similar to those involved in the analysis of (\ref{ernst}).

\subsection{The Euler-Darboux equation}
If the Ernst potential $\mathcal{E}$ is strictly positive, we can write $\mathcal{E}(x,y) = e^{-V(x,y)}$, where $V(x,y)$  is a real-valued function. A simple computation then shows that $\mathcal{E}$ satisfies the Ernst equation (\ref{ernst}) if and only if $V$ satisfies the linear hyperbolic equation
\begin{align}\label{linearernst}  
  V_{xy} - \frac{V_x + V_y}{2(1-x-y)} = 0,
\end{align}  
which is a version of the Euler-Darboux equation \cite{M1973}. Since (\ref{linearernst}) is a linear equation, we can, without loss of generality, assume that $V$ is real-valued and that $V(0,0) =0$. 

\begin{remark}[Linear limit]\upshape
In addition to being a reformulation of (\ref{ernst}) in the special case of collinearly polarized waves, equation (\ref{linearernst}) can also be viewed as the linearized version of (\ref{ernst}). Indeed, substituting $\mathcal{E}(x,y) = 1 + \epsilon V(x,y) + O(\epsilon^2)$ into (\ref{ernst}) and considering the terms of $O(\epsilon)$, we see that (\ref{linearernst}) is the linear limit of (\ref{ernst}).
\end{remark}

The analysis of the Euler-Darboux equation (\ref{linearernst}) presented in this section serves two purposes. First, it is used to prove the part of Theorem \ref{mainth3} regarding existence in the collinearly polarized case. Second, it turns out that the more difficult case of noncollinearly polarized solutions can be analyzed following steps which are conceptually very similar to---but technically more difficult than---those involved in the analysis of the collinear case.
In fact, the analysis of (\ref{ernst}) presented in later sections strongly relies on the insight gained in this section.

We are interested in the following Goursat problem for (\ref{linearernst}) in the triangle $D$: Given $V_0(x)$, $x \in [0, 1)$, and $V_1(y)$, $y \in [0,1)$, find a solution $V(x,y)$ of (\ref{linearernst}) in $D$ such that $V(x,0) = V_0(x)$ for $x \in [0,1)$ and $V(0,y) = V_1(y)$ for $y \in [0,1)$.
We introduce a notion of $C^n$-solution of this problem as follows.

\begin{definition}\label{linearsolutiondef}\upshape
Let $V_0(x)$, $x \in [0, 1)$, and $V_1(y)$, $y \in [0,1)$, be real-valued functions and $\alpha \in [0,1)$.
We define a function $V:D \to \R$ to be a {\it $C^n$-solution of the Goursat problem for (\ref{linearernst}) in $D$ with data $\{V_0, V_1\}$} if 
\begin{align*}
\begin{cases}
  V \in C(D) \cap C^n(\Int(D)), 
  	\\
  \text{$V(x,y)$ satisfies the Euler-Darboux equation (\ref{linearernst}) in $\Int(D)$,}
  	\\
 \text{$x^\alpha V_x, y^\alpha V_y, x^\alpha y^\alpha V_{xy} \in C(D)$ for some $\alpha \in [0,1)$,}
  	\\
  \text{$V(x,0) = V_0(x)$ for $x \in [0,1)$,}
 	\\
  \text{$V(0,y) = V_1(y)$ for $y \in [0,1)$.}
\end{cases}
\end{align*}
\end{definition}

The following theorem establishes the unique existence of a solution of the Goursat problem for (\ref{linearernst}) in $D$. It also provides a representation for the solution in terms of the boundary data and characterizes the singular behavior near the boundary. 

\begin{theorem}[Solution of the Euler-Darboux equation in a triangle]\label{linearmainth}
Let $n \geq 2$ be an integer. Let $V_0(x)$, $x \in [0, 1)$, and $V_1(y)$, $y \in [0,1)$, be two real-valued functions such that 
\begin{align}\label{V0V1assumptions}
\begin{cases}
 V_0, V_1 \in C([0,1)) \cap C^n((0,1)), 
 	\\
  \text{$x^\alpha V_{0x}, y^\alpha V_{1y} \in C([0,1))$ for some $\alpha \in [0,1)$,}
 	\\
V_0(0) = V_1(0) = 0.
\end{cases}
\end{align}

Then there exists a unique $C^n$-solution $V(x,y)$ of the Goursat problem for (\ref{linearernst}) in $D$ with data $\{V_0, V_1\}$.
Moreover, this solution is given in terms of the boundary values $V_0(x)$ and $V_1(y)$ by 
\begin{align}\label{linearVrecover}
V(x,y) = -\frac{1}{2}m(x,y,0), \qquad (x,y) \in D,
\end{align}
where $m(x,y,z)$ is the unique solution of the following scalar RH problem:
\begin{align}\label{linearRHm}
\begin{cases} 
\text{$m(x, y, \cdot)$ is analytic in $\C \setminus \Gamma$}, 
	\\
\text{$m_+(x, y, z) = m_-(x, y, z)  + v(x, y, z)$ for all $z \in \Gamma$},
	\\
\text{$m(x,y,z) =  O(z^{-1})$ as $z \to \infty$},
\end{cases} 
\end{align}
and the jump $v(x,y,z)$ is defined by
\begin{align}\label{linearjumpdef}
v(x,y, z) = \begin{cases}
 \Phi_0\big(x, F_{(x,y)}^{-1}(z)\big), \quad & z \in \Gamma_0, 
 	\\
\Phi_1\big(y, F_{(x,y)}^{-1}(z)\big), & z \in \Gamma_1,
\end{cases}
\end{align}
with
\begin{subequations}\label{linearPhi0Phi1}
\begin{align}\label{linearPhi0Phi1a}
\Phi_0(x,k^\pm) = \int_0^x \lambda(x',0,k^\pm) V_{0x}(x') dx', \qquad x \in [0, 1), \ k \in \C \setminus [0,1],
	\\\label{linearPhi0Phi1b}
\Phi_1(y,k^\pm) = \int_0^y \frac{1}{\lambda(0,y',k^\pm)} V_{1y}(y') dy', \qquad y \in [0, 1), \ k \in \C \setminus [0,1].
\end{align}
\end{subequations}

Furthermore, if $\alpha \in (0,1)$ is such that the functions $x^\alpha V_{0x}$ and $y^\alpha V_{1y}$ are continuous on $[0, 1)$ and 
\begin{align}\label{linearm1m2def}
m_1 := \lim_{x \downarrow 0}  x^\alpha V_{0x}(x), \qquad
m_2 := \lim_{y \downarrow 0}  y^\alpha V_{1y}(y),
\end{align}
then the solution $V(x,y)$ has the following behavior near the boundary:
\begin{subequations}
\begin{align}\label{linearboundarylimita}
& \lim_{x \downarrow 0} x^\alpha V_x(x,y) = \frac{m_1}{\sqrt{1-y}} && \text{for each $y \in [0,1)$},
	\\\label{linearboundarylimitb}
& \lim_{y \downarrow 0} y^\alpha V_y(x,y) = \frac{m_2}{\sqrt{1-x}} && \text{for each $x \in [0,1)$}.
\end{align}
\end{subequations}

\end{theorem}

\begin{remark}\upshape
The scalar RH problem (\ref{linearRHm}) has the unique solution
$$m(x,y,z) = \frac{1}{2\pi i} \int_{\Gamma} \frac{v(x,y,z')}{z'-z} dz'.$$
Hence the solution $V(x,y)$ can be expressed in terms of $v$ by
\begin{align}\label{linearV}
V(x,y) = -\frac{1}{4\pi i} \int_{\Gamma} \frac{v(x,y,z)}{z} dz.
\end{align}
Collapsing the contour $\Gamma$ in (\ref{linearV}) onto the intervals in (\ref{twointervals}) and changing variables from  $z$ to $k$ leads to the following representation for the solution in terms of Abel type integrals:
\begin{align}\nonumber
  V(x,y) = &\; \frac{1}{\pi} \int_0^x \frac{\sqrt{1-k}}{\sqrt{(1-y-k)(x-k)}} \bigg(\int_0^k \frac{V_{0x}(x')}{\sqrt{k - x'}}dx'\bigg) dk
  	\\ \label{linearUrepresentation}
&  + \frac{1}{\pi} \int_{1-y}^1 \frac{\sqrt{k}}{\sqrt{(k - (1-y))(k-x)}} \bigg(\int_0^{1-k} \frac{V_{1y}(y')}{\sqrt{1 - y' - k}} dy'\bigg) dk
\end{align}  
for $(x,y) \in D$. Formulas analogous to (\ref{linearUrepresentation}) for equation (\ref{linearernst}) have been derived in \cite{HE1990} and \cite{FST1999}. 
\end{remark}

\begin{remark}\upshape
The representation (\ref{linearUrepresentation}) can be found more directly by formulating a RH problem for $\Phi$ on $\mathcal{S}_{(x,y)}$ with jump across $\Sigma$. This is essentially the approach adopted in \cite{FST1999}.
The representation (\ref{linearUrepresentation}) has the advantage that it is explicit in its dependence on $V_0$ and $V_1$, but it has the disadvantage that the integrands are singular at some of the endpoints of the integration intervals. These singularities complicate the verification that $V$ satisfies the appropriate regularity and boundary conditions, especially in the situation relevant for gravitational waves where $V_{0x}$ and $V_{1y}$ are singular at the origin. 
For the nonlinear equation (\ref{ernst}), this becomes a serious complication. For this reason, we have formulated the RH problems in Theorem \ref{mainth1} and Theorem \ref{linearmainth} in terms of the contour $\Gamma$ (which avoids the problematic endpoints of the intervals in (\ref{twointervals})) rather than in terms of a contour running along the real axis. However, the representation (\ref{linearUrepresentation}) allows for applying more classical techniques. This approach is used in \cite{MEulerDarboux} to compute an asymptotic expansion of the solution near the diagonal of $D$.
\end{remark}

\begin{remark}\upshape
In \cite{S1972} there was derived an alternative integral formula for the solution of the Goursat problem for the Euler--Darboux equation by applying Riemann's classical method \cite{CH1962, G1964}. Whereas the representation (\ref{linearUrepresentation}) relies on Abel integrals, the expression of \cite{S1972} is given in terms of the Legendre function $P_{-1/2}$ of order $-1/2$. 
\end{remark}

\begin{remark}\upshape
In order to emphasize the analogy between (\ref{ernst}) and its linearized version (\ref{linearernst}), we will use the same symbols in this section for the various linearized quantities as we use elsewhere for the corresponding quantities of the nonlinear problem. Many quantities which are matrices in the noncollinear case reduce to scalar quantities in the collinear case.  For example, in other sections $\Phi$ will denote a $2 \times 2$-matrix valued eigenfunction, but in this section $\Phi$ is a scalar-valued eigenfunction.
\end{remark}

\subsection{Proof of Theorem \ref{linearmainth}}
The proof of Theorem \ref{linearmainth} is divided into three parts. In the first part, we prove uniqueness and establish the solution representation formula (\ref{linearVrecover}). In the second part, we prove existence. In the third part, we consider the boundary behavior.

\subsubsection{Proof of uniqueness and of (\ref{linearVrecover})}
Let $V_0(x)$, $x \in [0, 1)$, and $V_1(y)$, $y \in [0,1)$ be real-valued functions satisfying (\ref{V0V1assumptions}) for some $n \geq 2$ and $\alpha \in [0, 1)$. Suppose that $V(x,y)$ is a $C^n$-solution of the Goursat problem for (\ref{linearernst}) in $D$ with data $\{V_0, V_1\}$. We will show that $V(x,y)$ can be expressed in terms of $V_0$ and $V_1$ by (\ref{linearVrecover}).

Equation (\ref{linearernst}) admits the Lax pair
\begin{align}\label{linearlax}  
\begin{cases}
\Phi_x(x,y,k) = \lambda V_x(x,y),	\\
\Phi_y(x,y,k) = \frac{1}{\lambda} V_y(x,y),
\end{cases}
\end{align}
where $\Phi(x,y,k)$ is an eigenfunction, $\lambda = \lambda(x,y,k)$ is defined by (\ref{lambdadef}), and $k$ is a complex spectral parameter. 
Indeed, using the relations
\begin{align*}
  \lambda_x = \frac{\lambda}{2(k-x)}
  = \frac{(1- \lambda^2)\lambda}{2(1-x-y)}, \qquad 
  \lambda_y = \frac{1}{2(k-x)\lambda}
  = \frac{(1- \lambda^2)}{2(1-x-y)\lambda},
\end{align*}
it is straightforward to check that the compatibility condition $\Phi_{xy} = \Phi_{yx}$ of (\ref{linearlax}) is equivalent to (\ref{linearernst}).

The occurrence of $\lambda$ in (\ref{linearlax}) implies that the spectral parameter is naturally considered as an element of the Riemann surface $\mathcal{S}_{(x,y)}$. Thus, we will henceforth view $\Phi(x,y,\cdot)$ as a function defined on $\mathcal{S}_{(x,y)}$ and write $\Phi(x,y,P)$ for the value of $\Phi$ at $P = (\lambda, k) \in \mathcal{S}_{(x,y)}$. We emphasize, however, that the partial derivatives $\Phi_x(x,y,P)$ and $\lambda_x(x,y,P)$ (resp. $\Phi_y(x,y,P)$ and $\lambda_y(x,y,P)$) are still computed with $(y,k)$ (resp. $(x,k)$) held fixed (and $\lambda$ allowed to change). 

The basic idea in what follows is to write (\ref{linearlax}) in the differential form $d\Phi = W$, where $W$ denotes the one-form $W =  \lambda V_x dx + \frac{1}{\lambda} V_y dy$, and then define a solution $\Phi$ of (\ref{linearlax}) by
\begin{align*}
  \Phi(x,y, k^\pm) = \int_{(0,0)}^{(x,y)} W(x',y',k^\pm), \qquad (x,y) \in D, \ k \in \hat{\C} \setminus [0,1].
\end{align*}
Since the one-form $W$ is closed, the integral on the right-hand side is independent of path. However, since $W$ in general is singular on the boundary of $D$, we need to be more careful when defining $\Phi$. We therefore choose to define $\Phi$ using the specific contour which consists of the horizontal segment from $(0,0)$ to $(x,0)$ followed by the vertical segment from $(x,0)$ to $(x,y)$ (see the left half of Figure \ref{Dcontoursfig}), that is, we define
\begin{align}\nonumber
&\Phi(x,y, k^\pm) = \int_0^x \lambda(x',0,k^\pm) V_x(x',0) dx' + \int_0^y \lambda(x,y',k^\pm)^{-1} V_y(x,y') dy',
	\\ \label{linearPhidef}
&\hspace{8cm} (x,y) \in D, \ k \in \hat{\C} \setminus [0,1].	
\end{align}
Since $x^\alpha V_x, y^\alpha V_y \in C(D)$, the integrals on the right-hand side of (\ref{linearPhidef}) are well-defined. The next lemma establishes several properties of $\Phi$. 
%In particular, $\Phi$ is a $C^n$-function of $(x,y) \in \Int D$.

\begin{figure}
\bigskip
\begin{center}
 \begin{overpic}[width=.37\textwidth]{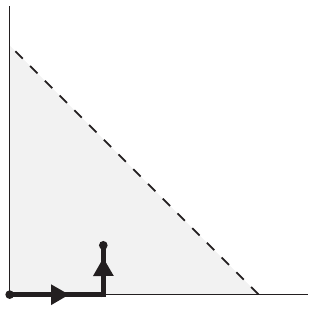}
 \put(103,3.5){\small $x$}
 \put(2,102){\small $y$}
 \put(82.5,-2){\small $1$}
 \put(-3,82){\small $1$}
 \put(-4,-3.5){\small $(0,0)$}
 \put(26,-3.5){\small $(x,0)$}
 \put(26,25){\small $(x,y)$}
 \end{overpic}
\hspace{2cm}  
\begin{overpic}[width=.37\textwidth]{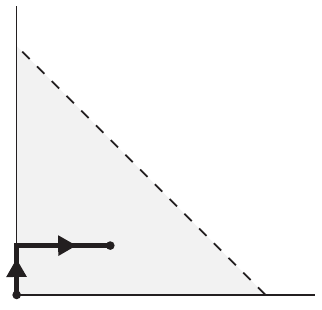}
 \put(103,3.5){\small $x$}
 \put(3,101){\small $y$}
 \put(82.5,-2.5){\small $1$}
 \put(-1,81){\small $1$}
 \put(-3,-3.5){\small $(0,0)$}
 \put(-12,20){\small $(0,y)$}
 \put(27,25){\small $(x,y)$}
 \end{overpic}
      \bigskip
   \begin{figuretext}\label{Dcontoursfig}
      The integration contours in (\ref{linearPhidef}) (left) and (\ref{linearPhidef2}) (right).
      \end{figuretext}
   \end{center}
\end{figure}

\begin{lemma}[Solution of Lax pair equations]\label{linearclaim1}
The function $\Phi(x,y,P)$ defined in (\ref{linearPhidef}) has the following properties:
\begin{enumerate}[$(a)$]

\item $\Phi$ can be alternatively expressed using the contour consisting of the vertical segment from $(0,0)$ to $(0,y)$ followed by the horizontal segment from $(0,y)$ to $(x,y)$ (see the right half of Figure \ref{Dcontoursfig}):
\begin{align}\nonumber
&\Phi(x,y, k^\pm) = \int_0^y \lambda(0,y',k^\pm)^{-1} V_y(0,y') dy' + \int_0^x \lambda(x',y,k^\pm) V_x(x',y) dx', 
	\\\label{linearPhidef2}
&\hspace{8cm} (x,y) \in D, \ k \in \hat{\C} \setminus [0,1].	
\end{align}

\item For each $k \in \hat{\C} \setminus [0,1]$, the function $(x,y) \mapsto \Phi(x,y,k^+)$ is continuous on $D$ and is $C^n$ on  $\Int D$.

\item For each $k \in \hat{\C} \setminus [0,1]$, the functions 
$$(x,y) \mapsto x^\alpha \Phi_x(x,y,k^+), \quad
(x,y) \mapsto y^\alpha \Phi_y(x,y,k^+), \quad
(x,y) \mapsto x^\alpha y^\alpha \Phi_{xy}(x,y,k^+),$$
are continuous on $D$.

   \item $\Phi$ obeys the symmetries
  \begin{align*}
\begin{cases} \Phi(x,y,k^+) = -\Phi(x, y, k^-), 
	\\
\Phi(x,y,k^\pm) = \overline{\Phi(x, y,\bar{k}^\pm)}, 
\end{cases} \qquad (x,y) \in D, \ k \in \hat{\C} \setminus [0,1].
\end{align*}

  \item For each  $(x,y) \in D$, $\Phi(x,y,P)$ extends continuously to an analytic function of $P \in \mathcal{S}_{(x,y)} \setminus \Sigma$, where $\Sigma = \Sigma_0 \cup \Sigma_1$ is the contour defined in (\ref{Sigma01def}).
    
  \item $\Phi(x,y,\infty^+) = V(x,y)$ for $(x,y) \in D$.

\end{enumerate}
\end{lemma}
\begin{proof}
Let $(x,y) \in D$. In order to prove $(a)$, we need to show that the expression
\begin{align}\nonumber
& \int_0^y \big[\lambda(0,y',k^\pm)^{-1} V_y(0,y') - \lambda(x,y',k^\pm)^{-1} V_y(x,y')\big] dy' 
	\\ \label{linearGreen1}
& +\int_0^x \big[\lambda(x',y,k^\pm) V_x(x',y) - \lambda(x',0,k^\pm) V_x(x',0)\big] dx'
\end{align}
vanishes. Since $x^\alpha V_x, y^\alpha V_y, x^\alpha y^\alpha V_{xy} \in C(D)$, the function $\lambda(\cdot,y',k^\pm)^{-1} V_y(\cdot,y')$ is absolutely continuous on the compact interval $[0,x]$ for each $y' \in (0, y]$. Similarly, the function $\lambda(x',\cdot,k^\pm) V_x(x',\cdot)$ is absolutely continuous on $[0,y]$ for each $x' \in (0, x]$.
Hence, we can write (\ref{linearGreen1}) as
\begin{align}\nonumber
& -\int_0^y \int_0^x \frac{\partial}{\partial x'} \big[\lambda(x',y',k^\pm)^{-1} V_y(x',y')\big] dx' dy' 
	\\ \label{linearGreen2}
& +\int_0^x \int_0^y \frac{\partial}{\partial y'}  \big[\lambda(x',y',k^\pm) V_x(x',y')\big] dy' dx'.
\end{align}
Since $V$ is a solution of (\ref{linearernst}), the Lax pair compatibility condition $(\lambda V_x)_y = (\lambda^{-1} V_y)_x$ is satisfied for $(x,y) \in \Int D$. The assumption $x^\alpha V_x, y^\alpha V_y, x^\alpha y^\alpha V_{xy} \in C(D)$ implies that $V_x, V_y, V_{xy} \in L^1(D_\delta)$ for each $\delta > 0$. Hence Fubini's theorem implies that the expression in (\ref{linearGreen2}) vanishes. This proves $(a)$.
Moreover, if $k \in \hat{\C} \setminus [0,1]$, then it follows from (\ref{linearPhidef}) and (\ref{linearPhidef2}) that $\Phi$ is a continuous function of $(x,y) \in D$ and a $C^n$-function of $(x,y) \in \Int D$, which proves $(b)$.

Let $k \in \hat{\C} \setminus [0,1]$. Then
\begin{align}\label{linearxalphaPhix}
x^\alpha \Phi_x(x,y,k^+) = x^\alpha \lambda(x,y,k^+) V_x(x,y).
\end{align}
The assumption $x^\alpha V_x \in C(D)$ implies that the right-hand side of (\ref{linearxalphaPhix}) is a continuous function of $(x,y) \in D$. 
Similarly, we see that $y^\alpha \Phi_y(x,y,k^+)$ and
$$x^\alpha y^\alpha \Phi_{xy}(x,y,k^+) = \frac{x^\alpha y^\alpha V_x(x,y)}{2(k-x) \lambda(x,y,k^+)} 
+  x^\alpha y^\alpha \lambda(x,y,k^+) V_{xy}(x,y)$$
are continuous functions of $(x,y) \in D$. This proves $(c)$.

The symmetries in $(d)$ are a consequence of the symmetries
\begin{align}\label{lambdasymm}
\lambda(x,y,k^+) = -\lambda(x,y,k^-),  \qquad \lambda(x,y,k^\pm) = \overline{\lambda(x,y,\bar{k}^\pm)},
\end{align}
and the definition (\ref{linearPhidef}) of $\Phi$.

To prove $(e)$, we note that $\lambda(x',0,k^+)$ is an analytic function of $k \in \hat{\C} \setminus [x',1]$ and $\lambda(x,y',k^+)^{-1}$ is an analytic function of $k \in \hat{\C} \setminus [x, 1-y']$. It follows that $\Phi(x,y, k^+)$ and $\Phi(x,y, k^-) = -\Phi(x,y, k^+)$ are analytic functions of $k \in \hat{\C} \setminus [0,1]$. 
Moreover, since
$$\lambda(x,y, (k+i0)^+) = \lambda(x,y,(k-i0)^-), \qquad k \in (x, 1-y),$$
we have
%$$W(x,y, (k+i0)^+) = W(x,y, (k-i0)^-), \qquad k \in (x, 1-y),$$
%and so
\begin{align*}
\Phi(x,y, (k + i0)^+) 
%= \int_{(0,0)}^{(x,y)} W(x',y', (k+i0)^+) = \int_{(0,0)}^{(x,y)} W(x',y', (k-i0)^-) \\
 = \Phi(x,y,(k-i0)^-), \qquad (x,y) \in D, \ k \in (x, 1-y).
\end{align*}
This shows that the values of $\Phi$ on the upper and lower sheets of $\mathcal{S}_{(x,y)}$ fit together across the branch cut; hence $\Phi$ extends to an analytic function of $P \in \mathcal{S}_{(x,y)} \setminus \Sigma$. This proves $(e)$.

To prove $(f)$, we note that $\lambda(x,y,\infty^+) = 1$ for all $(x,y) \in D$, which gives
%$$W(x,y,\infty^+) = V_x(x,y) dx + V_y(x,y) dy = dV(x,y).$$
\begin{align}\label{linearPhiVxVy}
\Phi(x,y, k^\pm) = \int_0^x V_{0x}(x') dx' + \int_0^y V_y(x,y') dy'.
\end{align}
Let $\delta > 0$. Since $V_{0x} \in L^1((1-\delta))$, $V_0$ belongs to the Sobolev space $W^{1,1}((0,1-\delta))$. Hence $V_0$ is absolutely continuous on $(0, 1-\delta)$. Using that $V_0\in C([0, 1))$, we see that $V_0$ is absolutely continuous on the compact interval $[0, 1-\delta]$. Hence,
\begin{align}\label{linearPhiVx}
\int_0^x V_{0x}(x') dx' = V(x,0)-V(0,0), \qquad x \in [0,1 - \delta).
\end{align}
Moreover, since $V_y \in L^1(D_\delta)$, we have $V_y(x, \cdot) \in L^1((0,1-x-\delta))$ for a.e. $x \in [0,1-\delta)$. Hence $V(x, \cdot) \in W^{1,1}((0,1-x-\delta))$ for a.e. $x \in [0,1-\delta)$. Since $V$ is also continuous on $D$, we conclude that $V(x, \cdot)$ is absolutely continuous on the compact interval $[0,1-x-\delta]$ for a.e. $x \in [0,1-\delta)$. 
Hence,
\begin{align}\label{linearPhiVy}
\int_0^y V_y(x,y') dy' = V(x,y)-V(x,0), \qquad (x,y) \in D_\delta.
\end{align}

Hence, substituting \eqref{linearPhiVx} and \eqref{linearPhiVy} into (\ref{linearPhiVxVy}) yields
$$\Phi(x,y, k^\pm) = V(x,0)-V(0,0) + V(x,y) - V(x,0).$$
Since $V(0,0) = 0$, part $(f)$ follows. 
\end{proof}

\begin{lemma}\label{linearclaim2}
For each $(x,y) \in D$, 
\begin{align}\label{linearphiminusphi}
P \mapsto \Phi(x,y,P) - \Phi(x,0, P) \quad \text{and} \quad P \mapsto \Phi(x,y,P) - \Phi(0,y,P)
\end{align}
extend continuously to analytic functions $\mathcal{S}_{(x,y)} \setminus \Sigma_1 \to \C$ and $\mathcal{S}_{(x,y)} \setminus \Sigma_0 \to \C$, respectively.
\end{lemma}

\begin{remark}\label{tilderemark}\upshape
The point $P$ in (\ref{linearphiminusphi}) belongs to $\mathcal{S}_{(x,y)}$ whereas the
maps $\Phi(x,0,\cdot)$ and $\Phi(0,y,\cdot)$ are defined on $\mathcal{S}_{(x,0)}$ and $\mathcal{S}_{(0,y)}$, respectively. The interpretation of equation (\ref{linearphiminusphi}) therefore deserves a comment of clarification:
If $(x,y)$ and $(\tilde{x}, \tilde{y})$ are two points in  $D$ and $F$ is a map from $\mathcal{S}_{(x,y)}$ to some space $X$, then $F$ naturally induces a map $\tilde{F}$ from $\mathcal{S}_{(\tilde{x}, \tilde{y})} \setminus \big([0,1]^+ \cup [0,1]^-\big)$ to $X$ according to $\tilde{F}(k^\pm) = F(k^\pm)$ for $k \in \hat{\C} \setminus [0,1].$
We sometimes, as in (\ref{linearphiminusphi}) (and also in (\ref{jumpdef})), identify these two maps and simply write $F$ for $\tilde{F}$.
%Strictly speaking, the point $F_{(x,y)}^{-1}(z)$ in (\ref{jumpdef}) belongs to $\mathcal{S}_{(x,y)}$ whereas the maps $\Phi_0(x,\cdot)$ and $\Phi_1(y, \cdot)$ are defined on $\mathcal{S}_{(x,0)}$ and $\mathcal{S}_{(0,y)}$, respectively. 
\end{remark}

\begin{proof}[Proof of Lemma \ref{linearclaim2}.]
Fix $(x,y) \in D$. 
Let $U$ be an open set in $\mathcal{S}_{(x,y)} \setminus \Sigma_0$. Then
\begin{align}\label{linearphiPmap}
\Phi(x,y, P) - \Phi(0,y,P) = \int_0^x \lambda(x',y,P) V_x(x',y) dx', \qquad P \in U,
\end{align}
where the values of $\Phi(0,y,P)$ and $\lambda(x',y,P)$ in (\ref{linearphiPmap}) are to be interpreted as in Remark \ref{tilderemark}. Since
$$P \mapsto \lambda(x',y,P) = \sqrt{\frac{k - (1-y)}{k - x'}}$$
defines an analytic map $U \to \C$ for each $x' \in [0,x]$, the map (\ref{linearphiPmap}) is also analytic for $P \in U$. This establishes the desired statement for the second map in (\ref{linearphiminusphi}); the proof for the first map is similar. 
\end{proof}

\begin{figure}
\bigskip\bigskip
\begin{center}
 \begin{overpic}[width=.55\textwidth]{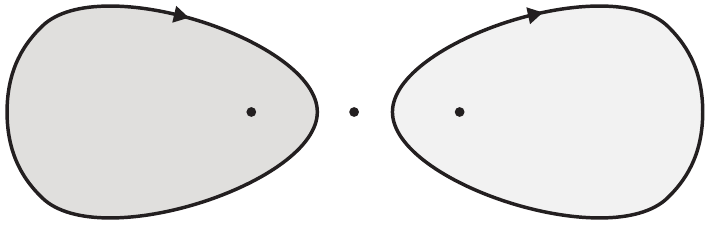}
 \put(24,32){\small $\Gamma_0$}
 \put(73,32){\small $\Gamma_1$}
 \put(63.8,10.5){\small $1$}
 \put(49,10.5){\small $0$}
 \put(32,10.5){\small $-1$}
 \put(13,15){\small $\Omega_0$}
 \put(81,15){\small $\Omega_1$}
 \put(47.5,34){\small $\Omega_\infty$}
 \end{overpic}
   \begin{figuretext}\label{Omegas.pdf}
      The domains $\Omega_0$, $\Omega_1$, and $\Omega_\infty$ in the complex $z$-plane.
      \end{figuretext}
   \end{center}
\end{figure}

Let $\Omega_0$, $\Omega_1$, and $\Omega_\infty$ denote the three open components of $\hat{\C} \setminus \Gamma$ chosen so that (see Figure \ref{Omegas.pdf})
\begin{align}\label{Omegadef}
-1 \in \Omega_0, \qquad 1 \in \Omega_1, \qquad \infty \in \Omega_\infty.
\end{align}

\begin{lemma}\label{linearclaim3}
The complex-valued function $m(x,y,z)$ defined by
\begin{align}\label{linearmVPhi}
m(x,y,z) = 
-V(x,y) + \Phi\big(x,y,F_{(x,y)}^{-1}(z)\big) - \begin{cases} \Phi\big(x,0,F_{(x,y)}^{-1}(z)\big), \;& z \in \Omega_0, \\
 \Phi\big(0,y,F_{(x,y)}^{-1}(z)\big), & z \in \Omega_1, \\
0, & z \in \Omega_\infty,
\end{cases} \; (x,y)\in D,
\end{align}
satisfies the RH problem (\ref{linearRHm}) and the relation (\ref{linearVrecover}) for each $(x,y) \in D$. 
\end{lemma}
\begin{proof}
Since $F_{(x,y)}$ is a biholomorphism $\mathcal{S}_{(x,y)} \to \hat{\C}$, we infer from Lemmas \ref{linearclaim1} and \ref{linearclaim2} that $m(x,y, \cdot)$ is analytic in $\hat{\C} \setminus \Gamma$ and $m(x,y,z) =O(z^{-1})$ as $z\to \infty$ for each $(x,y) \in D$. The jump condition in (\ref{linearRHm}) holds as a consequence of the definition (\ref{linearjumpdef}) of $v(x,y,z)$ and the fact that 
$$\Phi_0(x,k) = \Phi(x,0,k), \qquad \Phi_1(y,k) = \Phi(0,y,k).$$
Finally, since $0 \in \Omega_\infty$ and $F_{(x,y)}^{-1}(0) = \infty^-$, (\ref{linearmVPhi}) and Lemma \ref{linearclaim1} yield
$$m(x,y,0) = - V(x,y) + \Phi(x,y, \infty^-) = - 2V(x,y).$$
This proves (\ref{linearVrecover}).
\end{proof}

We have showed that if $V(x,y)$ is a $C^n$-solution of the Goursat problem for (\ref{linearernst}) in $D$ with data $\{V_0, V_1\}$, then $V(x,y)$ can be expressed in terms of $V_0$ and $V_1$ by (\ref{linearVrecover}). This also proves that the solution $V$ is unique if it exists, and completes the first part of the proof.

\subsubsection{Proof of existence}
The second part of the proof is devoted to proving existence. Let us therefore suppose that $V_0(x)$, $x \in [0, 1)$, and $V_1(y)$, $y \in [0,1)$ are real-valued functions satisfying (\ref{V0V1assumptions}) for some $n \geq 2$. We will construct a solution $V(x,y)$ of the associated Goursat problem as follows: Using the given data $V_0$ and $V_1$, we define $\Phi_0(x,P)$ and $\Phi_1(x,P)$ by (\ref{linearPhi0Phi1}). Then we define the jump matrix $v$ by (\ref{linearjumpdef}) and let $m(x,y,z)$ denote the unique solution of the RH problem (\ref{linearRHm}). Finally, we show that the function $V(x,y)$ defined in terms of $m(x,y,0)$ via  (\ref{linearVrecover}) constitutes a $C^n$-solution of the Goursat problem in $D$ with data $\{V_0, V_1\}$.
The proof proceeds through a series of lemmas.

\begin{lemma}[Solution of the $x$-part]\label{linearclaim1E}
The eigenfunction $\Phi_0(x,P)$ defined in (\ref{linearPhi0Phi1a}) has the following properties:
\begin{enumerate}[$(a)$]

\item For each $k \in \hat{\C} \setminus [0,1]$, the function $x \mapsto \Phi_0(x,k^+)$ is continuous on $[0,1)$ and is $C^n$ on  $(0,1)$.

   \item $\Phi_0$ obeys the symmetries
  \begin{align}\label{linearphi0symmetries}
\begin{cases} \Phi_0(x,k^+) = -\Phi_0(x, k^-), 
	\\
\Phi_0(x,k^\pm) = \overline{\Phi_0(x, \bar{k}^\pm)}, 
\end{cases} \qquad x \in [0,1), \ k \in \hat{\C} \setminus [0, 1].
\end{align}

  \item For each  $x \in [0, 1)$, $\Phi_0(x,P)$ extends continuously to an analytic function of $P \in \mathcal{S}_{(x,0)} \setminus \Sigma_0$.
    
  \item $\Phi_0(x,\infty^+) = V_0(x)$ for $x \in [0, 1)$.

  \item For each $x \in (0,1)$, $\Phi_{0x}(x,P)$ is an analytic function of $P \in \mathcal{S}_{(x,0)}$ except for a simple pole (at most) at the branch point $k = x$. 

  \item For each $x_0 \in (0,1)$  and each compact subset $K \subset \hat{\C} \setminus [0, x_0]$, 
  \begin{align}\label{linearxkphimap}
  x \mapsto \big(k \mapsto \Phi_0(x,k^+)\big)
  \end{align}
  is a continuous map $[0, x_0) \to L^\infty(K)$ and a $C^n$-map $(0, x_0) \to L^\infty(K)$. Moreover,
$x \mapsto x^\alpha \Phi_{0x}(x,k^+)$  and $x \mapsto \Phi_{0k}(x,k^+)$ are continuous maps $[0, x_0) \to L^\infty(K)$.
\end{enumerate}
\end{lemma}
\begin{proof}
If we note that $\Phi_0(x,P)$ is analytic at the points $1^\pm \in \mathcal{S}_{(x,0)}$ for each $x \in [0, 1)$, the properties $(a)$-$(d)$ follow immediately by setting $y = 0$ in Lemma \ref{linearclaim1}.
Moreover, since $\Phi_{0x}(x,k^\pm) = \lambda(x,0,k^\pm) V_{0x}(x)$ for $x \in (0,1)$, property $(e)$ follows from the definition of $\lambda$.

It remains to prove $(f)$. Fix $x_0 \in (0,1)$  and let $K$  be a compact subset $\hat{\C} \setminus [0, x_0]$.
The function $\lambda(x,0,\cdot)$ is bounded on $\mathcal{S}_{(x,0)}$ except for a simple pole at $k = x$.
Hence, for $x_1, x_2 \in [0, x_0)$, 
\begin{align*}
& \sup_{k \in K} \big|\Phi_0(x_2, k^+) - \Phi_0(x_1, k^+)\big|
= \sup_{k \in K} \bigg|\int_{x_1}^{x_2} \lambda(x,0,k^+) V_{0x}(x) dx\bigg|
	\\
& \leq \bigg( \sup_{k \in K} \sup_{x \in [0, x_0)} |\lambda(x,0,k^+)|\bigg) \int_{x_1}^{x_2} |V_{0x}(x)| dx
\leq C \int_{x_1}^{x_2} |V_{0x}(x)| dx,
\end{align*}
where the right-hand side tends to zero as $x_2 \to x_1$ because $V_{0x} \in L^1((0,x_0))$.
This shows that the map (\ref{linearxkphimap}) is continuous $[0, x_0) \to L^\infty(K)$.

If $x \in (0, x_0)$, then
\begin{align*}
  \sup_{k \in K} \bigg| &\frac{\Phi_0(x+h, k^+) - \Phi_0(x,k^+)}{h} - \Phi_{0x}(x, k)\bigg|
  	\\
  & \leq \sup_{k \in K} \bigg| \frac{1}{h} \int_x^{x+h} \lambda(x', 0, k) V_{0x}(x') dx' - \Phi_{0x}(x,k)\bigg|
  	\\
&  \leq \sup_{k \in K} \bigg| \lambda(\xi, 0, k) V_{0x}(\xi)  - \lambda(x, 0, k) V_{0x}(x)\bigg|,
\end{align*}
where $\xi$ lies between $x$ and $x+h$. As $h \to 0$, the right-hand side goes to zero. Hence (\ref{linearxkphimap}) is differentiable as a map $(0,x_0) \to L^\infty(K)$ and the derivative satisfies $\Phi_{0x}(x,k^+) = \lambda(x,0,k^+)V_{0x}(x)$. The same argument with $\lambda_k$ instead of $\lambda$ implies continuity of $x \mapsto \Phi_{0k}(x,k)$.

The map
$$x \mapsto \big(k \mapsto \lambda(x,0,k^+)\big)$$
is $C^\infty$ from $(0, x_0)$ to $L^\infty(K)$ and $V_{0x}$ is $C^{n-1}$ on $(0,1)$. Hence the map
$$x \mapsto \big(k \mapsto \lambda(x,0,k^+)V_{0x}(x)\big)$$
is $C^{n-1}$ from $(0, x_0)$ to $L^\infty(K)$.
It follows that (\ref{linearxkphimap}) is a $C^n$-map $(0, x_0) \to L^\infty(K)$. 
Moreover, equation (\ref{linearxalphaPhix}) evaluated at $y = 0$ implies $x \mapsto x^\alpha \Phi_{0x}(x,k^+)$ is continuous $[0, x_0) \to L^\infty(K)$.
This proves $(f)$ and completes the proof of the lemma.
\end{proof}

In the same way that we constructed the eigenfunction $\Phi_0(x,k)$ of the $x$-part, we can construct an eigenfunction $\Phi_1(y,k)$ of the $y$-part. 

\begin{lemma}[Solution of the $y$-part]\label{linearclaim2E}
The eigenfunction $\Phi_1(y,P)$ defined in (\ref{linearPhi0Phi1b}) has the following properties:
\begin{enumerate}[$(a)$]

\item For each $k \in \hat{\C} \setminus [0,1]$, the function $y \mapsto \Phi_1(y,k^+)$ is continuous on $[0,1)$ and is $C^n$ on  $(0,1)$. 

   \item $\Phi_1$ obeys the symmetries
  \begin{align}\label{linearphi1symmetries}
\begin{cases} \Phi_1(y,k^+) = -\Phi_1(y, k^-), 
	\\
\Phi_1(y,k^\pm) = \overline{\Phi_1(y, \bar{k}^\pm)}, 
\end{cases} \qquad y \in [0,1), \ k \in \hat{\C} \setminus [0,1].
\end{align}

  \item For each  $y \in [0, 1)$, $\Phi_1(y,P)$ extends continuously to an analytic function of $P \in \mathcal{S}_{(0,y)} \setminus \Sigma_1$.
    
  \item $\Phi_1(y,\infty^+) = V_1(y)$ for $y \in [0, 1)$.

  \item For each $y \in (0,1)$, $\Phi_{1y}(y,P)$ is an analytic function of $P \in \mathcal{S}_{(0,y)}$ except for a simple pole at the branch point $k = 1- y$. 

  \item For each $y_0 \in (0,1)$  and each compact subset $K \subset \hat{\C} \setminus [1-y_0, 1]$, 
  \begin{align}
  y \mapsto \big(k \mapsto \Phi_1(y,k^+)\big)
  \end{align}
  is a continuous map $[0, y_0] \to L^\infty(K)$ and a $C^n$-map $(0, y_0) \to L^\infty(K)$. Moreover,
  $y\mapsto y^\alpha \Phi_{1y}(y,k^+)$ and $xy\mapsto \Phi_{1k}(y,k^+)$ are continuous maps $[0, x_0) \to L^\infty(K)$.
\end{enumerate}
\end{lemma}
\begin{proof}
The proof is analogous to that of Lemma \ref{linearclaim1E}.
\end{proof}

Recall from the definition in Section \ref{notationsec} that the contour $\Gamma \equiv \Gamma(x,y)$ consists of two nonintersecting clockwise loops $\Gamma_0$ and $\Gamma_1$ which encircle the intervals $F_{(x,y)}(\Sigma_0)$ and $F_{(x,y)}(\Sigma_1)$ respectively, but which do not encircle the origin. We are free to choose $\Gamma_0$ and $\Gamma_1$ as long as these requirements are met. It turns out to be convenient to choose $\Gamma_0$ and $\Gamma_1$  independent of $(x,y)$. However, we see from (\ref{twointervals}) that the intervals $F_{(x,y)}(\Sigma_0)$ and $F_{(x,y)}(\Sigma_1)$  get arbitrarily close to the origin as  $(x,y)$ approaches the diagonal edge $x+y=1$ of $D$. Hence we cannot take $\Gamma$ independent of $(x,y)$  for all  $(x,y) \in D$. However, if we restrict ourselves to points $(x,y)$ which lie in the slightly smaller triangle $D_\delta$, $\delta > 0$, defined in (\ref{Ddeltadef}), then we can choose $\Gamma$ independent of $(x,y)$.

\begin{figure}
\bigskip\bigskip
\begin{center}
 \begin{overpic}[width=.7\textwidth]{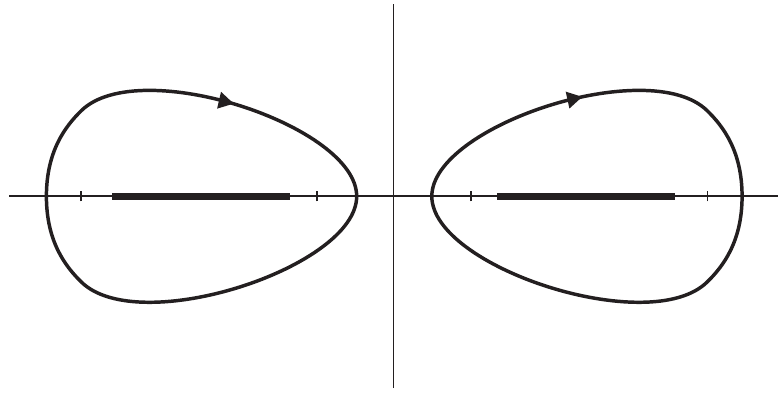}
 \put(102,24.5){\small $\re z$}
 \put(47,53){\small $\im z$}
 \put(28,40){\small $\Gamma_0$}
 \put(71,40.5){\small $\Gamma_1$}
 \put(59.5,22){\small $\epsilon$}
 \put(88.5,21.5){\small $\epsilon^{-1}$}
 \put(37.5,22){\small $-\epsilon$}
 \put(7,21.5){\small $-\epsilon^{-1}$}
  \put(18,21){\small $F_{(x,y)}(\Sigma_0)$}
 \put(68,21){\small $F_{(x,y)}(\Sigma_1)$}
 \end{overpic}
   \bigskip\bigskip\bigskip
   \begin{figuretext}\label{Gamma01.pdf}
      We choose the loops $\Gamma_0$ and $\Gamma_1$ in the complex $z$-plane so that they encircle the intervals $[-\epsilon^{-1}, -\epsilon]$ and $[\epsilon, \epsilon^{-1}]$, respectively.
       \end{figuretext}
   \end{center}
\end{figure}

Thus, fix $\delta \in (0,1)$ and choose $\epsilon > 0$ so small that $F_{(x,y)}(\Sigma_0)$ and $F_{(x,y)}(\Sigma_1)$ are contained in the intervals $[-\epsilon^{-1}, -\epsilon]$ and $[\epsilon, \epsilon^{-1}]$, respectively, for all $(x,y) \in D_\delta$. Fix two smooth nonintersecting clockwise contours $\Gamma_0$ and $\Gamma_1$ in the complex $z$-plane which encircle once the intervals $[-\epsilon^{-1}, -\epsilon]$ and $[\epsilon, \epsilon^{-1}]$, respectively, but which do not encircle zero, see Figure \ref{Gamma01.pdf}. Suppose also that $\Gamma_0$ and $\Gamma_1$ are invariant under the involutions $z \mapsto z^{-1}$ and $z \mapsto \bar{z}$. 
Let $\Gamma = \Gamma_0 \cup \Gamma_1$ and, using this particular choice of $\Gamma$, define $V(x,y)$ for $(x,y) \in D_\delta$ by (\ref{linearV}), i.e.,
\begin{align}\label{linearV2}
  V(x,y) = -\frac{1}{4\pi i} \int_{\Gamma} \frac{v(x,y,z)}{z} dz,
\end{align}
where $v(x,y,z)$ is given by (\ref{linearjumpdef}).
We will show that
\begin{align}\label{linearVDdelta}
\begin{cases}
  V \in C(D_\delta) \cap C^n(\Int D_\delta), 
  	\\
   \text{$V(x,y)$ satisfies the Euler-Darboux equation (\ref{linearernst}) in $\Int(D_\delta)$,}
  	\\
    \text{$x^\alpha V_x, y^\alpha V_y, x^\alpha y^\alpha V_{xy} \in C(D_\delta)$ for some $\alpha \in [0,1)$,}
  	\\
  \text{$V(x,0) = V_0(x)$ for $x \in [0,1-\delta)$,}
 	\\
  \text{$V(0,y) = V_1(y)$ for $y \in [0,1-\delta)$.}
\end{cases}
\end{align}
Since $\delta > 0$ can be chosen arbitrarily small, this will complete the proof of the theorem. 

Consider the family of scalar RH problems given in (\ref{linearRHm}) parametrized by the two parameters $(x,y) \in D_\delta$. For each $(x,y) \in D_\delta$, the unique solution of (\ref{linearRHm}) is given by
\begin{align}\label{linearmsolution}
m(x,y,z) = \frac{1}{2\pi i} \int_{\Gamma} \frac{v(x,y,z')}{z'-z} dz', \qquad (x,y) \in D_\delta, \ z \in \hat{\C} \setminus \Gamma.
\end{align}

\begin{lemma}\label{linearclaim3E}
The map $(x,y) \mapsto v(x,y, \cdot)$ is continuous from $D_\delta$ to $L^\infty(\Gamma)$ and $C^n$ from $\Int D_\delta$ to $L^\infty(\Gamma)$.
Moreover, the three maps 
\begin{align}\label{linearvthreemaps}
(x,y) \mapsto x^\alpha v_x(x,y, \cdot), \qquad
(x,y) \mapsto y^\alpha v_x(x,y, \cdot), \qquad
(x,y) \mapsto x^\alpha y^\alpha v_{xy}(x,y, \cdot),
\end{align}
are continuous from $D_\delta$ to $L^\infty(\Gamma)$.
\end{lemma}
\begin{proof}
The map $(x,y) \mapsto v(x,y, \cdot)$ is continuous from $D_\delta$ to $L^\infty(\Gamma)$ and $C^n$ from $\Int D_\delta$ to $L^\infty(\Gamma)$ as a consequence of part $(f)$ of Lemmas \ref{linearclaim1E} and \ref{linearclaim2E}.
Furthermore,
$$x^\alpha v_x(x,y, z) = \begin{cases} x^\alpha \Phi_{0x}\big(x, F_{(x,y)}^{-1}(z)\big) + x^\alpha \Phi_{0k}\big(x, F_{(x,y)}^{-1}(z)\big) \Big(\frac{d}{dx}F_{(x,y)}^{-1}(z)\Big), \quad & z \in \Gamma_0,
	\\
x^\alpha \Phi_{1k}\big(y, F_{(x,y)}^{-1}(z)\big) \Big(\frac{d}{dx}F_{(x,y)}^{-1}(z)\Big), \quad & z \in \Gamma_1.
\end{cases}
$$
Part $(f)$ of Lemma \ref{linearclaim1E} implies that the terms $x^\alpha \Phi_{0x}\big(x, F_{(x,y)}^{-1}(\cdot)\big)$ and $\Phi_{0k}\big(x, F_{(x,y)}^{-1}(\cdot)\big)$ are continuous $D_\delta \to L^\infty(\Gamma_0))$.
Similarly, part $(f)$ of Lemma \ref{linearclaim2E} implies that the term $\Phi_{1k}\big(y, F_{(x,y)}^{-1}(\cdot)\big)$ is continuous $D_\delta \to L^\infty(\Gamma_1))$.
We conclude that $(x,y) \mapsto x^\alpha v_x(x,y, \cdot)$ is continuous $D_\delta \to L^\infty(\Gamma)$. The other two maps in (\ref{linearvthreemaps}) are treated in a similar way. 
\end{proof}

\begin{lemma}\label{linearclaim4E}
The solution $m(x,y,z)$ defined in (\ref{linearmsolution}) has the following properties:
\begin{enumerate}[$(a)$]
   \item For each point $(x,y) \in D_\delta$, $m(x,y,\cdot)$ obeys the symmetries
  \begin{align}\label{linearmsymm}
 m(x,y,z) = m(x,y,0) - m(x,y,z^{-1}) = \overline{m(x,y,\bar{z})}, \qquad z \in \hat{\C} \setminus \Gamma.
\end{align}

  \item For each  $z \in \hat{\C}\setminus \Gamma$, the map $(x,y) \mapsto m(x,y,z)$ is continuous from $D_\delta$ to $\C$ and is $C^n$ from $\Int D_\delta$ to $\C$.

  \item For each  $z \in \hat{\C}\setminus \Gamma$, the three maps 
  $$(x,y) \mapsto x^\alpha m_x(x,y, z), \qquad
(x,y) \mapsto y^\alpha m_x(x,y, z), \qquad
(x,y) \mapsto x^\alpha y^\alpha m_{xy}(x,y, z),$$
are continuous from $D_\delta$ to $\C$.

\end{enumerate}
\end{lemma}
\begin{proof}
The symmetries in (\ref{linearphi0symmetries}) and (\ref{linearphi1symmetries}) show that $v$ satisfies
\begin{align}\label{linearvsymmetries}
\begin{cases} v(x,y,z) = -v(x, y, z^{-1}), 
	\\
v(x,y,z) = \overline{v(x, y, \bar{z})}, 
\end{cases} \qquad z \in \Gamma, \ (x,y) \in D_\delta.
\end{align}
These symmetries imply that $m(x,y,0) - m(x,y,z^{-1})$ and $\overline{m(x,y,\bar{z})}$ satisfy the same RH problem as $m(x,y,z)$. Hence, by uniqueness, (\ref{linearmsymm}) holds. This proves $(a)$. 

For each $z \in \hat{\C}\setminus \Gamma$, the map 
$$f \mapsto \int_\Gamma \frac{f(z')}{z' - z} dz'$$
is a bounded linear map $L^\infty(\Gamma) \to \C$.
Hence properties $(b)$ and $(c)$ follow immediately from (\ref{linearmsolution}) and Lemma \ref{linearclaim3E}.
\end{proof}

Given a contour $\gamma \subset \C$, we use the notation $N(\gamma)$ to denote an open tubular neighborhood of $\gamma$. We extend the definition (\ref{linearjumpdef}) of $v$ to a tubular neighborhood $N(\Gamma) = N(\Gamma_0) \cup N(\Gamma_1)$ of $\Gamma$ as follows, see Figure  \ref{Gammatubular.pdf}:
\begin{align}\label{linearjumpdef2}
v(x,y, z) = \begin{cases}
 \Phi_0\big(x, F_{(x,y)}^{-1}(z)\big), \quad & z \in N(\Gamma_0), 
 	\\
\Phi_1\big(y, F_{(x,y)}^{-1}(z)\big), & z \in N(\Gamma_1),
\end{cases}\qquad (x,y) \in D_\delta.
\end{align}
We choose $N(\Gamma)$ so narrow that it does not intersect the intervals $[-\epsilon^{-1}, -\epsilon]$ and $[\epsilon, \epsilon^{-1}]$. Then, for each $(x,y) \in D_\delta$, $v(x,y,\cdot)$ is an analytic function of $z \in N(\Gamma)$. 
Using the notation $z(x,y,P) := F_{(x,y)}(P)$, we can write (\ref{linearjumpdef2}) as
\begin{align}\label{linearvzPhi0}
  v(x,y,z(x,y,P)) = \begin{cases} \Phi_0(x, P), \quad & P \in F_{(x,y)}^{-1}\big(N(\Gamma_0)\big), \\
 \Phi_1(y, P), \quad & P \in F_{(x,y)}^{-1}\big(N(\Gamma_1)\big),  
 \end{cases}
  \quad (x,y) \in D_\delta.
\end{align}
We define functions $f_0(x,y,z)$ and $f_1(x,y,z)$ for $(x,y) \in D_\delta$ by
\begin{align*}
f_0(x,y,z) = m_x(x,y,z) + z_x\big(x,y, F_{(x,y)}^{-1}(z)\big) m_z(x,y,z), \qquad z \in \hat{\C} \setminus \Gamma,
	\\
f_1(x,y,z) = m_y(x,y,z) + z_y\big(x,y, F_{(x,y)}^{-1}(z)\big) m_z(x,y,z), \qquad z \in \hat{\C} \setminus \Gamma.
\end{align*}
Moreover, we let $n_0(x,y,z)$ and $n_1(x,y,z)$ denote the functions given by
\begin{subequations}\label{linearndef}
\begin{align}\label{linearndefa}
n_0(x,y,z) = \begin{cases}
  f_0(x,y,z) + \Phi_{0x}\big(x,F_{(x,y)}^{-1}(z)\big), \quad &  z \in \Omega_0,	\\
  f_0(x,y,z), & z \in \Omega_1 \cup \Omega_\infty,
\end{cases}
\end{align}
and
\begin{align}\label{linearndefb}
n_1(x,y,z) = \begin{cases}
  f_1(x,y,z) + \Phi_{1y}\big(y,F_{(x,y)}^{-1}(z)\big), \quad &  z \in \Omega_1,	\\
  f_1(x,y,z), & z \in \Omega_0 \cup \Omega_\infty.
\end{cases}
\end{align}
\end{subequations}

\begin{figure}
\bigskip\bigskip
\begin{center}
 \begin{overpic}[width=.7\textwidth]{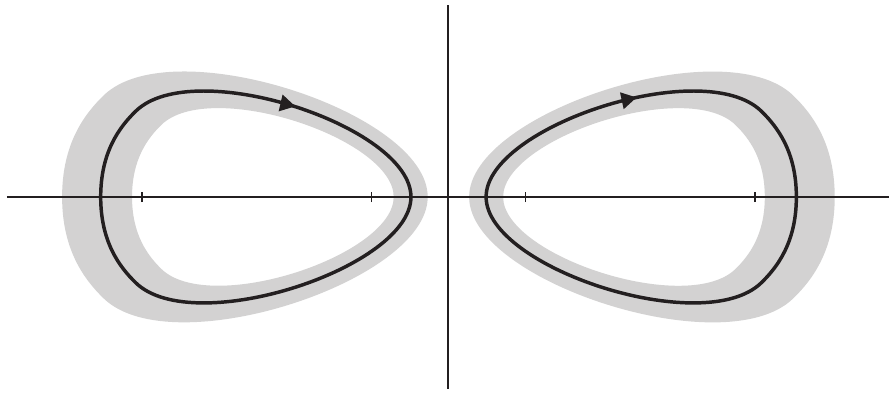}
 \put(101,21.2){\small $\re z$}
 \put(47,46){\small $\im z$}
 \put(11,38){\small $N(\Gamma_0)$}
 \put(80,38){\small $N(\Gamma_1)$}
  \put(58,19){\small $\epsilon$}
 \put(82,18.5){\small $\epsilon^{-1}$}
 \put(38,19){\small $-\epsilon$}
 \put(12,18.5){\small $-\epsilon^{-1}$}
 \end{overpic}
   \begin{figuretext}\label{Gammatubular.pdf}
      The tubular neighborhood $N(\Gamma) = N(\Gamma_0) \cup N(\Gamma_1)$ of the contour $\Gamma$ in the complex $z$-plane.
      \end{figuretext}
   \end{center}
\end{figure}

\begin{lemma}\label{linearclaim5E}
For each $(x,y) \in \Int D_\delta$, it holds that
\begin{enumerate}[$(a)$]
\item $n_0(x,y,z)$ is an analytic function of $z \in \hat{\C} \setminus \{-1\}$ and has at most a simple pole at $z = -1$. 
\item $n_1(x,y,z)$ is an analytic function of $z \in \hat{\C} \setminus \{1\}$ and has at most a simple pole at $z = 1$. 
\item $n_0(x,y,\infty) = 0$ and $n_0(x,y,0) = -2V_x(x,y)$.
\item $n_1(x,y,\infty) = 0$ and $n_1(x,y,0) = -2V_y(x,y)$.
\end{enumerate}
\end{lemma}
\begin{proof}
Let $(x,y) \in \Int D_\delta$.
The function
\begin{align}\label{linearzx}
z_x\big(x,y, F_{(x,y)}^{-1}(z)\big) = -\frac{1-z}{1+z} \frac{z}{1 - x - y}
\end{align}
is analytic for $z \in \hat{\C}\setminus \{-1, \infty\}$ with simple poles at $z = -1$ and $z = \infty$. 
Equation (\ref{linearmsolution}) implies that $m_z(x,y, z) = O(z^{-2})$ and $m_x(x,y, z) = O(z^{-1})$ as $z \to \infty$. Hence $f_0(x,y,z)$ is analytic at $z = \infty$. 
It follows that $f_0(x,y,z)$ is analytic for all $z \in \hat{\C}\setminus (\Gamma \cup \{-1\})$ with a simple pole at $z = -1$ at most.
Now $f_0$ has continuous boundary values on $\Gamma$ and satisfies the following jump condition across $\Gamma$:
\begin{align}\label{linearf0jump}
f_{0+}(x,y,z) = f_{0-}(x,y,z) + v_x(x,y,z) + z_x\big(x,y, F_{(x,y)}^{-1}(z)\big) v_z(x,y,z), \qquad z \in \Gamma.
\end{align}
Differentiating (\ref{linearvzPhi0}) with respect to $x$ and $y$ and evaluating the resulting equations at $k = F_{(x,y)}^{-1}(z)$, we find, for $(x,y) \in \Int D_\delta$,
\begin{align}\label{linearvxzxvza}
\begin{cases}
v_x(x,y,z) + z_x\big(x,y, F_{(x,y)}^{-1}(z)\big) v_z(x,y,z) = \Phi_{0x}(x,F_{(x,y)}^{-1}(z)),
	\\
v_y(x,y,z) + z_y\big(x,y, F_{(x,y)}^{-1}(z)\big) v_z(x,y,z) = 0,
\end{cases} \quad z \in N(\Gamma_0),
\end{align}	
and
\begin{align}\label{linearvxzxvzb}
\begin{cases}
v_x(x,y,z) + z_x\big(x,y, F_{(x,y)}^{-1}(z)\big) v_z(x,y,z) = 0,
	\\
v_y(x,y,z) + z_y\big(x,y, F_{(x,y)}^{-1}(z)\big) v_z(x,y,z) = \Phi_{1y}(x,F_{(x,y)}^{-1}(z)),
\end{cases} \quad z \in N(\Gamma_1).
\end{align}	
Using the first equations in (\ref{linearvxzxvza}) and (\ref{linearvxzxvzb}) in (\ref{linearf0jump}), we conclude that $f_0$ is analytic across $\Gamma_1$ and has the following jump across $\Gamma_0$:
\begin{align}\label{linearf0jump2}
f_{0+}(x,y,z) = f_{0-}(x,y,z) + \Phi_{0x}(x,F_{(x,y)}^{-1}(z)), \qquad z \in \Gamma_0.
\end{align}
Consequently, $n_0$ is analytic across $\Gamma$. Furthermore, by Lemma \ref{linearclaim1E}, $\Phi_{0x}(x,F_{(x,y)}^{-1}(z))$ is analytic for $z \in \hat{\C}\setminus \{-1\}$ with at most a simple pole at $z = -1$. It follows that $n_0$ satisfies $(a)$. The proof of $(b)$ is similar and relies on the second equations in (\ref{linearvxzxvza}) and (\ref{linearvxzxvzb}).

Using (\ref{linearzx}) in the definition (\ref{linearndefa}) of $n_0$, we can write
\begin{align}\label{linearn0f0}
n_0(x,y,z) = f_0(x,y,z) = m_x(x,y,z) -\frac{1-z}{1+z} \frac{z}{1 - x - y} m_z(x,y,z), \qquad z \in \Omega_\infty.
\end{align}
Since $m_z(x,y, z) = O(z^{-2})$ and $m_x(x,y, z) = O(z^{-1})$  as $z \to \infty$, this gives $n_0(x,y,\infty) = 0$. On the other hand, evaluating (\ref{linearn0f0}) at $z = 0$, we find $n_0(x,y,0) = m_x(x,y,0) = -2V_x(x,y)$. This proves $(c)$; the proof of $(d)$ is analogous. 
\end{proof}

Equation (\ref{linearmVPhi}) suggests that we define a function $\Phi(x,y,P)$ for $(x,y) \in D_\delta$ and $P \in F_{(x,y)}^{-1}(\Omega_\infty) \subset \mathcal{S}_{(x,y)}$ by
\begin{align}\label{linearphidef}
\Phi(x,y,P) = V(x,y) + m(x,y,F_{(x,y)}(P)).
\end{align}

\begin{lemma}\label{linearclaim6E}
The function $\Phi$  defined in (\ref{linearphidef}) satisfies the Lax pair equations
  \begin{align}\label{linearphilax}
\begin{cases}
\Phi_x(x,y,P) = \lambda(x,y,P)  V_x(x,y), \\
\Phi_y(x,y,P) =  \frac{1}{\lambda(x,y,P)} V_y(x,y),
\end{cases}
\end{align}
for $(x,y) \in \Int D_\delta$ and $P \in F_{(x,y)}^{-1}(\Omega_\infty)$.
\end{lemma}
\begin{proof}
The analyticity structure of $n_0$ established in Lemma \ref{linearclaim5E} implies that there exists a function  $C(x,y)$ independent of  $z$ such that
\begin{align}\label{linearn0C}
n_0(x,y,z) = \frac{C(x,y)}{z+1}, \qquad z \in \hat{\C}.
\end{align}
We determine $C(x,y)$ by evaluating (\ref{linearn0C}) at $z = 0$. By Lemma \ref{linearclaim5E} $(d)$, this gives $C(x,y) = -2V_x(x,y)$.
It follows that
\begin{align}\label{linearnoutside}
  n_0 = -\frac{2V_x(x,y)}{z+1}, \qquad (x,y) \in D_\delta, \  z \in \hat{\C}.
\end{align}
Note that we did not exclude that $n_0$ is free of singularities. In this case we have $C=-2V_x=0$ by Lemma \ref{linearclaim5E}.

Differentiating (\ref{linearphidef}) with respect to $x$ and using (\ref{linearn0f0}) and (\ref{linearnoutside}), we find, for $P \in F_{(x,y)}^{-1}(\Omega_\infty)$,
\begin{align*}
\Phi_x(x,y,P)
& = V_x(x,y) + f_0(x,y,z(x,y,P))
= V_x(x,y) - \frac{2V_x(x,y)}{z(x,y,P) +1}.
\end{align*}
Since
$$1 - \frac{2}{z+1} = \lambda,$$
this yields the first equation in (\ref{linearphilax}). A similar argument gives the second equation in (\ref{linearphilax}). This proves the lemma.
\end{proof}

\begin{lemma}\label{linearclaim7E}
The real-valued function $V:D \to \R$ defined by (\ref{linearV2}) has the properties listed in (\ref{linearVDdelta}).
\end{lemma}
\begin{proof}
The function $V(x,y) = -\frac{1}{2}m(x,y,0)$ is real-valued by (\ref{linearmsymm}).
Moreover, by part $(b)$ of Lemma \ref{linearclaim4E}, the map $(x,y) \mapsto m(x,y,0)$ is continuous from $D_\delta$ to $\C$ and is $C^n$ from $\Int D_\delta$ to $\C$. Hence $V \in C(D_\delta) \cap C^n(\Int D_\delta)$. Similarly, part $(c)$ of Lemma \ref{linearclaim4E} implies that $x^\alpha V_x, y^\alpha V_y, x^\alpha y^\alpha V_{xy} \in C(D_\delta)$.

Let $P = (\lambda, k)$ be a point in $F_{(x,y)}^{-1}(\Omega_\infty) \subset \mathcal{S}_{(x,y)}$. 
For  each  fixed  $k \in \hat{\C}$ with $k^+ \in F_{(x,y)}^{-1}(\Omega_\infty)$, the map $(x,y) \to \Phi(x,y,k^+)$ is $C^n$ from $\Int D_\delta$ to $\C$.
By Lemma \ref{linearclaim6E}, it satisfies the Lax pair equations (\ref{linearphilax}). Since  $n \geq 2$, it follows that
\begin{align*}
0 &= \Phi_{xy}(x,y,P) - \Phi_{yx} (x,y,P)
	\\
& = \lambda_y  V_x + \lambda V_{xy} 
+ \frac{\lambda_x}{\lambda^2} V_y - \frac{1}{\lambda} V_{xy} 
	\\
& = \frac{1}{2\lambda(k-x)}  (V_x + V_y) + \bigg(\lambda  - \frac{1}{\lambda}\bigg)V_{xy}
	\\
& = \frac{1}{2\lambda(k-x)} \big(V_x + V_y - 2(1-x-y)V_{xy}\big), \qquad (x,y) \in \Int D_\delta.
\end{align*}
It follows that $V(x,y)$ satisfies Euler-Darboux equation (\ref{linearernst}) for $(x,y) \in \Int D_\delta$. 

Finally, we show that $V(x,0) = V_0(x)$ for $x \in [0, 1-\delta)$; the proof that $V(0,y) = V_1(y)$ for $y \in [0,1-\delta)$ is similar.
By definitions (\ref{linearV2}) and (\ref{linearjumpdef}) of $V$ and $v$, we have
$$V(x,0) = -\frac{1}{4\pi i} \int_{\Gamma} \frac{v(x,0,z)}{z} dz
= -\frac{1}{4\pi i} \int_{\Gamma_0} \frac{\Phi_0(x,F_{(x,0)}^{-1}(z))}{z} dz, \qquad x \in [0, 1-\delta).$$
But $\Phi_0(x,F_{(x,0)}^{-1}(z))$ is analytic for $z \in \hat{\C} \setminus [-\epsilon^{-1}, -\epsilon]$ by Lemma \ref{linearclaim1E}, so using Cauchy's formula to compute the contributions from  $z = 0$ and $z = \infty$, we find
\begin{align*}
V(x,0) & = - \frac{1}{2}\Phi_0\big(x,F_{(x,0)}^{-1}(0)\big) + \frac{1}{2}\Phi_0\big(x,F_{(x,0)}^{-1}(\infty)\big)
	\\
& = -\frac{1}{2}\Phi_0(x, \infty^-) + \frac{1}{2}\Phi_0(x, \infty^+)
 = \Phi_0(x, \infty^+) = V_0(x), \qquad x \in [0, 1-\delta).
\end{align*}
This completes the proof of the lemma. Since $\delta > 0$ was arbitrary, it also completes the proof of existence. 
\end{proof}

\subsubsection{Proof of boundary behavior}
Let $V_0(x)$, $x \in [0, 1)$, and $V_1(y)$, $y \in [0,1)$ be real-valued functions satisfying (\ref{V0V1assumptions}) for some $n \geq 2$ and some $\alpha \in (0, 1)$. Suppose $V(x,y)$ is a $C^n$-solution of the Goursat problem for (\ref{linearernst}) in $D$ with data $\{V_0, V_1\}$ and define $m_1, m_2 \in \R$ by (\ref{linearm1m2def}).
By (\ref{linearV}), we have
\begin{align*}
V(x,y) = -\frac{1}{4\pi i} \int_{\Gamma_0} \frac{\Phi_0\big(x, F_{(x,y)}^{-1}(z)\big)}{z} dz
-\frac{1}{4\pi i} \int_{\Gamma_1} \frac{\Phi_1\big(y, F_{(x,y)}^{-1}(z)\big)}{z} dz.
\end{align*}
Hence
\begin{align}\nonumber
V_x(x,y) = & -\frac{1}{4\pi i} \int_{\Gamma_0} \frac{\Phi_{0x}\big(x, F_{(x,y)}^{-1}(z)\big)}{z} dz
-\frac{1}{4\pi i} \int_{\Gamma_0} \frac{\Phi_{0k}\big(x, F_{(x,y)}^{-1}(z)\big)}{z} \Big(\frac{d}{dx}F_{(x,y)}^{-1}(z)\Big)dz
	\\ \label{Vxxy}
& -\frac{1}{4\pi i} \int_{\Gamma_1} \frac{\Phi_{1k}\big(y, F_{(x,y)}^{-1}(z)\big)}{z} \Big(\frac{d}{dx}F_{(x,y)}^{-1}(z)\Big) dz.
\end{align}
Now
$$k = F_{(x,y)}^{-1}(z) = -\frac{x (z-1)^2+(y-1) (z+1)^2}{4 z},$$
so
\begin{align}\label{Finversexyderivatives}
\frac{d}{dx}F_{(x,y)}^{-1}(z) = -\frac{(z-1)^2}{4 z}, \qquad
\frac{d}{dy}F_{(x,y)}^{-1}(z) = -\frac{(z+1)^2}{4 z}.
\end{align}
It follows from Lemma \ref{linearclaim1E} and Lemma \ref{linearclaim2E} that the last two integrals on the right-hand side of (\ref{Vxxy}) remain bounded as $x \downarrow 0$. 
Moreover,
\begin{align*}
\lim_{x \downarrow 0} x^\alpha \Phi_{0x}\big(x, F_{(x,y)}^{-1}(z)\big)
& = \lim_{x \downarrow 0} x^\alpha \lambda(x,0, F_{(x,y)}^{-1}(z)) V_{0x}(x)
= m_1 \lambda(0,0, F_{(0,y)}^{-1}(z)).
\end{align*}
Using that $F_{(0,y)}^{-1}(z)= -\frac{(y-1) (z+1)^2}{4 z}$, we find
\begin{align}\label{lambda00F0yinv}
\lambda(0,0, F_{(0,y)}^{-1}(z)) %= \sqrt{\frac{k-1}{k}}
= \sqrt{\frac{1}{(z+1)^2}\left(z-\frac{1-\sqrt{y}}{1+\sqrt{y}}\right)
   \left(z-\frac{1+\sqrt{y}}{1-\sqrt{y}}\right)},
\end{align}
where the square roots have positive (negative) real part for $|z| > 1$ ($|z| < 1$). Thus
\begin{align*}\nonumber
\lim_{x \downarrow 0} x^\alpha \Phi_{0x}\big(x, F_{(x,y)}^{-1}(z)\big)
 =\frac{-m_1}{z+1} \sqrt{\left(z-\frac{1-\sqrt{y}}{1+\sqrt{y}}\right)
   \left(z-\frac{1+\sqrt{y}}{1-\sqrt{y}}\right)},
 \end{align*}
where the square root has a branch cut along the interval $[\frac{1-\sqrt{y}}{1+\sqrt{y}}, \frac{1+\sqrt{y}}{1-\sqrt{y}}]$ and the branch is fixed so that the root has positive real part for $z < 0$.
Hence
\begin{align}\nonumber
\lim_{x \downarrow 0} x^\alpha V_x(x,y) = &\; \frac{m_1}{4\pi i} \int_{\Gamma_0} \frac{1}{z+1} \sqrt{\left(z-\frac{1-\sqrt{y}}{1+\sqrt{y}}\right) \left(z-\frac{1+\sqrt{y}}{1-\sqrt{y}}\right)}\frac{dz}{z}
	\\\nonumber
= & -\frac{m_1}{2} \underset{z = -1}{\res} \frac{1}{z+1} \sqrt{\left(z-\frac{1-\sqrt{y}}{1+\sqrt{y}}\right) \left(z-\frac{1+\sqrt{y}}{1-\sqrt{y}}\right)}\frac{1}{z}
	\\\nonumber
= & -\frac{m_1}{2} \frac{-2}{\sqrt{1-y}} = \frac{m_1}{\sqrt{1-y}}.
\end{align}
This proves (\ref{linearboundarylimita}); the proof of (\ref{linearboundarylimitb}) is similar.
Thus the proof of Theorem \ref{linearmainth} is complete.

\section{Lax pair and eigenfunctions}\label{ernstsec}
In this section we introduce a Lax pair for (\ref{ernst}) and define appropriate eigenfunctions in preparation for the proofs of Theorems  \ref{mainth1}-\ref{mainth4}. 

\subsection{Lax pair}
The hyperbolic Ernst equation (\ref{ernst}) admits the Lax pair
\begin{equation}
\begin{cases}\label{lax}
\Phi_x(x,y, k) = \mathsf{U}(x,y, k) \Phi(x,y, k),
	\\
\Phi_y(x,y, k) = \mathsf{V}(x,y,k) \Phi(x,y,k),
\end{cases}
\end{equation}
where $k$  is the spectral parameter, the function $\Phi(x,y, k)$ is a $2 \times 2$-matrix valued eigenfunction, and the $2\times 2$-matrix valued functions $\mathsf{U}(x,y,k)$ and $\mathsf{V}(x,y,k)$ are defined as follows:
$$\mathsf{U} = \frac{1}{\mathcal{E} + \bar{\mathcal{E}}} \begin{pmatrix} \bar{\mathcal{E}}_x & \lambda \bar{\mathcal{E}}_x \\
\lambda \mathcal{E}_x & \mathcal{E}_x \end{pmatrix}, \qquad 
\mathsf{V} =  \frac{1}{\mathcal{E} + \bar{\mathcal{E}}} \begin{pmatrix} \bar{\mathcal{E}}_y & \frac{1}{\lambda} \bar{\mathcal{E}}_y \\
\frac{1}{\lambda} \mathcal{E}_y & \mathcal{E}_y  
\end{pmatrix},$$
with $\lambda$ given by (\ref{lambdadef}). 
We write the Lax pair (\ref{lax}) in terms of differential forms as
\begin{equation}\label{laxdiffform}  
  d\Phi = W\Phi,
\end{equation}
where $W$ is the closed one-form
\begin{align}\label{Wdef}  
  W = \mathsf{U}dx + \mathsf{V}dy.
\end{align}
As in Section \ref{linearlimitsec}, we will view the map $\Phi(x,y,\cdot)$ as being defined on the Riemann surface $\mathcal{S}_{(x,y)}$ and write $\Phi(x,y,P)$ for the value of $\Phi$ at $P = (\lambda, k) \in \mathcal{S}_{(x,y)}$.

\subsection{Spectral analysis}

Suppose that $\mathcal{E}_0(x)$, $x \in [0, 1)$, and $\mathcal{E}_1(y)$, $y \in [0,1)$ are real-valued functions satisfying (\ref{E0E1assumptions}) for some $n \geq 2$. Let $\mathsf{U}_0$ and $\mathsf{V}_1$ be given by (\ref{U0V1def}), i.e., $\mathsf{U}_0$ and $\mathsf{V}_1$ denote the functions $\mathsf{U}$ and $\mathsf{V}$ evaluated at  $y = 0$ and $x = 0$, respectively.
Let $\Phi_0(x,P)$ and $\Phi_1(y,P)$ be the eigenfunctions defined in terms of $\mathcal{E}_0$ and $\mathcal{E}_1$ via the Volterra integral equations (\ref{Phi0Phi1}).

\begin{lemma}[Solution of the $x$-part]\label{xpartlemma}
The eigenfunction $\Phi_0(x,P)$ defined via the Volterra integral equation (\ref{Phi0Phi1})  has the following properties:
\begin{enumerate}[$(a)$]

\item For each $k \in \hat{\C} \setminus [0,1]$, the function $x \mapsto \Phi_0(x,k^+)$ is continuous on $[0,1)$ and is $C^n$ on $(0,1)$. Furthermore, for each $x \in [0,1)$, the function $k \mapsto \Phi_0(x,k^+)$ is analytic on $\hat{\C} \setminus [0,1]$.

   \item $\Phi_0$ obeys the symmetries
  \begin{align}\label{phi0symmetries}
\begin{cases} \Phi_0(x,k^+) = \sigma_3\Phi_0(x, k^-)\sigma_3, 
	\\
\Phi_0(x,k^\pm) = \sigma_1\overline{\Phi_0(x, \bar{k}^\pm)}\sigma_1, 
\end{cases} \qquad x \in [0,1), \ k \in \hat{\C} \setminus [0, 1].
\end{align}

  \item For each  $x \in [0, 1)$, $\Phi_0(x,P)$ extends continuously to an analytic function of $P \in \mathcal{S}_{(x,0)} \setminus \Sigma_0$.
    
  \item The value of $\Phi_0$ at $P = \infty^+$ is given by
  \begin{align}\label{Phi0atinftyplus}
  \Phi_0(x,\infty^+) = \frac{1}{2} \begin{pmatrix} \overline{\mathcal{E}_0(x)} & 1 \\ \mathcal{E}_0(x) & -1 \end{pmatrix}\begin{pmatrix}1 & 1 \\ 1 & -1 \end{pmatrix}, \qquad x \in [0, 1).
  \end{align}

\item The determinant of $\Phi_0$ is given by
$$\det \Phi_0(x,P) = \re \mathcal{E}_0(x), \qquad x \in [0,1), \ P \in \mathcal{S}_{(x,0)} \setminus \Sigma_0.$$

  \item For each $x_0 \in (0,1)$  and each compact subset $K \subset \hat{\C} \setminus [0, x_0]$, 
  \begin{align}\label{xkphimap}
  x \mapsto \big(k \mapsto \Phi_0(x,k^+)\big)
  \end{align}
  is a continuous map $[0, x_0) \to L^\infty(K)$ and a $C^n$-map $(0, x_0) \to L^\infty(K)$. Moreover, the map
$x \mapsto \big( k \mapsto x^\alpha \Phi_{0x}(x,k^+) \big)$ is continuous $[0, x_0) \to L^\infty(K)$.

\end{enumerate}
\end{lemma}
\begin{proof}
We first use successive approximations to show that the integral equation
\begin{align}\label{phixk}
\Phi_0(x, k^+) = I + \int_0^x \mathsf{U}_0(x',k^+) \Phi_0(x',k^+) dx', \qquad x \in [0,1),
\end{align}
has a unique solution for each $k \in \hat{\C} \setminus [0,1]$. Let $K$ be a compact subset of $\hat{\C} \setminus [0,1]$.
Let $\Phi_0^{(0)} = I$ and define $\Phi_0^{(j)}(x,k^+)$ for $j \geq 1$ inductively by 
\begin{align*}
& \Phi_0^{(j+1)}(x,k^+) = \int_0^x \mathsf{U}_0(x',k^+) \Phi_0^{(j)}(x',k^+) dx',	  \qquad x \in [0,1), \ k \in K.
\end{align*}
Then
\begin{align}\label{xPhijiterated}
\Phi_0^{(j)}(x,k^+) = &\; \int_{0 \leq x_1 \leq \cdots \leq x_j \leq x} \mathsf{U}_0(x_j, k^+) \mathsf{U}_0(x_{j-1}, k^+)
\cdots \mathsf{U}_0(x_1, k^+) dx_1 \cdots dx_j. 
\end{align}
The function $\lambda(x,0,k^+)$ is analytic for $k \in \hat{\C} \setminus [x,1]$; in particular, it is a bounded function of $k \in K$ for each fixed $x \in [0,1)$. In view of the assumptions (\ref{E0E1assumptions}), this implies 
$$\|\mathsf{U}_0(x,k^+)\|_{L^1([0,x])} < C(x), \qquad x \in [0,1),\ k \in K,$$
where the function  $C(x)$  is bounded on each compact subset of $[0,1)$.
Thus
\begin{align} \label{phi0jestimate}
|\Phi_0^{(j)}(x,k^+)| \leq & \frac{1}{j!} \|\mathsf{U}_0(\cdot, k^+)\|_{L^1([0, x])}^j 
\leq \frac{1}{j!} C(x)^j, \qquad x \in [0,1), \ k \in K.
\end{align}
Hence the series
\begin{align} \label{phi0assum}
	\Phi_0(x,k^+) = \sum_{j=0}^\infty \Phi_0^{(j)}(x,k^+)
\end{align}
converges absolutely and uniformly for $k \in K$ and $x$ in compact subsets of $[0,1)$ to a continuous solution $\Phi_0(x,k^+)$ of (\ref{phixk}). The fact that $x \mapsto \Phi_0(x,k^+) \in C^n((0,1)) $ follows from differentiating $x\mapsto \Phi_0^{(j)}(x,k^+)$ and applying estimates similar to \eqref{phi0jestimate} to the derivative.
Differentiating (with respect to $k$) under the integral sign in (\ref{xPhijiterated}), we see that $k \mapsto \Phi_0^{(j)}(x, k^+)$ is analytic on $\Int K$ for each $j$; the uniform convergence then proves that $k \mapsto \Phi_0(x, k^+)$  is analytic on $\Int K$. 
A similar argument applies to the integral equation defining $\Phi_0(x,k^-)$. We conclude that the functions $\Phi_0(x,k^+)$ and $\Phi_0(x,k^-)$ are well-defined for $x \in [0,1)$ and $k \in \hat{\C} \setminus [0,1]$ and are analytic functions of $k \in \hat{\C} \setminus [0,1]$ for each fixed $x$.  

We next show uniqueness. Assume that $\tilde \Phi_0$ is another solution of the Volterra equation \eqref{phixk} such that $x \mapsto \Phi_0(x,k^\pm)$ is continuous on $[0,1)$, respectively, and let $\Psi =\Phi_0 - \tilde \Phi_0$. Then $\Psi$ is a solution of the homogeneous equation
\[
\Psi(x,k^\pm) = \int_0^x \mathsf{U}_0(x',k^\pm) \Psi (x',k^\pm) dx'.
\] 
Iterating this yields
\begin{align*}
\Psi(x,k^\pm) &= \int_0^x  \mathsf{U}_0(x_j,k^\pm) \int_{0}^{x_j}  \mathsf{U}_0(x_{j-1},k^\pm) \cdots \int_{0}^{x_2}  \mathsf{U}_0(x_{1},k^\pm)\Psi(x_1,k^\pm) \,dx_1 \ldots dx_n
\\
&= \int_{0 \leq x_1 \leq \cdots \leq x_j \leq x} \mathsf{U}_0(x_j, k^\pm) \mathsf{U}_0(x_{j-1}, k^\pm)
\cdots \mathsf{U}_0(x_1, k^\pm)\Psi(x_1,k^\pm) dx_1 \cdots dx_j.
\end{align*}
Hence, as in the proof of existence, we get the estimate
\[
|\Psi(x,k^\pm) | \le \sup_{x'\in[0,x]} |\Psi(x',k^\pm) | \frac{\lVert  \mathsf{U}_0(\cdot,k^\pm) \rVert_{L^1([0,x])}^j}{j!} \to 0, \qquad j \to \infty,
\]
which yields $\Psi = 0$. This proves $(a)$. 

The symmetries (\ref{lambdasymm}) of $\lambda$ show that
$$\mathsf{U}_0(x,k^+) = \sigma_3\mathsf{U}_0(x, k^-)\sigma_3, \qquad \mathsf{U}_0(x,k^+) = \sigma_1\overline{\mathsf{U}_0(x,\bar{k}^+)}\sigma_1.$$
Hence $\sigma_3\Phi_0(x, k^-)\sigma_3$ and $\sigma_1\overline{\Phi_0(x, \bar{k}^+)}\sigma_1$ satisfy the same Volterra equation as $\Phi_0(x,k^+)$. By uniqueness, all three functions must be equal. This proves $(b)$.

We next show that $\Phi_0(x, k^\pm)$ can be continuously extended across the branch cut to an analytic function on $\mathcal{S}_{(x,0)} \setminus \Sigma_0$.
Since $\mathsf{U}_0(x,k^\pm)$ has continuous boundary values on the interval $(x, 1)$, the above argument (applied with a $K$ that reaches up to the boundary) shows that $\Phi_0(x,k^\pm)$ also has continuous boundary values on $(x, 1)$. 
Moreover, since 
$$\lambda(x, 0, (k+i0)^+) = \lambda(x,0,(k-i0)^-), \qquad k \in (x, 1),$$
the boundary functions $\Phi(x,0, (k + i0)^+)$ and $\Phi(x,0, (k - i0)^-)$ satisfy the same integral equation, so by uniqueness they are equal:
\begin{align*}
\Phi(x,y, (k + i0)^+) = \Phi(x,y,(k-i0)^-), \qquad (x,y) \in D, \ k \in (x, 1).
\end{align*}
Hence the values of $\Phi_0$ on the upper and lower sheets of $\mathcal{S}_{(x,0)}$ fit together across the branch cut $(x,1)$, showing that $\Phi_0$ extends to an analytic function of $P \in \mathcal{S}_{(x,0)} \setminus \big(\Sigma_0 \cup \{1\}\big)$. But $\lambda(x,0,P)$ is bounded in a neighborhood of the branch point $1$, hence the possible singularity of $\Phi_0(x,P)$ at this point must be removable. This shows that $\Phi_0$ satisfies $(c)$.

Since $\lambda(x,y,\infty^+) = 1$, $\Phi_0(x,\infty^+)$ satisfies the equation
$$\Phi_{0x}(x, \infty^+) = \frac{1}{2\re\mathcal{E}_0(x)} \begin{pmatrix} \overline{\mathcal{E}_{0x}(x)} & \overline{\mathcal{E}_{0x}(x)} \\
\mathcal{E}_{0x}(x) & \mathcal{E}_{0x}(x) \end{pmatrix} \Phi_0(x,\infty^+), \qquad x \in [0, 1).$$
This equation has the two linearly independent solutions
$$\begin{pmatrix} \overline{\mathcal{E}_0(x)}\\ \mathcal{E}_0(x) \end{pmatrix} \quad \text{and} \quad \begin{pmatrix} 1 \\  -1 \end{pmatrix}.$$
Hence there exists a constant matrix $A$ such that
$$\Phi_0(x, \infty^+) = \begin{pmatrix} \overline{\mathcal{E}_0(x)} & 1 \\ \mathcal{E}_0(x) & -1 \end{pmatrix} A, \qquad x \in [0, 1).$$
We determine $A$ by evaluating this equation at $x = 0$ and using that $\mathcal{E}_0(0) = 1$ and $\Phi_0(0, \infty^+) = I$. This yields (\ref{Phi0atinftyplus}) and proves $(d)$.

The proof of $(e)$ relies on the general identity
$$(\ln \det B)_x = \tr(B^{-1} B_x),$$
where $B = B(x)$ is a differentiable matrix-valued function taking values in $GL(n, \C)$.
We find
$$(\ln \det \Phi_0)_x = \tr(\Phi_0^{-1} \mathsf{U}_0 \Phi_0) = \tr \mathsf{U}_0 = \frac{\re \mathcal{E}_{0x}}{\re \mathcal{E}_0} = (\ln \re \mathcal{E}_0)_x.$$
This relation is valid at least for small $x$ because $\Phi_0(0, k^\pm) = I$ is invertible. In fact, since $\re \mathcal{E}_0(x) > 0$ for $x \in [0,1)$ by assumption (\ref{E0E1assumptions}), it extends to all of $[0,1)$ and we infer that, for $P \in \mathcal{S}_{(x,0)} \setminus \Sigma_0$,
$$\det \Phi_0(x, P) = C(P) \re \mathcal{E}_0(x), \qquad x \in [0,1),$$
where $C(P) \in \C$ is independent of $x$. Evaluation at $x = 0$ gives $C_P = 1$. This proves $(e)$.

It remains to prove $(f)$. Fix $x_0 \in (0,1)$  and let $K$  be a compact subset of $\hat{\C} \setminus [0, x_0]$.
The function $\lambda(x,0,\cdot)$ is bounded on $\mathcal{S}_{(x,0)}$ except for a simple pole at $k = x$.
Hence,  
\begin{align*}
& \sup_{k \in K} \big|\Phi_0(x_2, k^+) - \Phi_0(x_1, k^+)\big|
= \sup_{k \in K} \bigg|\int_{x_1}^{x_2} (\mathsf{U}_0\Phi_0)(x, k^+) dx\bigg|
\\
& \leq \bigg( \sup_{k \in K} \sup_{x \in [0, x_0)} |x^\alpha \mathsf{U}_0(x,k^+)|\bigg) \sup_{k \in K} \left(\int_{x_1}^{x_2} |x^{-\alpha}\Phi_0(x, k^+)| dx \right) \\
&\leq C \sup_{k \in K} \left(\int_{x_1}^{x_2} |x^{-\alpha}\Phi_0(x, k^+)| dx \right), \qquad x_1, x_2 \in [0, x_0),
\end{align*}
where the right-hand side tends to zero as $x_2 \to x_1$, because
\begin{align*}
\sup_{k \in K} \left(\int_{x_1}^{x_2} |x^{-\alpha} \Phi_0(x, k^+)| dx \right) 
& \le \sum_{j=0}^{\infty} \frac{1}{j!} \sup_{k\in K} \lVert \mathsf{U}_0(\cdot,k^+)\rVert_{L^1([0,x_0])}^j  \int_{x_1}^{x_2} x^{-\alpha} dx 
	\\
& \le \frac{e^{C(x_0)}(x_2^{1-\alpha}-x_1^{1-\alpha})}{1-\alpha},
\end{align*}
where $C(x_0)$ is chosen as in the proof of $(a)$.
This shows that the map (\ref{xkphimap}) is continuous $[0, x_0) \to L^\infty(K)$.
If $x \in (0, x_0)$, then
\begin{align*}
  \sup_{k \in K} \bigg| &\frac{\Phi_0(x+h, k^+) - \Phi_0(x,k^+)}{h} - \Phi_{0x}(x, k^+)\bigg|
  	\\
  & \leq \sup_{k \in K} \bigg| \frac{1}{h} \int_x^{x+h} (\mathsf{U}_0\Phi_0)(x', k^+) dx' - \Phi_{0x}(x,k^+)\bigg|
  	\\
&  \leq \sup_{k \in K} \bigg| (\mathsf{U}_0\Phi_0)(\xi, k^+)  - (\mathsf{U}_0\Phi_0)(x, k^+)\bigg|,
\end{align*}
where $\xi$ lies between $x$ and $x+h$. As $h \to 0$, the right-hand side goes to zero. Hence (\ref{xkphimap}) is differentiable as a map $(0,x_0) \to L^\infty(K)$ and the derivative satisfies $\Phi_{0x}(x,k^+) = \mathsf{U}_0(x,k^+)\Phi_{0}(x)$. 
Furtermore, the map
$$x \mapsto \big(k \mapsto \lambda(x,0,k^+)\big)$$
is $C^\infty$ from $(0, x_0)$ to $L^\infty(K)$ and $\mathcal{E}_{0}$ is $C^{n}$ on $(0,1)$. Hence the map
$$x \mapsto \big(k \mapsto \mathsf{U}_0(x, k^+)\big)$$
is $C^{n-1}$ from $(0, x_0)$ to $L^\infty(K)$.
It follows that (\ref{xkphimap}) is a $C^n$-map $(0, x_0) \to L^\infty(K)$. 

Finally, since
$$x^\alpha \Phi_{0x}(x,k^+) = x^\alpha \mathsf{U}_0(x, k^+) \Phi_0(x,k^+)$$
we see that $x \mapsto x^\alpha \Phi_{0x}(x,k^+)$ is continuous $[0, x_0) \to L^\infty(K)$.
This proves $(f)$ and completes the proof of the lemma.
\end{proof}

\begin{lemma}[Solution of the $y$-part]\label{ypartlemma}
The eigenfunction $\Phi_1(y,P)$ is well-defined for $y \in [0,1)$  and $P \in \mathcal{S}_{(0,y)} \setminus \Sigma_1$ and has the following properties:
\begin{enumerate}[$(a)$]

	\item For each $k \in \hat{\C} \setminus [0,1]$, the function $y \mapsto \Phi_1(y,k^+)$ is continuous on $[0,1)$ and is $C^n$ on  $(0,1)$. Furthermore, for each $y \in [0,1)$, the function $k \mapsto \Phi_0(x,k^+)$ is analytic on $\hat{\C} \setminus [0,1]$.

   \item $\Phi_1$ obeys the symmetries
  \begin{align}\label{phi1symmetries}
\begin{cases} \Phi_1(y,k^+) = \sigma_3\Phi_1(y, k^-)\sigma_3, 
	\\
\Phi_1(y,k^\pm) = \sigma_1\overline{\Phi_1(y, \bar{k}^\pm)}\sigma_1, 
\end{cases} \qquad y \in [0,1), \ k \in \hat{\C} \setminus [0, 1].
\end{align}
      \item For each  $y \in [0, 1)$, $\Phi_1(y,P)$ is an analytic function of $P \in \mathcal{S}_{(0,y)} \setminus \Sigma_1$.
  \item The value of $\Phi_1$ at $P = \infty^+$ is given by
  \begin{align}\label{Phi1atinftyplus}
  \Phi_1(y,\infty^+) = \frac{1}{2} \begin{pmatrix} \overline{\mathcal{E}_1(y)} & 1 \\ \mathcal{E}_1(y) & -1 \end{pmatrix}\begin{pmatrix}1 & 1 \\ 1 & -1 \end{pmatrix}, \qquad y \in [0, 1).
  \end{align}

\item The determinant of $\Phi_1$ is given by
$$\det \Phi_1(y,P) = \re \mathcal{E}_1(y), \qquad y \in [0,1), \ P \in \mathcal{S}_{(0,y)} \setminus \Sigma_1.$$

  \item For each $y_0 \in (0,1)$  and each compact subset $K \subset \hat{\C} \setminus [0, y_0]$, 
  \begin{align}\label{ykphimap}
  y \mapsto \big(k \mapsto \Phi_1(y,k^+)\big)
  \end{align}
  is a continuous map $[0, y_0) \to L^\infty(K)$ and a $C^n$-map $(0, y_0) \to L^\infty(K)$. Moreover, the map
  $y \mapsto \big( k \mapsto y^\alpha \Phi_{1y}(y,k^+)\big)$ is continuous $[0, y_0) \to L^\infty(K)$.
\end{enumerate}
\end{lemma}
\begin{proof}
The proof is similar to that of Lemma \ref{xpartlemma}.
\end{proof}

\subsection{Uniqueness}

The following lemma ensures uniqueness of the solution of the RH problem (\ref{RHm}). The proof relies on the fact that the determinant of the jump matrix $v$ defined in (\ref{jumpdef}) is constant on each of the subcontours $\Gamma_0$ and $\Gamma_1$.

\begin{lemma}\label{uniquenesslemma}
Suppose that $\mathcal{E}_0(x)$, $x \in [0, 1)$, and $\mathcal{E}_1(y)$, $y \in [0,1)$ are real-valued functions satisfying (\ref{E0E1assumptions}) for some $n \geq 2$. 
Then, for each $(x,y)\in D$, the solution $m(x,y,\cdot)$ of the RH problem (\ref{RHm}) is unique, if it exists. Moreover, 
\begin{align}\label{detmone}
\det m(x,y,z) = 1, \qquad (x,y)\in D, \ z \in \Omega_\infty.
\end{align} 
\end{lemma}
\begin{proof}
Fix $(x,y) \in D$. 
By (\ref{Phidet}) and the definition (\ref{jumpdef}) of $v$, we have 
$$\det v(x,y,z) = \begin{cases} \re \mathcal{E}_0(x) > 0, \quad & z \in \Gamma_0, \\
 \re \mathcal{E}_1(y) > 0, \quad & z \in \Gamma_1.
\end{cases}$$
Hence 
$$\sqrt{\det v(x,y,z)} = \begin{cases} c_0(x), \quad & z \in \Gamma_0, \\
c_1(y), & z \in \Gamma_1,
\end{cases} \quad (x,y) \in D,$$
where the two functions $c_0(x) > 0$ and $c_1(y) > 0$ are independent of $z$. 
The function $m(x,y,\cdot)$ is a solution of the RH problem (\ref{RHm}) if and only if the function $\tilde{m}(x,y,\cdot)$ defined by 
$$\tilde{m}(x,y,z) = \begin{cases} c_0(x) m(x,y,z), \quad & z \in \Omega_0, \\
c_1(y) m(x,y,z), & z \in \Omega_1, \\ 
m(x,y,z),  & z \in \Omega_\infty, \end{cases}$$
satisfies the RH problem 
\begin{align*}
\begin{cases}
\text{$\tilde{m}(x, y, \cdot)$ is analytic in $\C \setminus \Gamma$},\\
\text{$\tilde{m}_+(x,y,z) = \tilde{m}_-(x,y,z) \tilde{v}(x,y,z)$ for all $z \in \Gamma$},\\
\text{$\tilde{m}(x, y, z) = I + O(z^{-1})$ as $z\to \infty$},
\end{cases}
\end{align*}
where
$$\tilde{v}(x,y,z) = \begin{cases} \frac{1}{c_0(x)}v(x,y,z), & z \in \Gamma_0, \\
\frac{1}{c_1(y)} v(x,y,z), \quad & z \in \Gamma_1.
\end{cases}$$
But $\det \tilde{v}(x,y,z) = 1$ for all $z \in \Gamma$; hence the solution $\tilde{m}(x,y,\cdot)$ is unique and $\det \tilde{m} =1$. It follows that the solution $m$ is unique and that $\det m(x,y,z) = \det \tilde{m}(x,y,z) = 1$ for $z \in \Omega_\infty$. 
\end{proof}

\section{Proofs of main results}\label{proofsec}
In this section, we use the lemmas from the previous section to prove Theorem \ref{mainth1}--\ref{mainth4}.

\subsection{Proofs of Theorem 1 \& 2}
Let $\mathcal{E}_0(x)$, $x \in [0, 1)$, and $\mathcal{E}_1(y)$, $y \in [0,1)$ be complex-valued functions satisfying (\ref{E0E1assumptions}) for some $n \geq 2$. Suppose $\mathcal{E}(x,y)$ is a $C^n$-solution of the Goursat problem for (\ref{ernst}) in $D$ with data $\{\mathcal{E}_0, \mathcal{E}_1\}$. We will show that $\mathcal{E}(x,y)$ can be uniquely expressed in terms of $\mathcal{E}_0$ and $\mathcal{E}_1$ by (\ref{Erecover}). 

The idea in what follows is to introduce a solution $\Phi$ of (\ref{lax}) as the solution of the integral equation 
$$\Phi(x,y,k^\pm) = I + \int_{(0,0)}^{(x,y)} (W\Phi)(x',y', k^\pm).$$
However, since $W$ in general is singular on the boundary of $D$, we need to be more careful with the definition. We therefore instead define $\Phi$ as the solution of 
\begin{align}\label{Phidef}
\Phi(x, y, k^+) = \Phi_0(x, k^+) + \int_0^y (\mathsf{V} \Phi)(x,y',k^+) dy', \qquad (x,y) \in D, \ k \in \hat{\C} \setminus [0,1].
\end{align}

\begin{lemma}[Solution of Lax pair equations]\label{claim1}
The function $\Phi(x,y,P)$ defined in (\ref{Phidef}) has the following properties:
\begin{enumerate}[$(a)$]

\item $\Phi(x,y,k^\pm)$ is a well-defined $2\times 2$-matrix valued function of $(x,y) \in D$ and $k \in \hat{\C} \setminus [0,1]$ which also satisfies the alternative Volterra integral equation:
\begin{align}\label{Phidef2}
\Phi(x, y, k^+) = \Phi_1(y, k^+) + \int_0^x (\mathsf{U} \Phi)(x',y,k^+) dx', \qquad (x,y) \in D, \ k \in \hat{\C} \setminus [0,1].
\end{align}

\item For each $k \in \hat{\C} \setminus [0,1]$, the function $(x,y) \mapsto \Phi(x,y,k^+)$ is continuous on $D$ and is $C^n$ on  $\Int D$.

\item For each $k \in \hat{\C} \setminus [0,1]$, the functions 
$$(x,y) \mapsto x^\alpha \Phi_x(x,y,k^+), \quad
(x,y) \mapsto y^\alpha \Phi_y(x,y,k^+), \quad
(x,y) \mapsto x^\alpha y^\alpha \Phi_{xy}(x,y,k^+),$$
are continuous on $D$.

   \item $\Phi$ obeys the symmetries
  \begin{align}\label{phisymmetries}
\begin{cases} \Phi(x,y,k^+) = \sigma_3\Phi(x, y, k^-)\sigma_3, 
	\\
\Phi(x,y,k^\pm) = \sigma_1\overline{\Phi(x, y,\bar{k}^\pm)}\sigma_1, 
\end{cases} \qquad (x,y) \in D, \ k \in \hat{\C} \setminus [0,1].
\end{align}
    
  \item For each  point $(x,y) \in D$, $\Phi(x,y,P)$ extends continuously to an analytic function of $P \in \mathcal{S}_{(x,y)} \setminus \Sigma$, where $\Sigma = \Sigma_0 \cup \Sigma_1$ is the contour defined in (\ref{Sigma01def}).

  \item The value of $\Phi$ at $P = \infty^+$ is given by
  \begin{align}\label{Phiatinftyplus}
  \Phi(x,y,\infty^+) = \frac{1}{2} \begin{pmatrix} \overline{\mathcal{E}(x,y)} & 1 \\ \mathcal{E}(x,y) & -1 \end{pmatrix}\begin{pmatrix}1 & 1 \\ 1 & -1 \end{pmatrix}, \qquad (x,y) \in D.
  \end{align}

\item The determinant of $\Phi$ is given by
\begin{align}\label{Phidet}
\det \Phi(x,y,P) = \re \mathcal{E}(x,y) > 0, \qquad (x,y) \in D, \ P \in \mathcal{S}_{(x,y)} \setminus \Sigma.
\end{align}
\end{enumerate}
\end{lemma}
\begin{proof}
By Lemma \ref{xpartlemma} the lemma holds for $y = 0$, i.e., the function $\Phi(x,0,P)$ is well-defined and  the properties $(a)$-$(d)$ are satisfied when $x = 0$ or $y=0$. 
In order to see that $\Phi$ is well-defined also for $(x,y)$ in the interior of $D$, we note that (\ref{Phidef}) implies
\begin{align}\label{Phidefxy}
\Phi(x, y, k^+) = \Phi(x,0, k^+) + \int_0^y \mathsf{V}(x,y',k^+) \Phi(x,y',k^+) dy', \qquad (x,y) \in D, \ k \in \hat{\C} \setminus [0,1].
\end{align}
The same type of successive approximation argument already used in the proof of Lemma \ref{xpartlemma} shows that the Volterra equation (\ref{Phidefxy}) has a unique solution for each fixed $x \in (0,1)$ and each $k \in \hat{\C} \setminus [0,1]$, and that this solution $\Phi(x,y,P)$ extends continuously to an analytic function of $P \in \mathcal{S}_{(x,y)} \setminus \Sigma$. This proves $(b)$.

 In order to prove $(a)$, it remains to deduce the alternative representation (\ref{Phidef2}). Note that $\Phi_y = V\Phi$ by definition and
\begin{align*}
\Phi_x(x,y,k^+)
&=\Phi_x(x,0,k^+) + \int_{0}^y \V_x\Phi(x,y',k^+)+ \V \Phi_x(x,y',k^\pm)dy'.
\end{align*}
Since $\E$ is a solution of the Goursat problem, we have
\[
\V_x=\U_y +[\U,\V],
\]
and, moreover, $\Phi_x(x,0,k^+) =\U\Phi(x,0,k^+)$. Now a straightforward calculation shows
\begin{align*}
\Phi_x(x,y,k^+)=  \U \Phi(x,y,k^+) + \int_{0}^y  \V \Phi_x(x,y',k^\pm)-\V \U\Phi(x,y',k^+) dy'.
\end{align*}
Thus the function $\tilde{\Phi}=\Phi_x -\U\Phi$ is the unique solution of the Volterra integral equation
\[
\tilde{\Phi}(x,y',k^+) = \int_0^y \V \tilde{\Phi}(x,y',k^+)dy'
\]
giving $\tilde{\Phi}=0$. This implies $\Phi_x = \U \Phi$. Consequently, $\Phi$, defined by (\ref{Phidef}), is an eigenfunction for the Lax pair equations \eqref{lax}. The difference between (\ref{Phidef}) and (\ref{Phidef2}) is given by
\begin{align*}
&\Phi_0(x,k^+) - \Phi_1(y,k^+) +\int_0^y \V \Phi(x,y',k^+)dy' - \int_0^x \U \Phi(x',y,k^+)dx' 
\\
=& \, \int_0^y \int_0^x (\V \Phi)_x(x',y',k^+)dx' dy' - \int_0^x \int_0^y (\U \Phi)_y(x',y',k^+)dy'dx' 
\\
=& \, \int_0^x \int_0^y (\V \Phi)_x(x',y',k^+)-(\U \Phi)_y(x',y',k^+) dy' dx'
\end{align*}
and $(\V\Phi)_x = (\U \Phi)_y$ is the compatibility condition for the Lax pair. Hence the two representations (\ref{Phidef}) and (\ref{Phidef2}) are equal.
This proves $(a)$.

The symmetries (\ref{lambdasymm}) of $\lambda$ show that
$$W(x,y,k^+) = \sigma_3W(x,y, k^-)\sigma_3, \qquad W(x,y,k^+) = \sigma_1\overline{W(x,y,\bar{k}^+)}\sigma_1.$$
Since $\lambda(x,y,\infty^+) = 1$, $\Phi(x,y,\infty^+)$ satisfies the equation
$$\Phi_y(x, y, \infty^+) = \frac{1}{2\re\mathcal{E}(x,y)} \begin{pmatrix} \overline{\mathcal{E}_y(x,y)} & \overline{\mathcal{E}_y(x,y)} \\
\mathcal{E}_y(x,y) & \mathcal{E}_y(x,y) \end{pmatrix} \Phi(x,y,\infty^+), \qquad (x,y) \in D.$$
Using the above equations and arguing as in the proof of Lemma \ref{xpartlemma}, the statements $(c)$, $(d)$, $(e)$, $(f)$, and $(g)$ follow from equation (\ref{Phidefxy}) and the corresponding statements in Lemma \ref{xpartlemma}.
\end{proof}

Part $(g)$ of Lemma \ref{claim1} implies that the inverse matrix $\Phi(x,y,P)^{-1}$ is well-defined for $(x,y) \in D$ and $P \in \mathcal{S}_{(x,y)} \setminus \Sigma$.

\begin{lemma}\label{claim2}
For each $(x,y) \in D$, 
\begin{align}\label{phiminusphi}
P \mapsto \Phi(x,y,P)\Phi(x,0, P)^{-1} \quad \text{and} \quad P \mapsto \Phi(x,y,P)\Phi(0,y,P)^{-1}
\end{align}
are analytic functions of $P \in \mathcal{S}_{(x,y)} \setminus \Sigma_1$ and $P \in \mathcal{S}_{(x,y)} \setminus \Sigma_0$, respectively.
\end{lemma}
\begin{proof}
Let $U$ be an open set in $\mathcal{S}_{(x,y)} \setminus \Sigma_1$.
Multiplying (\ref{Phidefxy}) by $\Phi(x,0,P)^{-1}$ from the right, we find
\begin{align}\nonumber
& \Phi(x, y, P)\Phi(x,0,P)^{-1} = I + \int_0^y \mathsf{V}(x,y',P) \Phi(x,y',P)\Phi(x,0,P)^{-1} dy', 
	\\ \label{phiPmap}
& \hspace{8cm} (x,y) \in D, \ k \in \hat{\C} \setminus [0,1].
\end{align}
where the values of $\Phi(x,0,P)$ and $\lambda(x,y',P)$ in (\ref{phiPmap}) are to be interpreted as in Remark \ref{tilderemark}. Since
$$P \mapsto \lambda(x,y',P)^{-1} = \sqrt{\frac{k - x}{k - (1-y')}}$$
is an analytic map $U \to \C$ for each $y'$, so is $\mathsf{V}(x,y',\cdot)$. It follows that the solution $\Phi(x, y, P)\Phi(x,0,P)^{-1}$ of (\ref{phiPmap}) also is analytic for $P \in U$. This establishes the desired statement for the first map in (\ref{phiminusphi}); the proof for the second map is similar. 
\end{proof}

Let $\Omega_0$, $\Omega_1$, and $\Omega_\infty$ denote the three components of $\hat{\C} \setminus \Gamma$ defined in (\ref{Omegadef}) and displayed in Figure \ref{Omegas.pdf}.

\begin{lemma}\label{claim3}
The $2\times 2$-matrix valued function $m(x,y,z)$ defined for $(x,y)\in D$ by
\begin{align}\label{mVPhi}
m(x,y,z) = 
\Phi\big(x,y, \infty^+\big)^{-1} \Phi\big(x,y,F_{(x,y)}^{-1}(z)\big) \times \begin{cases}  \Phi\big(x,0,F_{(x,y)}^{-1}(z)\big)^{-1},  & z \in \Omega_0, \\
 \Phi\big(0,y,F_{(x,y)}^{-1}(z)\big)^{-1}, & z \in \Omega_1, \\
I, & z \in \Omega_\infty,
\end{cases}
\end{align}
satisfies the RH problem (\ref{RHm}) and the relation (\ref{Erecover}) for each $(x,y) \in D$. 
\end{lemma}
\begin{proof}
Since $F_{(x,y)}$ is a biholomorphism $\mathcal{S}_{(x,y)} \to \hat{\C}$, we infer from Lemma \ref{claim1} together with Lemma \ref{claim2} that $m(x,y, \cdot) $ is analytic in $\C \setminus \Gamma$ and that $m(x,y,z) \to I$ as $z \to \infty$ for each $(x,y) \in D$. The jump condition in (\ref{RHm}) holds as a consequence of the definition (\ref{jumpdef}) of $v(x,y,z)$ and the fact that 
$$\Phi_0(x,k) = \Phi(x,0,k), \qquad \Phi_1(y,k) = \Phi(0,y,k).$$
Finally, since $0 \in \Omega_\infty$ and $F_{(x,y)}^{-1}(0) = \infty^-$, the first symmetry in (\ref{phisymmetries}) yields 
\begin{align}\label{mxy0}
m(x,y,0) = \Phi\big(x,y, \infty^+\big)^{-1}  \Phi(x,y, \infty^-)
= \Phi\big(x,y, \infty^+\big)^{-1}  \sigma_3 \Phi\big(x,y, \infty^+\big) \sigma_3.
\end{align}
Substituting in the expression (\ref{Phiatinftyplus}) for $\Phi\big(x,y, \infty^+\big)$, the $(11)$ and $(21)$ entries of (\ref{mxy0}) give
\begin{align*}
 (m(x,y,0))_{11} =  \frac{1 + \mathcal{E}(x,y) \overline{\mathcal{E}(x,y)}}{\mathcal{E}(x,y) + \overline{\mathcal{E}(x,y)}}, \qquad
  (m(x,y,0))_{21} =  \frac{(1 - \mathcal{E}(x,y))(1 + \overline{\mathcal{E}(x,y)})}{\mathcal{E}(x,y) + \overline{\mathcal{E}(x,y)}}.
\end{align*}
Solving these two equations for $\mathcal{E}$ and $\bar{\mathcal{E}}$, we find (\ref{Erecover}).
\end{proof}

We have showed that if $\mathcal{E}(x,y)$ is a $C^n$-solution of the Goursat problem for (\ref{ernst}) in $D$ with data $\{\mathcal{E}_0, \mathcal{E}_1\}$, then $\mathcal{E}(x,y)$ can be expressed in terms of the function $m$ defined in (\ref{mVPhi}) via equation (\ref{Erecover}). 
By Lemma \ref{uniquenesslemma}, this function $m(x,y,z)$ is the unique solution of the RH-problem (\ref{RHm}) whose formulation involves only the values $\mathcal{E}_0(x')$ and $\mathcal{E}_1(y')$ for $0\le x' \le x$ and $0\le y'\le y$. As a consequence, the value of the solution $\mathcal{E}$ at $(x,y)$ is uniquely determined by the values $\mathcal{E}_0(x')$ and $\mathcal{E}_1(y')$ for $0\le x' \le x$ and $0\le y'\le y$, if it exists. This completes the proofs of Theorem \ref{mainth1} and \ref{mainth2}.

\subsection{Proof of Theorem \ref{mainth3}}
This subsection is devoted to proving Theorem \ref{mainth3} regarding existence. 
Let us therefore suppose that $\mathcal{E}_0(x)$, $x \in [0, 1)$, and $\mathcal{E}_1(y)$, $y \in [0,1)$ are real-valued functions satisfying (\ref{E0E1assumptions}) for some $n \geq 2$. 
Define $\Phi_0(x,P)$ and $\Phi_1(y,P)$ in terms of $\mathcal{E}_0$ and $\mathcal{E}_1$ via the Volterra equations (\ref{Phi0Phi1}). Then $\Phi_0$ and $\Phi_1$ have the properties listed in Lemma \ref{xpartlemma} and Lemma \ref{ypartlemma}.
Let $\delta \in (0,1)$ and let $D_\delta$ be the triangle defined in (\ref{Ddeltadef}). 
As in the proof of Theorem \ref{linearmainth}, choose $\epsilon > 0$ so small that $F_{(x,y)}(\Sigma_0)$ and $F_{(x,y)}(\Sigma_1)$ are contained in the intervals $[-\epsilon^{-1}, -\epsilon]$ and $[\epsilon, \epsilon^{-1}]$, respectively, for all $(x,y) \in D_\delta$. Fix two smooth nonintersecting clockwise contours $\Gamma_0$ and $\Gamma_1$ in the complex $z$-plane which encircle the intervals $[-\epsilon^{-1}, -\epsilon]$ and $[\epsilon, \epsilon^{-1}]$, respectively, but which do not encircle zero, see Figure \ref{Gamma01.pdf}. Suppose $\Gamma_0$ and $\Gamma_1$ are invariant under the involutions $z \mapsto z^{-1}$ and $z \mapsto \bar{z}$. 
Let $\Gamma = \Gamma_0 \cup \Gamma_1$ and consider the family of RH problems given in (\ref{RHm}) parametrized by the two parameters $(x,y) \in D_\delta$. We will show that if (\ref{RHm}) has a (unique) solution $m(x,y,z)$  for each $(x,y) \in D_\delta$, then the function $\mathcal{E}(x,y)$ defined in terms of $m$ via equation (\ref{Erecover}) satisfies
\begin{align}\label{EDdelta}
\begin{cases}
  \mathcal{E} \in C(D_\delta) \cap C^n(\Int D_\delta), 
  	\\
   \text{$\mathcal{E}(x,y)$ satisfies the hyperbolic Ernst equation (\ref{ernst}) in $\Int(D_\delta)$,}
  	\\
       \text{$x^\alpha \mathcal{E}_x, y^\alpha \mathcal{E}_y, x^\alpha y^\alpha \mathcal{E}_{xy} \in C(D_\delta)$ for some $\alpha \in [0,1)$,}
  	\\
  \text{$\mathcal{E}(x,0) = \mathcal{E}_0(x)$ for $x \in [0,1-\delta)$,}
 	\\
  \text{$\mathcal{E}(0,y) = \mathcal{E}_1(y)$ for $y \in [0,1-\delta)$.}
  	\\
  \text{$\re \mathcal{E}(x,y) > 0$ for $(x,y) \in D_\delta$.}		
\end{cases}
\end{align}

We next list some facts about Cauchy integrals that we will use throughout the proof. If $h \in L^2(\Gamma)$, then the Cauchy transform $\mathcal{C}h$ is defined by
\begin{align}\label{Cauchytransform}
(\mathcal{C}h)(z) = \frac{1}{2\pi i} \int_\Gamma \frac{h(z')}{z' - z} dz', \qquad z \in \C \setminus \Gamma,
\end{align}
We denote the nontangential boundary values of $\mathcal{C}f$ from the left and right sides of $\Gamma$ by $\mathcal{C}_+ f$ and $\mathcal{C}_-f$ respectively. Then $\mathcal{C}_+$ and $\mathcal{C}_-$ are bounded operators on $L^2(\Gamma)$ and $\mathcal{C}_+ - \mathcal{C}_- = I$.
Let $w(x,y,z) = v(x,y,z) - I$. We define the operator $\mathcal{C}_w: L^2(\Gamma) + L^\infty(\Gamma) \to L^2(\Gamma)$ by 
\begin{align}\label{Cwdef}
\mathcal{C}_{w}(f) = \mathcal{C}_-(f w).
\end{align}
Then
\begin{align}\label{Cwnorm}
\|\mathcal{C}_w\|_{\mathcal{B}(L^2(\Gamma))} \leq C \|w\|_{L^\infty(\Gamma)},
\end{align}
where $C = \|\mathcal{C}_-\|_{\mathcal{B}(L^2(\Gamma))}$. 

We henceforth assume that the RH problem \eqref{RHm} has a solution for all $(x,y) \in D_\delta$ or, equivalently, that $I - \mathcal{C}_{w} \in \mathcal{B}(L^2(\Gamma))$ is bijective for each  $(x,y) \in D_\delta$.

For each $(x,y) \in D_\delta$, we have $v \in C(\Gamma)$ and $v, v^{-1} \in I + L^2(\Gamma) \cap L^\infty(\Gamma)$. The theory of singular integral equations then implies that the solution of the RH problem (\ref{RHm}) is given by (see e.g. \cite{D1999} or \cite[Proposition 5.8]{LCarleson})
\begin{align}\label{msolution}
  m = I + \mathcal{C}(\mu w),
\end{align}
where the $2\times 2$-matrix valued function $\mu(x,y,\cdot)$ is defined by
$$\mu = I + (I - \mathcal{C}_w)^{-1}\mathcal{C}_w I \in I + L^2(\Gamma).$$
Equation (\ref{msolution}) can be written more explicitly as
\begin{align}\label{msolution2}
  m(x,y,z) = I + \frac{1}{2\pi i} \int_\Gamma \frac{(\mu w)(x,y,s) ds}{s-z}, \qquad (x,y) \in D_\delta, \ z \in \hat{\C} \setminus \Gamma.
\end{align}

\begin{lemma}\label{claim2E}
The map 
\begin{align} \label{xtww1}
(x,y) \mapsto w(x,y, \cdot) 
\end{align}
 is continuous from $D_\delta$ to $L^\infty(\Gamma)$ and $C^n$ from $\Int D_\delta$ to $L^\infty(\Gamma)$.
Moreover, the three maps 
\begin{align}\label{vthreemaps}
(x,y) \mapsto x^\alpha w_x(x,y, \cdot), \quad
(x,y) \mapsto y^\alpha w_x(x,y, \cdot), \quad
(x,y) \mapsto x^\alpha y^\alpha w_{xy}(x,y, \cdot),
\end{align}
are continuous from $D_\delta$ to $L^\infty(\Gamma)$.
\end{lemma}
\begin{proof}
For $N \geq 0$, let $C^N(K)$ denote the Banach space of functions on $K$ with continuous partial derivatives of order $\leq N$ equipped with the usual norm
$$\|f\|_{C^N(K)} = \sup_{|\alpha| \leq N} \|D^\alpha f\|_{L^\infty(K)}.$$
By part $(f)$ of Lemma \ref{xpartlemma} the map
\begin{align}\label{}
(x,y) \mapsto \Phi_0(x,\cdot): D_\delta \to C(K)
\end{align}
is continuous for any compact set $K$ not intersecting $\Sigma$. 
Moreover, assuming $F_{(x,y)}^{-1}(\Gamma) \subset K $, the map
\begin{align}\label{}
(x,y) \mapsto (f \mapsto f(F_{(x,y)}^{-1}(\cdot))):D_\delta \to \mathcal{B}(C(K), C(\Gamma))
\end{align}
is continuous, because
$$\sup_{\|f\|_{C(K)}=1} \sup_{z \in \Gamma} \big|f(F_{(x,y)}^{-1}(z)) - f(F_{(x',y')}^{-1}(z))\big| \to 0$$
as $(x',y') \to (x,y)$ by uniform continuity of $f \in C(K)$ on the compact set $K$.
It follows that the composed map
$$(x,y) \mapsto \Phi_0(x,F_{(x,y)}^{-1}(\cdot)):\Int D_\delta \to C(\Gamma_0)$$
also is continuous. A similar argument shows that 
$$(x,y) \mapsto \Phi_1(y,F_{(x,y)}^{-1}(\cdot)): D_\delta \to C(\Gamma_1)$$
is continuous. Recalling the definition (\ref{jumpdef}) of $v$, this shows that the map (\ref{xtww1}) is continuous from $D_\delta $ to $ L^\infty(\Gamma)$.

If a sequence of holomorphic functions $f_n$ converges uniformly on an open set $\Omega$ then the sequence of derivatives $f_n'$ converges uniformly on compact subsets of $\Omega$.
Fix $N \geq n$ and let $K$ be a compact subset of $\mathcal{S}_{(x,0)} \setminus \Sigma_0$. Then part $(f)$ of Lemma \ref{xpartlemma} implies that the map
\begin{align}\label{xyPhi0map}
(x,y) \mapsto \Phi_0(x,\cdot):\Int D_\delta \to C^N(K)
\end{align}
is $C^n$.
On the other hand, the map
\begin{align}\label{xyfFmap}
(x,y) \mapsto (f \mapsto f(F_{(x,y)}^{-1}(\cdot))):\Int D_\delta \to \mathcal{B}(C^N(K), C(\Gamma))
\end{align}
is $C^n$. Indeed, the map is continuous because
$$\sup_{\|f\|_{C^N(K)}=1} \sup_{z \in \Gamma} \big|f(F_{(x,y)}^{-1}(z)) - f(F_{(x',y')}^{-1}(z))\big| \to 0$$
as $(x',y') \to (x,y)$ by uniform continuity of $f$ on the compact set $K$.
Moreover, the map has a continuous partial derivative with respect to $x$ because
$$\sup_{\|f\|_{C^N(K)}=1} \sup_{z \in \Gamma} \bigg|\frac{f(F_{(x+h,y)}^{-1}(z)) - f(F_{(x,y)}^{-1}(z))}{h} - \frac{d}{dx}f(F_{(x,y)}^{-1}(z))  \bigg| \to 0$$
as $h \to 0$ by the mean-value theorem and the uniform continuity of the first partial derivatives of $f$.
Similar arguments show that all partial derivatives of order $\leq n$ exist and are continuous.
We conclude that the composed map
$$(x,y) \mapsto \Phi_0(x,F_{(x,y)}^{-1}(\cdot)):\Int D_\delta \to C(\Gamma_0)$$
built from (\ref{xyPhi0map}) and (\ref{xyfFmap}) is $C^n$. 
A similar argument shows that 
$$(x,y) \mapsto \Phi_1(y,F_{(x,y)}^{-1}(\cdot)):\Int D_\delta \to C(\Gamma_1)$$
is $C^n$. Recalling the definition (\ref{jumpdef}) of $v$, this shows that the map (\ref{xtww1}) is $C^n$ as a map from $\Int D_\delta$ to $ L^\infty(\Gamma)$.
If $z \in \Gamma_0$, we have
\[
w_x(x,y,z) = \Phi_{0x}(x,F_{(x,y)}^{-1}(z)) + \Phi_{0k}(x,F_{(x,y)}^{-1}(z))\frac{d}{d x}F_{(x,y)}^{-1}(z),
\]
where $\frac{d}{d x}F_{(x,y)}^{-1}(z)$ denotes the derivative of the $k$-projection of $F_{(x,y)}^{-1}(z)$, which is given by
\[
\frac{d}{d x}F_{(x,y)}^{-1}(z) = -\frac{(z-1)^2}{4z}.
\]
Thus part $(f)$ of Lemma \ref{xpartlemma} and of Lemma \ref{ypartlemma} imply that $(x,y) \mapsto x^\alpha w_x(x,y,\cdot)$ is a continuous map $D_\delta \to L^\infty(\Gamma)$. The maps $(x,y) \mapsto y^\alpha w_y(x,y,\cdot)$ and $(x,y) \mapsto x^\alpha y^\alpha w_{xy}(x,y,\cdot)$ can be treated similarly.
\end{proof}

\begin{lemma}\label{claim3E}
The map
\begin{align}\label{xttomuI}
(x,y) \mapsto \mu(x,y, \cdot) - I
\end{align}
is continuous from $D_\delta$ to $L^2(\Gamma)$ and $C^n$ from $\Int D_\delta$ to $L^2(\Gamma)$.
Moreover, the three maps 
\begin{align}\label{muthreemaps}
(x,y) \mapsto x^\alpha \mu_x(x,y, \cdot), \qquad
(x,y) \mapsto y^\alpha \mu_x(x,y, \cdot), \qquad
(x,y) \mapsto x^\alpha y^\alpha \mu_{xy}(x,y, \cdot),
\end{align}
are continuous from $D_\delta$ to $L^2(\Gamma)$.
\end{lemma}
\begin{proof}
In view of the definition of $\mu$, the map (\ref{xttomuI}) is given by
$$(x,y) \mapsto (I - \mathcal{C}_{w(x,y,\cdot)})^{-1}\mathcal{C}_-(w(x,y,\cdot)).$$
We note that the map
\begin{align}\label{wICw}
f \mapsto I - \mathcal{C}_f: L^\infty(\Gamma) \to \mathcal{B}(L^2(\Gamma))
\end{align}
is smooth by the estimate
\begin{align*}
\|\mathcal{C}_f\|_{\mathcal{B}(L^2(\Gamma))} \leq C\|f\|_{L^\infty(\Gamma)},
\end{align*}
and that the linear map
\begin{align}\label{wwCwI}
f \mapsto \mathcal{C}_-f:  L^2(\Gamma) \to L^2(\Gamma)
\end{align}
is bounded.
Since (\ref{xttomuI}) can be viewed as a composition of maps of the form (\ref{xtww1}), (\ref{wICw}),  and (\ref{wwCwI}) together with the smooth inversion map $I - \mathcal{C}_w \mapsto (I - \mathcal{C}_w)^{-1}$, it follows that (\ref{xttomuI}) is continuous $D_\delta \to L^2(\Gamma)$ and $C^n$ from $\Int D_\delta$ to $L^2(\Gamma)$. Similarly, $(x,y) \mapsto x^\alpha \mu_x(x,y,\cdot)$ can be viewed as composition of the continuous maps (\ref{wICw}), (\ref{wwCwI}), $I - \mathcal{C}_w \mapsto (I - \mathcal{C}_w)^{-1}$, \eqref{xtww1}, and \eqref{muthreemaps}, and is hence continuous. The maps $(x,y) \mapsto y^\alpha \mu_y(x,y,\cdot)$ and $(x,y) \mapsto x^\alpha y^\alpha \mu_{xy}(x,y,\cdot)$ can be treated analogously.
\end{proof}

\begin{lemma}\label{claim4E}
The solution $m(x,y,z)$ of the RH problem (\ref{RHm}) defined in (\ref{msolution2}) has the following properties:
\begin{enumerate}[$(a)$]
   \item For each point $(x,y) \in D_\delta$, $m(x,y,\cdot)$ obeys the symmetries
  \begin{align}\label{msymm}
 m(x,y,z) = m(x,y,0)\sigma_3m(x,y,z^{-1})\sigma_3 = \sigma_1\overline{m(x,y,\bar{z})}\sigma_1, \qquad z \in \hat{\C} \setminus \Gamma.
\end{align}

  \item For each  $z \in \hat{\C}\setminus \Gamma$, the map $(x,y) \mapsto m(x,y,z)$ is continuous from $D_\delta$ to $\C^{2 \times 2}$ and is $C^n$ from $\Int D_\delta$ to $\C^{2 \times 2}$.
  
  \item For each  $z \in \hat{\C}\setminus \Gamma$, the three maps 
  $$(x,y) \mapsto x^\alpha m_x(x,y, z), \qquad
(x,y) \mapsto y^\alpha m_x(x,y, z), \qquad
(x,y) \mapsto x^\alpha y^\alpha m_{xy}(x,y, z),$$
are continuous from $D_\delta$ to $\C^{2\times 2}$.
\end{enumerate}
\end{lemma}
\begin{proof}
The symmetries in (\ref{phi0symmetries}) and (\ref{phi1symmetries}) show that $v$ satisfies
\begin{align}\label{vsymmetries}
\begin{cases} v(x,y,z) = \sigma_3v(x, y, z^{-1})\sigma_3, 
	\\
v(x,y,z) = \sigma_1\overline{v(x, y, \bar{z})}\sigma_1, 
\end{cases} \qquad z \in \Gamma, \ (x,y) \in D_\delta.
\end{align}
These symmetries imply that $\sigma_3 m(x,y,0)^{-1} m(x,y,z^{-1})\sigma_3$ 
and $\sigma_1\overline{m(x,y,\bar{z})}\sigma_1$ satisfy the same RH problem as $m(x,y,z)$. The symmetries in (\ref{msymm}) follow by uniqueness.

Properties $(b)$ and $(c)$ follow from (\ref{msolution}) together with the Lemmas \ref{claim2E} and \ref{claim3E}. 
\end{proof}

As in the proof of Theorem \ref{linearmainth}, we extend the definition (\ref{jumpdef}) of $v$ to an open tubular neighborhood $N(\Gamma) = N(\Gamma_0) \cup N(\Gamma_1)$ of $\Gamma$ as follows, see Figure  \ref{Gammatubular.pdf}:
\begin{align}\label{jumpdef2}
v(x,y, z) = \begin{cases}
 \Phi_0\big(x, F_{(x,y)}^{-1}(z)\big), \quad & z \in N(\Gamma_0), 
 	\\
\Phi_1\big(y, F_{(x,y)}^{-1}(z)\big), & z \in N(\Gamma_1),
\end{cases}\quad (x,y) \in D_\delta.
\end{align}
We choose $N(\Gamma)$ so narrow that it does not intersect the intervals $[-\epsilon^{-1}, -\epsilon]$ and $[\epsilon, \epsilon^{-1}]$. Then, for each $(x,y) \in D_\delta$, $v(x,y,\cdot)$ is an analytic function of $z \in N(\Gamma)$. 
Using the notation $z(x,y,P) := F_{(x,y)}(P)$, we can write (\ref{jumpdef2}) as
\begin{align}\label{vzPhi0}
  v(x,y,z(x,y,P)) = \begin{cases} \Phi_0(x, P), \quad & P \in F_{(x,y)}^{-1}\big(N(\Gamma_0)\big), \\
 \Phi_1(y, P), \quad & P \in F_{(x,y)}^{-1}\big(N(\Gamma_1)\big),  
 \end{cases}
  \quad (x,y) \in D_\delta.
\end{align}
We define functions $f_0(x,y,z)$ and $f_1(x,y,z)$ for $(x,y) \in D_\delta$ by
\begin{align*}
f_0(x,y,z) = \big[m_x(x,y,z) + z_x\big(x,y, F_{(x,y)}^{-1}(z)\big) m_z(x,y,z)\big]m(x,y,z)^{-1}, \qquad z \in \hat{\C} \setminus \Gamma,
	\\
f_1(x,y,z) = \big[m_y(x,y,z) + z_y\big(x,y, F_{(x,y)}^{-1}(z)\big) m_z(x,y,z)\big]m(x,y,z)^{-1}, \qquad z \in \hat{\C} \setminus \Gamma.
\end{align*}
Moreover, we let $n_0(x,y,z)$ and $n_1(x,y,z)$ denote the functions given by
\begin{subequations}\label{ndef}
\begin{align}\label{ndefa}
n_0(x,y,z) = \begin{cases}
  f_0(x,y,z) + m(x,y,z)\mathsf{U}_0\big(x,F_{(x,y)}^{-1}(z)\big)m(x,y,z)^{-1}, \quad &  z \in \Omega_0,	\\
  f_0(x,y,z), & z \in \Omega_1 \cup \Omega_\infty,
\end{cases}
\end{align}
and
\begin{align}\label{ndefb}
n_1(x,y,z) = \begin{cases}
  f_1(x,y,z) + m(x,y,z)\mathsf{V}_1\big(y,F_{(x,y)}^{-1}(z)\big)m(x,y,z)^{-1}, \quad &  z \in \Omega_1,	\\
  f_1(x,y,z), & z \in \Omega_0 \cup \Omega_\infty.
\end{cases}
\end{align}
\end{subequations}

\begin{lemma}\label{claim5E}
For each $(x,y) \in D_\delta$, it holds that
\begin{enumerate}[$(a)$]
\item $n_0(x,y,z)$ is an analytic function of $z \in \hat{\C} \setminus \{-1\}$ and has at most a simple pole at $z = -1$. 
\item $n_1(x,y,z)$ is an analytic function of $z \in \hat{\C} \setminus \{1\}$ and has at most a simple pole at $z = 1$. 
\item $n_0(x,y,\infty) = 0$ and $n_0(x,y,0) = m_x(x,y,0)m(x,y,0)^{-1}$.
\item $n_1(x,y,\infty) = 0$ and $n_1(x,y,0) = m_y(x,y,0)m(x,y,0)^{-1}$.
\end{enumerate}
\end{lemma}
\begin{proof}
By \eqref{linearzx} the function $z_x\big(x,y, F_{(x,y)}^{-1}(z)\big)$ is analytic for $z \in \hat{\C}\setminus \{-1, \infty\}$ with simple poles at $z = -1$ and $z = \infty$. 
Equation (\ref{msolution}) implies that $m_x(x,y, z) = O(z^{-1})$ and $m_z(x,y, z) = O(z^{-2})$ as $z \to \infty$. Hence $f_0(x,y,z)$ is analytic at $z = \infty$. 
It follows that $f_0(x,y,z)$ is analytic for all $z \in \hat{\C}\setminus (\Gamma \cup \{-1\})$ with a simple pole at $z = -1$ at most.
Now $f_0$ satisfies the following jump condition across $\Gamma$:
\begin{align}\nonumber
f_{0+}(x,y,z) = &\; f_{0-}(x,y,z) + m_-(x,y,z)\big[v_x(x,y,z) + z_x\big(x,y, F_{(x,y)}^{-1}(z)\big) v_z(x,y,z)\big]
	\\ \label{f0jump}
& \times v(x,y,z)^{-1} m_-(x,y,z)^{-1}, \qquad z \in \Gamma.
\end{align}
Differentiating (\ref{vzPhi0}) with respect to $x$ and $y$ and evaluating the resulting equations at $k = F_{(x,y)}^{-1}(z)$, we find
\begin{align}\label{vxzxvza}
\begin{cases}
v_x(x,y,z) + z_x\big(x,y, F_{(x,y)}^{-1}(z)\big) v_z(x,y,z) = \Phi_{0x}(x,F_{(x,y)}^{-1}(z)),
	\\
v_y(x,y,z) + z_y\big(x,y, F_{(x,y)}^{-1}(z)\big) v_z(x,y,z) = 0,
\end{cases} \quad z \in N(\Gamma_0),
\end{align}	
and
\begin{align}\label{vxzxvzb}
\begin{cases}
v_x(x,y,z) + z_x\big(x,y, F_{(x,y)}^{-1}(z)\big) v_z(x,y,z) = 0,
	\\
v_y(x,y,z) + z_y\big(x,y, F_{(x,y)}^{-1}(z)\big) v_z(x,y,z) = \Phi_{1y}(x,F_{(x,y)}^{-1}(z)),
\end{cases} \quad z \in N(\Gamma_1).
\end{align}	
Using the first equations in (\ref{vxzxvza}) and (\ref{vxzxvzb}) in (\ref{f0jump}), we conclude that $f_0$ is analytic across $\Gamma_1$ and has the following jump across $\Gamma_0$:
\begin{align}\label{f0jump2}
f_{0+}(x,y,z) = &\; f_{0-}(x,y,z) + m_-(x,y,z)\mathsf{U}_0\big(x,F_{(x,y)}^{-1}(z)\big) m_-(x,y,z)^{-1}, \quad \ z \in \Gamma_0.
\end{align}
Thus $n_0$ is analytic across $\Gamma$. Furthermore, since $\lambda(x,y,k)$ is analytic on $\mathcal{S}_{(x,y)}$ except for a simple pole at the branch point $k = x$, the function $\mathsf{U}_0\big(x,F_{(x,y)}^{-1}(z)\big)$ is analytic for $z \in \hat{\C}\setminus \{-1\}$ with a simple pole at $z = -1$. It follows that $n_0$ satisfies $(a)$. The proof of $(b)$ is similar and relies on the second equations in (\ref{vxzxvza}) and (\ref{vxzxvzb}).

Using (\ref{linearzx}) in the definition (\ref{ndefa}) of $n_0$, we can write, for $z \in \Omega_\infty$,
\begin{align}\label{n0f0}
n_0(x,y,z) = f_0(x,y,z) = \Big[m_x(x,y,z) -\frac{1-z}{1+z} \frac{z}{1 - x - y} m_z(x,y,z)\Big]m(x,y,z)^{-1}.
\end{align}
Since $m_x(x,y, z) = O(z^{-1})$ and $m_z(x,y, z) = O(z^{-2})$ as $z \to \infty$, it follows that $n_0(x,y,\infty) = 0$. On the other hand, evaluating (\ref{n0f0}) at $z = 0$, we find 
$$n_0(x,y,0) = m_x(x,y,0)m(x,y,0)^{-1}.$$ 
This proves $(c)$; the proof of $(d)$ is analogous. 
\end{proof}

Let $\hat{m}(x,y)$ denote the function $m(x,y,z)$  evaluated at $z = 0$, that is, 
$$\hat{m}(x,y) = m(x,y,0).$$ 
Evaluating the first symmetry in (\ref{msymm}) at $z = \infty$, we find
\begin{align}\label{m0sigma3}
I = \hat{m}(x,y)\sigma_3\hat{m}(x,y)\sigma_3.
\end{align}
The unit determinant condition (\ref{detmone}) implies that $\det \hat{m} = 1$. Hence equation (\ref{m0sigma3}) reduces to
$$\text{adj}(\hat{m})=\sigma_3 \hat{m} \sigma_3,$$
where $\text{adj}$ denotes the adjugate matrix, which shows that $\hat{m}_{11} = \hat{m}_{22}$.
%Since $\det \hat{m} = 1$, this relation implies $\tr(\hat{m}\sigma_3) = 0$ (i.e. $\hat{m}_{11} = \hat{m}_{22}$).
%[Conversely, the two relations $\det m = 1$ and $\tr(\hat{m}\sigma_3) = 0$ imply that $\hat{m} \sigma_3\hat{m} \sigma_3=I$.]
A straightforward algebraic computation then yields
\begin{align}\label{mhatexpression}
\hat{m}(x,y) = \tilde{\Phi}(x,y) \sigma_3\tilde{\Phi}(x,y)\sigma_3, \qquad (x,y) \in D_\delta,
\end{align}
where the $2\times 2$-matrix valued function $\tilde{\Phi}(x,y)$ is defined by
%(our goal is to show that $\tilde{\Phi}(x,y) = \Phi(x,y,\infty^+)$)
\begin{align}\label{tildePhidef}
\tilde{\Phi}(x,y) = \frac{1}{2} \begin{pmatrix} \overline{\mathcal{E}(x,y)} & 1 \\ \mathcal{E}(x,y) & -1 \end{pmatrix}\begin{pmatrix}1 & 1 \\ 1 & -1 \end{pmatrix}
\end{align}
and the functions $\mathcal{E}(x,y)$ and $\overline{\mathcal{E}(x,y)}$ are defined by
\begin{align}\label{Edef}
\mathcal{E} = \frac{1 + \hat{m}_{11} - \hat{m}_{21}}{1 + \hat{m}_{11} + \hat{m}_{21}}, \qquad
\bar{\mathcal{E}} = -\frac{1 - \hat{m}_{11} + \hat{m}_{21}}{1 - \hat{m}_{11} - \hat{m}_{21}}.
\end{align}
The second symmetry in (\ref{msymm}) evaluated at $z = 0$ implies
\begin{align}\label{m0symm}
  \hat{m}_{11} = \overline{\hat{m}_{22}}, \qquad  \hat{m}_{12} = \overline{\hat{m}_{21}}. 
\end{align}
Recalling the relations $\hat{m}_{11} = \hat{m}_{22}$ and $\det \hat{m} = 1$, it follows that $\bar{\mathcal{E}}$ is the complex conjugate of $\mathcal{E}$. The next lemma shows, among other things, that $\mathcal{E}$ is free of singularities.

\begin{lemma}\label{claim6E}
  The function $\mathcal{E}(x,y)$ defined in (\ref{Edef}) has the following properties:
  \begin{align*}
\begin{cases}
  \mathcal{E} \in C(D_\delta) \cap C^n(\Int D_\delta), 
  \\
  x^\alpha \mathcal{E}_x, y^\alpha \mathcal{E}_y, x^\alpha y^\alpha \mathcal{E}_{xy} \in C(D_\delta),
  	\\
  \text{$\mathcal{E}(x,0) = \mathcal{E}_0(x)$ for $x \in [0,1-\delta)$,}
 	\\
  \text{$\mathcal{E}(0,y) = \mathcal{E}_1(y)$ for $y \in [0,1-\delta)$.}
  	\\
  \text{$\re \mathcal{E}(x,y) > 0$ for $(x,y) \in D_\delta$.}		
\end{cases}
\end{align*}
\end{lemma}
\begin{proof}
By Lemma \ref{claim4E}, the map $(x,y) \mapsto \hat{m}(x,y)$ is continuous from $D_\delta$ to $\C$ and is $C^n$ from $\Int D_\delta$ to $\C$.
The first equation in (\ref{Edef}) shows that $\mathcal{E}(x,y)$ also has these regularity properties except possibly on the set
\begin{align}\label{singular1}
\{(x,y) \in D_\delta \, | \, (\hat{m}(x,y))_{11} + (\hat{m}(x,y))_{21} = -1\}
\end{align}
where the denominator vanishes. 
In the same way, the second equation in (\ref{Edef}) shows that $\mathcal{E}(x,y)$ is regular away from the set
\begin{align}\label{singular2}
\{(x,y) \in D_\delta \, | \, (\hat{m}(x,y))_{11} + (\hat{m}(x,y))_{21} = 1\}.
\end{align}
Since the sets (\ref{singular1}) and (\ref{singular2}) are disjoint and closed in $D_\delta$, we conclude that $\mathcal{E} \in C(D_\delta) \cap C^n(\Int D_\delta)$. That $ x^\alpha \mathcal{E}_x, y^\alpha \mathcal{E}_y, x^\alpha y^\alpha \mathcal{E}_{xy} \in C(D_\delta)$ follows by differentiating \eqref{Edef} and applying Lemma \ref{claim4E}.
    
We next show that $\re \mathcal{E} > 0$ on $D_\delta$. 
Equation (\ref{Edef}) yields
$$\mathcal{E} + \bar{\mathcal{E}} = \frac{4\hat{m}_{21}}{(\hat{m}_{11} +\hat{m}_{21} )^2 -1}.$$
In light of the relations $\hat{m}_{11} = \hat{m}_{22}$ and $\det \hat{m} = 1$, this gives
\begin{align}\label{reEm0}
\re \mathcal{E} = \frac{2(1 + \hat{m}_{11})}{|1 + \hat{m}_{11} + \hat{m}_{12}|^2}.
\end{align}
On the other hand, the relations $\hat{m}_{11} = \hat{m}_{22}$ and $\det \hat{m} = 1$ together with (\ref{m0symm}) yield $\hat{m}_{11} \in \R$ and $\hat{m}_{11}^2 - |\hat{m}_{12}|^2 = 1$. We infer that $\hat{m}_{11} \in (-\infty,-1] \cup [1, \infty)$.
For $(x,y) = (0,0)$ we have $m(0,0,z) = I$ for all $z$, because the jump matrix $v$ is the identity matrix. In particular, $\hat{m}_{11}(0,0) = 1$. By continuity, this gives $(\hat{m}(x,y))_{11} \geq 1$ for all  $(x,y) \in D_\delta$.  In view of (\ref{reEm0}), it follows that $\re \mathcal{E}(x,y) > 0$ on $D_\delta$.

Finally, we show that $\mathcal{E}(x,0) = \mathcal{E}_0(x)$ for $x \in [0, 1-\delta)$; the proof that $\mathcal{E}(0,y) = \mathcal{E}_1(y)$ for $y \in [0,1-\delta)$ is similar.
For $y = 0$, the definition (\ref{jumpdef}) of $v$ yields
\begin{align}\label{jumpdef0}
v(x,0,z) = \begin{cases}
 \Phi_0\big(x, F_{(x,y)}^{-1}(z)\big), \quad & z \in \Gamma_0, 
 	\\
I, & z \in \Gamma_1,
\end{cases} \quad x \in [0, 1-\delta).
\end{align}
It follows from part $(c)$ of Lemma \ref{xpartlemma} that the $2\times 2$-matrix valued function $m_0(y,z)$ defined for $x \in [0,1-\delta)$ by
\begin{align}\label{m0def}
m_0(x,z) =  \Phi_0\big(x,\infty^+\big)^{-1} \times
\begin{cases} I,  & z \in \Omega_0, \\
 \Phi_0\big(x,F_{(x,0)}^{-1}(z)\big), & z \in \Omega_1 \cup \Omega_\infty, 
\end{cases}
\end{align}
satisfies the RH problem (\ref{RHm}) associated with $(x,y) = (x,0)$ for each $x \in [0,1-\delta)$.
Furthermore, since $0 \in \Omega_\infty$ and $F_{(x,y)}^{-1}(0) = \infty^-$, the first symmetry in (\ref{phi0symmetries}) yields 
\begin{align}\label{mx00}
m_0(x,0) = \Phi_0\big(x,\infty^+\big)^{-1}\Phi_0\big(x,\infty^-\big)
= \Phi_0\big(x,\infty^+\big)^{-1}\sigma_3\Phi_0\big(x,\infty^+\big)\sigma_3.
\end{align}
Substituting in the expression (\ref{Phi0atinftyplus}) for $\Phi_0\big(x,\infty^+\big)$, the $(11)$ and $(21)$ entries of (\ref{mx00}) give
\begin{align*}
 (m_0(x,0))_{11} =  \frac{1 + \mathcal{E}_0(x) \overline{\mathcal{E}_0(x)}}{\mathcal{E}_0(x) + \overline{\mathcal{E}_0(x)}}, \qquad
  (m_0(x,0))_{21} =  \frac{(1 - \mathcal{E}_0(x))(1 + \overline{\mathcal{E}_0(x)})}{\mathcal{E}_0(x) + \overline{\mathcal{E}_0(x)}}.
\end{align*}
Solving these two equations for $\mathcal{E}_0$ and $\bar{\mathcal{E}}_0$, we find 
\begin{align}\label{E0recover}
\mathcal{E}_0(x) = \frac{1 + (m_0(x,0))_{11} - (m_0(x,0))_{21}}{1 + (m_0(x,0))_{11} + (m_0(x,0))_{21}}.
\end{align}
But by uniqueness of the solution of the RH problem (\ref{RHm}), we have $m_0(x,z) = m(x,0,z)$; hence, comparing (\ref{E0recover}) with (\ref{Erecover}), we deduce that $\mathcal{E}(x,0) = \mathcal{E}_0(x)$ for $x \in [0, 1-\delta)$.
\end{proof}

It only remains to show that $\mathcal{E}(x,y)$ satisfies the hyperbolic Ernst equation (\ref{ernst}) in $\Int(D_\delta)$. The proof of this relies on the construction of an eigenfunction $\Phi$ of the Lax pair.
Equations (\ref{Phiatinftyplus}) and (\ref{mVPhi}) suggest that we define $\Phi(x,y,P)$ for $(x,y) \in D_\delta$ and $P \in F_{(x,y)}^{-1}(\Omega_\infty) \subset \mathcal{S}_{(x,y)}$ by
\begin{align}\label{phidef}
\Phi(x,y,P) = \tilde{\Phi}(x,y)m(x,y,F_{(x,y)}(P)),
\end{align}
where $\tilde{\Phi}(x,y)$ is the function defined in (\ref{tildePhidef}).

\begin{lemma}\label{claim7E}
The function $\Phi$  defined in (\ref{phidef}) satisfies the Lax pair equations
  \begin{align}\label{philax}
\begin{cases}
\Phi_x(x,y, P) = \mathsf{U}(x,y, P) \Phi(x,y, P),
	\\
\Phi_y(x,y, P) = \mathsf{V}(x,y,P) \Phi(x,y,P),
\end{cases}
\end{align}
for $(x,y) \in \Int D_\delta$ and $P \in F_{(x,y)}^{-1}(\Omega_\infty)$.
\end{lemma}
\begin{proof}
The analyticity structure of $n_0$ established in Lemma \ref{claim5E} implies that there exists a $2\times 2$-matrix valued function  $C(x,y)$ independent of  $z$ such that
\begin{align}\label{n0C}
n_0(x,y,z) = \frac{C(x,y)}{z+1}, \qquad z \in \hat{\C}.
\end{align}
We determine $C(x,y)$ by evaluating (\ref{n0C}) at $z = 0$. By Lemma \ref{claim5E}, this gives $C(x,y) = \hat{m}_x(x,y)\hat{m}(x,y)^{-1}$.
It follows that
\begin{align}\label{noutside}
  n_0 = \frac{\hat{m}_x(x,y)\hat{m}(x,y)^{-1}}{z+1} = \bigg(m_x -\frac{1-z}{1+z} \frac{z}{1 - x - y} m_z\bigg)m^{-1}
\end{align}
for $(x,y) \in D_\delta$ and $z \in \Omega_\infty$. %Note that the case of $n_0$ being free of singularities is included again. Then $C=0=m_x$.

Differentiating (\ref{phidef}) with respect to $x$ and using (\ref{noutside}), we find, for $P \in F_{(x,y)}^{-1}(\Omega_\infty)$,
\begin{align*}
\Phi_x(x,y,P)
& = \tilde{\Phi}_x(x,y)m(x,y,F_{(x,y)}(z))
+ \tilde{\Phi}(x,y)(m_x + z_x m_z)
	\\
& = \tilde{\Phi}_x(x,y)m(x,y,F_{(x,y)}(z))
+ \tilde{\Phi}(x,y) \frac{\hat{m}_x(x,y)\hat{m}(x,y)^{-1}}{z+1}m(x,y,z) 
	\\
& = \bigg(\tilde{\Phi}_x(x,y)\tilde{\Phi}(x,y)^{-1} 
+ \tilde{\Phi}(x,y) \frac{\hat{m}_x(x,y)\hat{m}(x,y)^{-1}}{z(x,y,P)+1}\tilde{\Phi}(x,y)^{-1}\bigg)\Phi(x,y,P) 
\end{align*}
Substituting in the expressions (\ref{tildePhidef}) and (\ref{mhatexpression}) for $\tilde{\Phi}$ and $\hat{m}$ in terms of $\mathcal{E}$, $\bar{\mathcal{E}}$, and recalling that 
$$1 - \frac{2}{z+1} = \lambda,$$
this yields the first equation in (\ref{philax}). A similar argument gives the second equation in (\ref{philax}). 
\end{proof}

\begin{lemma}\label{claim8E}
The complex-valued function $\mathcal{E}:D \to \R$ defined by (\ref{Erecover}) satisfies the hyperbolic Ernst equation (\ref{ernst}) in $\Int(D_\delta)$.
\end{lemma}
\begin{proof}
Fix a point $P = (\lambda, k)$ in $F_{(x,y)}^{-1}(\Omega_\infty) \subset \mathcal{S}_{(x,y)}$. By Lemma \ref{claim6E}, the map $(x,y) \mapsto \Phi(x,y, P)$ is $C^n$ from $\Int D_\delta$ to $\C$ and satisfies the Lax pair equations (\ref{philax}). Since  $n \geq 2$, it follows that $\Phi$ satisfies 
\begin{align*}
\Phi_{xy}(x,y,P) - \Phi_{yx} (x,y,P) = 0, \qquad (x,y) \in \Int D_\delta.
\end{align*}
The $(21)$-entry of this equation reads
$$\frac{(1-x-y)\lambda}{2(\re \mathcal{E}(x,y))^2 (1-k-y)}\bigg\{(\re \mathcal{E})\bigg(\mathcal{E}_{xy} - \frac{\mathcal{E}_x + \mathcal{E}_y}{2(1-x-y)}\bigg) - \mathcal{E}_x \mathcal{E}_y\bigg\} = 0.$$
It follows that $\mathcal{E}(x,y)$ satisfies (\ref{ernst}) for $(x,y) \in \Int D_\delta$. 
 This completes the proof of the lemma.
\end{proof}

Lemma \ref{claim8E} completes the proof of part $(a)$ of Theorem \ref{mainth3}.

The following lemma proves part (b).

\begin{lemma}\label{lemma9E}
	There exists a constant $c_\delta > 0$ such that if
	\begin{align}
	\label{BDsmall}
	\| \mathcal{E}_0 / \re \mathcal{E}_0 \|_{L^1([0,1-\delta))} , \, \| \mathcal{E}_1/ \re \mathcal{E}_1 \|_{L^1([0,1-\delta))}    < c_\delta,
	\end{align}
	then the linear operator $I - \mathcal{C}_{w(x,y,\cdot)} \in \mathcal{B}(L^2(\Gamma))$ is bijective for each  $(x,y) \in D_\delta$.
\end{lemma}
\begin{proof}
It follows from \eqref{phi0jestimate} and \eqref{phi0assum} that, by choosing $c_\delta$ sufficiently small, equation \eqref{BDsmall} gives
\[
|\Phi_0(x,k^\pm)-I| < \|\mathcal{C}_-\|_{\mathcal{B}(L^2(\Gamma))}^{-1}
\]
and an analogous estimate holds for $|\Phi_1(y,k^\pm)-I|$. This yields
\begin{align}\label{wLinftyCminus}
\|w(x,y,\cdot)\|_{L^\infty(\Gamma)} <  \|\mathcal{C}_-\|_{\mathcal{B}(L^2(\Gamma))}^{-1}
\end{align}
for all $(x,y) \in D_\delta$ whenever (\ref{BDsmall}) holds. Indeed, equation (\ref{wLinftyCminus}) implies
$$\|\mathcal{C}_w\|_{\mathcal{B}(L^2(\Gamma))} \leq \|\mathcal{C}_-\|_{\mathcal{B}(L^2(\Gamma))} \|w\|_{L^\infty(\Gamma)} < 1$$
for all $(x,y) \in D_\delta$.
Hence $I - \mathcal{C}_{w(x,y,\cdot)}$ is invertible in $\mathcal{B}(L^2(\Gamma))$ for each  $(x,y) \in D_\delta$. 
\end{proof}

For part (c) assume $\mathcal{E}_0, \mathcal{E}_1>0$ and write $V_0= -\log \mathcal{E}_0$, $V_1 = -\log \mathcal{E}_1$. Then there exists a $C^n$-solution $V(x,y)$ of the Goursat problem for the Euler-Darboux equation \eqref{linearernst} with data $\{ V_0,V_1 \}$ by Theorem \ref{linearmainth}. Hence $\mathcal{E}=e^{- V}$ is a $C^n$-solution of the Goursat problem for \eqref{ernst} with data $\{ \mathcal{E}_0,\mathcal{E}_1 \}$. This completes the proof of part (c) and hence of Theorem \ref{mainth3}.

\subsection{Proof of Theorem \ref{mainth4}}
Let $\mathcal{E}_0(x)$, $x \in [0, 1)$, and $\mathcal{E}_1(y)$, $y \in [0,1)$, be complex-valued functions satisfying (\ref{E0E1assumptions}) for some $n \geq 2$ and some $\alpha \in (0,1)$. Suppose $\mathcal{E}(x,y)$ is a $C^n$-solution of the Goursat problem for (\ref{ernst}) in $D$ with data $\{\mathcal{E}_0, \mathcal{E}_1\}$ and define $m_1, m_2 \in \C$ by (\ref{m1m2def}).
We will prove (\ref{boundarylimita}); the proof of (\ref{boundarylimitb}) is similar. 

By \eqref{Erecover}, we have
\begin{align}\label{calEx}
x^\alpha \mathcal{E}_x(x,y) = 2x^\alpha \frac{\hat{m}_{21}(x,y) \hat{m}_{11x}(x,y) - (1 + \hat{m}_{11}(x,y))\hat{m}_{21x}(x,y)}{(1 + \hat{m}_{11}(x,y) + \hat{m}_{21}(x,y))^2},
\end{align}
where, as before, $\hat{m}(x,y) = m(x,y,0)$. Thus, in order to compute $\lim_{x \downarrow 0} x^\alpha \mathcal{E}_x(x,y)$, it is enough to compute $\hat{m}(0,y)$ and $\lim_{x \downarrow 0} x^\alpha \hat{m}_x(x,y)$.
Since $m = I + \mathcal{C}(\mu w)$ and 
\begin{align}\label{mxmux}
m_x =  \mathcal{C}(\mu_x w) + \mathcal{C}(\mu w_x), 
\end{align}
this means that we are interested in the values of
\[
w(0,y,z), \quad \mu(0,y,z),\quad  \lim_{x \downarrow 0} x^\alpha w_x(x,y,z), \quad \lim_{x \downarrow 0} x^\alpha \mu_x(x,y,z).
\]

\begin{lemma}\label{lemmaformulasx=0}
We have
\begin{align}\label{w(0)}
& w(0,y, z) = \begin{cases}
0,  & z \in \Gamma_0, 
\\
\Phi_1\big(y, F_{(0,y)}^{-1}(z)\big) - I, \quad & z \in \Gamma_1,
\end{cases} \quad y \in [0,1),
	\\ \label{mu(0)}
& \mu(0,y,z) = \begin{cases}
\Phi_1\big(y,\infty^+\big)^{-1}  \Phi_1\big(y,F_{(0,y)}^{-1}(z)\big), \quad & z \in \Gamma_0, 
\\
\Phi_1\big(y,\infty^+\big)^{-1} , & z \in \Gamma_1,
\end{cases}\quad y \in [0,1),
\end{align}
and
\begin{align}\label{hatm(0)}
 \hat{m}(0,y) = \Phi_1\big(y,\infty^+\big)^{-1}\sigma_3 \Phi_1\big(y,\infty^+\big) \sigma_3, \qquad y \in [0,1).
\end{align}
\end{lemma}
\begin{proof}
Equation \eqref{w(0)} is immediate from (\ref{jumpdef}). Moreover, by \eqref{mVPhi},
\begin{align} \label{formulam0y}
m(0,y,z) =  \Phi_1\big(y,\infty^+\big)^{-1} \times
\begin{cases} I,  & z \in \Omega_1, \\
\Phi_1\big(y,F_{(0,y)}^{-1}(z)\big), \quad & z \in \Omega_0 \cup \Omega_\infty.
\end{cases}
\end{align}
Equation \eqref{mu(0)} follows from (\ref{formulam0y}) and the fact that $\mu(x,y,z)=m_-(x,y,z)$ for $(x,y) \in D$ and $z \in \Gamma$. Since $0 \in \Omega_0$ and $F_{(0,y)}^{-1}(0) = \infty^-$, equation (\ref{hatm(0)}) follows by setting $z = 0$ in (\ref{formulam0y}) and using the first symmetry in (\ref{phi1symmetries}).
\end{proof}

\begin{lemma}\label{lemmaxalphawxxalphamux}
For $y \in [0,1)$, we have
\begin{align}\label{xalphawx}
\lim_{x \to 0} x^\alpha w_x(x,y, z) = \begin{cases}
\frac{1}{2} 
\begin{pmatrix} \bar{m}_1 & \bar{m}_1 \lambda(0,0, F_{(0,y)}^{-1}(z)) \\
m_1 \lambda(0,0, F_{(0,y)}^{-1}(z))  & m_1 \end{pmatrix}, \quad & z \in \Gamma_0, 
\\
0, & z \in \Gamma_1,
\end{cases}
\end{align}
and
\begin{align}\label{xalphamux}
\lim_{x \downarrow 0} x^\alpha \mu_x(x,y,z)
= \Pi(y,z), \qquad   z \in \Gamma_1,
\end{align}
where the function $\Pi(y,z)$  is defined by
$$\Pi(y,z) = - \frac{1}{\sqrt{1-y}}\frac{\Phi_1\big(y, \infty^+\big)^{-1}}{z+1}
\Phi_1\big(y, 0\big)\begin{pmatrix} 0 & \bar{m}_1  \\
m_1& 0  \end{pmatrix}\Phi_1\big(y, 0\big)^{-1}.$$
\end{lemma}
\begin{proof}
It follows from (\ref{jumpdef}) and (\ref{Finversexyderivatives}) that $\lim_{x \to 0} x^\alpha w_x(x,y, z) =0$ for $z \in \Gamma_1$ and that, for $z\in \Gamma_0$,
\begin{align*}
\lim_{x \to 0} x^\alpha w_x(x,y,z)& = \lim_{x \to 0} x^\alpha \bigg\{\Phi_{0x}(x,F_{(x,y)}^{-1}(z))
+ \Phi_{0k}(x,F_{(x,y)}^{-1}(z))\frac{d}{dx} F_{(x,y)}^{-1}(z)\bigg\}
	\\ 
& =  \lim_{x \to 0} x^\alpha \Phi_{0x}(x,F_{(x,y)}^{-1}(z))
=  \lim_{x \to 0} x^\alpha \U_0(x,F_{(x,y)}^{-1}(z)).
\end{align*}
Recalling the definition (\ref{U0def}) of $\mathsf{U}_0$, \eqref{xalphawx} follows.

To prove (\ref{xalphamux}), we note that differentiation of the relation $\mu = I + \mathcal{C}_w \mu$ gives
\begin{align} \label{mux}
\mu_x = (I-\mathcal{C}_w)^{-1}\mathcal{C}_-(\mu w_x).
\end{align}
We first compute $\lim_{x \downarrow 0} \mathcal{C}_-(\mu x^\alpha w_x)$. Equations \eqref{mu(0)} and \eqref{xalphawx} imply, for $z \in \Gamma_1$,
\begin{align}\nonumber
&\Big\{\mathcal{C}_- \big[\lim_{x \to 0} x^\alpha  \mu(x,y,\cdot) w_x(x,y,\cdot)\big]\Big\}(z)
= \frac{\Phi_1\big(y,\infty^+\big)^{-1} }{2\pi i} 
	\\\nonumber
& \times \int_{\Gamma_0} \frac{\Phi_1\big(y,F_{(0,y)}^{-1}(z')\big) 
	\frac{1}{2} \Big(\begin{smallmatrix}\bar{m}_1 & \bar{m}_1 \lambda(0,0, F_{(0,y)}^{-1}(z')) \\
	m_1 \lambda(0,0, F_{(0,y)}^{-1}(z'))  & m_1 \end{smallmatrix}\Big) dz'}{z' -z}
\\ \label{Cmuxalphawx}
& = 
- \Phi_1\big(y,\infty^+\big)^{-1} \underset{z' = -1}{\res} 
\frac{\Phi_1(y,F_{(0,y)}^{-1}(z'))  \lambda(0,0, F_{(0,y)}^{-1}(z'))\left(\begin{smallmatrix} 0 & \bar{m}_1 \\
	m_1  & 0 \end{smallmatrix}\right)}{2(z' -z)} =: \tilde{\Pi}(y,z).
\end{align}
Recalling the expression (\ref{lambda00F0yinv}) for $\lambda(0,0, F_{(0,y)}^{-1}(z))$ and using that $F_{(0,y)}^{-1}(-1) = 0$, we find
\begin{align*}
\tilde{\Pi}(y,z) = -\frac{\Phi_1\big(y,\infty^+\big)^{-1} \Phi_1(y,0)  \begin{pmatrix} 0 & \bar{m}_1 \\
	m_1  & 0 \end{pmatrix} }{(z+1)\sqrt{1-y}}. 
\end{align*}
In view of (\ref{mux}), it only remains to show that $(I-\mathcal{C}_w)\Pi = \tilde{\Pi}$.
We have, for $z \in \Gamma_1$,
\begin{align*}
& (\mathcal{C}_{w(0,y,\cdot)} \Pi)(z)
= \frac{1}{2\pi i} \int_{\Gamma_1}
\frac{\Pi(y, z')\big(\Phi_1\big(y, F_{(0,y)}^{-1}(z')\big) - I\big)}{z' - z_-} dz'
	\\
& = -\frac{\Phi_1\big(y, \infty^+\big)^{-1}\Phi_1\big(y, 0\big)\begin{pmatrix} 0 & \bar{m}_1  \\
	m_1& 0  \end{pmatrix}\Phi_1\big(y, 0\big)^{-1}}{2\pi i\sqrt{1-y}} 
 \int_{\Gamma_1}
 \frac{\Phi_1\big(y, F_{(0,y)}^{-1}(z')\big) - I}{z' - z_-} \frac{dz'}{z'+1}.
\end{align*}
Deforming the contour to infinity and using that
$$\underset{z' = -1}{\res}\frac{\Phi_1\big(y, F_{(0,y)}^{-1}(z')\big) - I}{z' - z} \frac{1}{z'+1} = -\frac{\Phi_1(y, 0) - I}{z+1}, $$
a residue computation gives
\begin{align*}
& (\mathcal{C}_{w(0,y,\cdot)} \Pi)(z)
=
\frac{\Phi_1\big(y, \infty^+\big)^{-1}\Phi_1\big(y, 0\big)\begin{pmatrix} 0 & \bar{m}_1  \\
	m_1& 0  \end{pmatrix}\Phi_1\big(y, 0\big)^{-1}}{\sqrt{1-y}} \frac{\Phi_1(y, 0) - I}{z+1}.
\end{align*}
Simple algebra now shows that $(I-\mathcal{C}_w)\Pi = \tilde{\Pi}$.
\end{proof}

\begin{lemma} \label{lemmaPhi10}
For $y \in [0,1)$, we have
	\begin{align}\label{formulaPhi10}
	\Phi_1(y, 0) = \begin{pmatrix} e^{\int_0^y \frac{\overline{\mathcal{E}_{1y}(y')}}{2\re \mathcal{E}_1(y')} dy'} & 0 \\ 0 & e^{\int_0^y \frac{\mathcal{E}_{1y}(y')}{2\re \mathcal{E}_1(y')} dy'} \end{pmatrix}.
	\end{align}
\end{lemma}
\begin{proof}
Since $0$ is a real branch point of the Riemann surface $\Sigma_{(0,y)}$, the symmetries (\ref{phi1symmetries}) of $\Phi_1$ imply that
$$\Phi_1(y, 0^+) = \Phi_1(y, 0^-) = \sigma_3 \Phi_1(y, 0^+)\sigma_3 \quad \text{and} \quad
\Phi_1(y, 0) = \sigma_1\overline{\Phi_1(y, 0)}\sigma_1.$$
Hence $\Phi_1(y, 0)$ has the form
$$\Phi_1(y, 0) = \begin{pmatrix} f(y) & 0 \\ 0 & \overline{f(y)} \end{pmatrix},$$
where $f(y)$ is a function of $y$. Since $\lambda(0, y, 0) = \infty$, we can determine $f(y)$ by solving the equation
$$\Phi_{1y}(y, 0) = \frac{1}{2 \re \mathcal{E}_{1}(y)}  \begin{pmatrix} \overline{\mathcal{E}_{1y}(y)} & 0 \\ 0 & \mathcal{E}_{1y}(y) \end{pmatrix} \Phi_1(y,0),$$
which is a consequence of (\ref{lax}). This gives the desired statement.
\end{proof}

The following lemma completes the proof of Theorem \ref{mainth4}.

\begin{lemma}\label{lemmaxalphaEx} For $y \in [0,1)$, we have
	\begin{align}\label{formulaxalphaEx}
 \lim_{x \downarrow 0} x^\alpha \mathcal{E}_x(x,y) = m_1 \frac{e^{i\int_0^y \frac{\im \mathcal{E}_{1y}(y')}{\re \mathcal{E}_1(y')} dy'} \re \mathcal{E}_1(y)}{\sqrt{1-y}}.
	\end{align}
\end{lemma}
\begin{proof}
We first compute $ \lim_{x \downarrow 0} x^\alpha m_x(x,y,0)$. Proceeding as in the proof of (\ref{xalphamux}), we find
\begin{align}\label{Cxalphamuwx}
\mathcal{C} \big[\lim_{x \to 0} x^\alpha  \mu(x,y,\cdot) w_x(x,y,\cdot)\big](0)=-\frac{\Phi_1\big(y,\infty^+\big)^{-1} \Phi_1(y,0)}{\sqrt{1-y}}\begin{pmatrix} 0 & \bar{m}_1 \\
	m_1  & 0 \end{pmatrix},
\end{align}
and
\begin{align}\nonumber
\mathcal{C} \big[\lim_{x \to 0} x^\alpha  \mu_x(x,y,\cdot) w(x,y,\cdot)\big](0)
= &\; \frac{\Phi_1\big(y, \infty^+\big)^{-1}
\Phi_1\big(y, 0\big)}{\sqrt{1-y}}\begin{pmatrix} 0 & \bar{m}_1  \\
m_1& 0  \end{pmatrix}\Phi_1\big(y, 0\big)^{-1}
\\ \label{Cxalphamuxw}
& \times 
\left( \Phi_1(y,0) - \sigma_3 \Phi_1(y,\infty^+) \sigma_3 \right),
\end{align}
where the derivation of (\ref{Cxalphamuxw}) employs Lemma \ref{lemmaformulasx=0} and Lemma \ref{lemmaxalphawxxalphamux} as well as the residue calculation
$$-\frac{1}{2\pi i} \int_{\Gamma_1} \frac{\Phi_1(y, F_{(0,y)}^{-1}(z)) - I}{z} \frac{dz}{z+1} = \Phi_1(y,0) - \sigma_3 \Phi_1(y,\infty^+) \sigma_3.$$
Adding (\ref{Cxalphamuwx}) and (\ref{Cxalphamuxw}) and recalling (\ref{mxmux}), we obtain
\begin{align}\nonumber
\lim_{x \to 0} x^\alpha m_x(x,y,0) 
= & - \frac{1}{\sqrt{1-y}}\Phi_1\big(y, \infty^+\big)^{-1}
\Phi_1\big(y, 0\big)\begin{pmatrix} 0 & \bar{m}_1  \\
m_1& 0  \end{pmatrix}
	\\\label{formulaxalphamx}
&\times \Phi_1\big(y, 0\big)^{-1}\sigma_3 \Phi_1(y,\infty^+) \sigma_3.
\end{align}
Substituting \eqref{Phi1atinftyplus}, \eqref{hatm(0)}, \eqref{formulaPhi10}, and \eqref{formulaxalphamx} into (\ref{calEx}), long but straightforward computations yield (\ref{boundarylimita}).
\end{proof}

\section{Examples}\label{examplesec}
We consider two examples of exact solutions---one with collinear polarization and one with noncollinear polarization. For each example, we verify explicitly that the formulas (\ref{boundarylimit}) of Theorem \ref{mainth4} on the behavior near the boundary are satisfied.

\subsection{The Khan-Penrose solution}
The Khan-Penrose \cite{KP1971} solution is given by the potential
		\[
		\mathcal{E}(x,y)= \frac{1+\sqrt{x}\sqrt{1-y}+\sqrt{y}\sqrt{1-x}}{1-\sqrt{x}\sqrt{1-y}-\sqrt{y}\sqrt{1-x}}, \qquad (x,y) \in D.
		\]
		Straightforward computations show that $m_1=1=m_2$ and 
		\begin{align*}
		 \lim_{x \downarrow 0} \sqrt{x} \mathcal{E}_x(x,y) &= \frac{\sqrt{1-y}}{(1-\sqrt{y})^2} = m_1 \frac{e^{i\int_0^y \frac{\im \mathcal{E}_{1y}(y')}{\re \mathcal{E}_1(y')} dy'} \re \mathcal{E}_1(y)}{\sqrt{1-y}} =\frac{\re \mathcal{E}_1(y)}{\sqrt{1-y}},
		 \\
		  \lim_{y \downarrow 0} \sqrt{y} \mathcal{E}_y(x,y) &= \frac{\sqrt{1-x}}{(1-\sqrt{x})^2} = m_2 \frac{e^{i\int_0^x \frac{\im \mathcal{E}_{0x}(x')}{\re \mathcal{E}_1(x')} dx'} \re \mathcal{E}_0(x)}{\sqrt{1-x}} =\frac{\re \mathcal{E}_0(x)}{\sqrt{1-x}}.
		\end{align*}

\subsection{The Nutku-Halil solution}	
One version of the Nutku-Halil \cite{NH1977} solution is given by
		\[
		\mathcal{E}(x,y)= \frac{1-i\sqrt{x}\sqrt{1-y}+i\sqrt{y}\sqrt{1-x}}{1+i\sqrt{x}\sqrt{1-y}-i\sqrt{y}\sqrt{1-x}}, \qquad (x,y)\in D.
		\]
		In this case, $m_1=-i=-m_2$ and we compute
		\begin{align*}
			\lim_{x \downarrow 0} \sqrt{x} \mathcal{E}_x(x,y) &= \frac{i\sqrt{1-y}}{(i+\sqrt{y})^2}  = m_1 \frac{e^{i\int_0^y \frac{\im \mathcal{E}_{1y}(y')}{\re \mathcal{E}_1(y')} dy'} \re \mathcal{E}_1(y)}{\sqrt{1-y}},
			\\
			\lim_{y \downarrow 0} \sqrt{y} \mathcal{E}_y(x,y) &= -\frac{i\sqrt{1-x}}{(i-\sqrt{x})^2}  = m_2 \frac{e^{i\int_0^x \frac{\im \mathcal{E}_{0x}(x')}{\re \mathcal{E}_1(x')} dx'} \re \mathcal{E}_0(x)}{\sqrt{1-x}}.
		\end{align*}

\appendix
\section{Gravitational waves and the hyperbolic Ernst equation} \label{Aapp}
\renewcommand{\theequation}{A.\arabic{equation}}
It is shown in Eq. (11.7) in \cite{G1991} that the Ernst potential $\mathcal{E}$ satisfies
$$2(\re \mathcal{E}) \left(2\mathcal{E}_{uv} - U_u \mathcal{E}_v - U_v \mathcal{E}_u\right) = 4\mathcal{E}_u \mathcal{E}_v.$$
where $e^{-U(u,v)} = f(u) + g(v)$ and $f(u)$ and $g(v)$ are monotonically decreasing for positive argument and $f(0) = g(0) = 1/2$. 
(Note that Griffiths writes $Z$ for the Ernst potential.)
%The metric has a singularity (a curvature or at least coordinate singularity) as $f + g \to 0$. 
As suggested by Szekeres \cite{S1972}, it is possible to use $(f, g)$ as coordinates. 
This leads to the equation
\begin{align}\label{ernstfg}
2(\re \mathcal{E}) \left(2\mathcal{E}_{fg} + \frac{\mathcal{E}_f + \mathcal{E}_g}{f+g}\right) = 4\mathcal{E}_f \mathcal{E}_g,
\end{align}
where $(f,g)$ belongs to the triangular region 
$$\bigg\{(f,g) \in \R^2 \, \bigg| \, f \leq \frac{1}{2}, \; g \leq \frac{1}{2}, \; f + g > 0\bigg\}.$$
The change of variables $x = \frac{1}{2} - g$, $y = \frac{1}{2} - f$ transforms (\ref{ernstfg}) into (\ref{ernst}).

In order for the solution to describe gravitational waves, the following boundary condition must be satisfied (Eq. (7.15) in \cite{G1991}; see also (11.23) in \cite{G1991} but in (11.23) equation $(f,g)$ approaches the corner whereas in (7.15) the two edges are approached; also in (7.15) there is a factor $(f+g)$ missing; this factor comes from (7.9))
\begin{align*}
& \lim_{g \to \frac{1}{2}} \bigg[\Big(\frac{1}{2} -g\Big) (f+g)\frac{|\mathcal{E}_g|^2}{(\mathcal{E} + \bar{\mathcal{E}})^2}\bigg] = \frac{k_2}{2},
	\\
& \lim_{f \to \frac{1}{2}} \bigg[\Big(\frac{1}{2} -f\Big) (f+g) \frac{|\mathcal{E}_f|^2}{(\mathcal{E} + \bar{\mathcal{E}})^2}\bigg] = \frac{k_1}{2},
\end{align*}
for some constants $k_1, k_2 \in [\frac{1}{2}, 1)$. In terms of $(x,y)$,  these conditions become
\begin{align*}
& \lim_{x \to 0} \frac{x(1-x-y)|\mathcal{E}_x|^2}{(\mathcal{E} + \bar{\mathcal{E}})^2}= \frac{k_2}{2},
	\\
& \lim_{y \to 0} \frac{y(1-x-y)|\mathcal{E}_y|^2}{(\mathcal{E} + \bar{\mathcal{E}})^2} = \frac{k_1}{2},
\end{align*}
for some constants $k_1, k_2 \in [\frac{1}{2}, 1)$. That is, since $\re \mathcal{E} > 0$,
\begin{align*}
& \lim_{x \to 0} \sqrt{x} \sqrt{1-x-y}\frac{|\mathcal{E}_x(x,y)|}{\re \mathcal{E}(x,y)} = \sqrt{2k_2} = m_1, \qquad y \in [0,1),
	\\
& \lim_{y \to 0} \sqrt{y} \sqrt{1-x-y}\frac{|\mathcal{E}_y(x,y)|}{\re \mathcal{E}(x,y)} = \sqrt{2k_1} = m_2,\qquad x \in [0,1),
\end{align*}
for some constants $m_1 , m_2 \in [1, \sqrt{2})$. 
If we assume that $\mathcal{E} \in C(D)$, these conditions become
\begin{align*}
& \lim_{x \to 0} \sqrt{x} \sqrt{1-y}\frac{|\mathcal{E}_x(x,y)|}{\re \mathcal{E}_1(y)} = \sqrt{2k_2} = m_1, \qquad y \in [0,1),
	\\
& \lim_{y \to 0} \sqrt{y} \sqrt{1-x}\frac{|\mathcal{E}_y(x,y)|}{\re \mathcal{E}_0(x)} = \sqrt{2k_1} = m_2,\qquad x \in [0,1).
\end{align*}
These are the conditions given in (\ref{boundaryconditions}) with $\alpha = 1/2$.
In particular,
\begin{align*}
& \mathcal{E}_{0x}(x) = \frac{m_1 + o(1)}{\sqrt{x}} \quad \text{i.e.} \quad \mathcal{E}_{0}(x) \sim 2m_1\sqrt{x}, \qquad  x \downarrow 0,
	\\
& \mathcal{E}_{1y}(y) = \frac{m_2 + o(1)}{\sqrt{y}} \quad \text{i.e.} \quad \mathcal{E}_{1}(y) \sim 2m_2\sqrt{y}, \qquad  y \downarrow 0,
\end{align*}
where $m_1, m_2 \in [1, \sqrt{2})$.

\bigskip
\noindent
{\bf Acknowledgement} {\it The authors acknowledge support from the European Research Council, Grant Agreement No. 682537, the Swedish Research Council, Grant No. 2015-05430, and the G\"oran Gustafsson Foundation, Sweden.}

\bibliographystyle{plain}
%\bibliography{is}

\end{document}